\documentclass[a4paper, oneside, 10pt]{amsart}
\usepackage{graphicx}
\usepackage{amsmath}
\usepackage{amssymb}
\usepackage{geometry}
\usepackage{tikz}
\usepackage{amsthm}
\usepackage{amsfonts}
\usepackage{mathrsfs}
\usepackage{ upgreek }

\usepackage[backref]{hyperref}  

\usepackage{comment}
\newcommand{\Hit}{\operatorname{Hit}}
\newcommand{\PSL}{\operatorname{PSL}}
\newcommand{\PSO}{\operatorname{PSO}}
\newcommand{\bR}{\mathbb{R}}

\newcommand{\sign}{\operatorname{sign}}
\newcommand{\Ad}{\operatorname{Ad}}
\newcommand{\Hom}{\operatorname{Hom}}

\newcommand{\kf}{\mathbf{\mathsf{B}}}
\newcommand{\HD}{\mathcal{H}^{\mathrm{tw}}}
\newcommand{\WT}{\mathcal{WT}^{\mathrm{tw}}}

\newcommand{\RP}{\mathbb{PR}}
\newcommand{\dr}{\mathfrak{r}}
\newtheorem{lemma}{Lemma}[section]
\newtheorem{theorem}[lemma]{Theorem}
\newtheorem{corollary}[lemma]{Corollary}

\theoremstyle{definition}
\newtheorem{definition}[lemma]{Definition}
\newtheorem{question}[lemma]{Question}
\newtheorem{remark}[lemma]{Remark}
\newtheorem{setup}[lemma]{Set-up}
\newtheorem{example}[lemma]{Example}
\newtheorem{claim}[lemma]{Claim}

\newcommand{\ui}{\mathrm{i}}
\newcommand{\ML}{\mathcal{ML}}
\newcommand{\der}{\mathrm{d}}
\newcommand{\PT}{\mathbb{P}\mathrm{T}}
\newcommand{\UT}{\mathrm{T}^1}
\newcommand{\dG}{\mathrm{d}_\mathsf{G}}
\newcommand{\dUT}{\mathrm{d}_{\mathrm{T}^1}}
\newcommand{\dPT}{\mathrm{d}_{\mathbb{P}\mathrm{T}}}

\newcommand{\dirac}{\updelta}
\newcommand{\e}{\mathrm{e}}
\newcommand{\supp}{\mathrm{supp}}
\title[Convergence of cataclysm deformations]{Convergence of cataclysm deformations on Anosov representations and applications}
\author{Hongtaek Jung}
\address{School of Mathematics\\ KIAS \\ Seoul} 
\begin{document}

\begin{abstract}
A cataclysm deformation, that shears and twists a given Anosov representation according to data known as a twisted transverse cocycle, is an intuitive and powerful tool for studying Anosov representations. We show that if a sequence of twisted measured laminations converges weakly, the sequence of corresponding cataclysm deformations on the space of Anosov representations converges uniformly on compact sets. 

This result leads to two applications. First, we obtain an extension of the Goldman product formula. Second, we consider strongly dense representations, introduced by Breuillard--Green--Guralnick--Tao and Long--Reid. Using cataclysm deformations, we show that, for a split real form $\mathsf{G}$ whose Weyl group contains $-1$, the set of strongly dense $\mathsf{G}$-Hitchin representations is not open in the $\mathsf{G}$-Hitchin component. 
\end{abstract}

\maketitle

\section{Introduction}

Let $S$ be a closed orientable surface of genus at least two equipped with a hyperbolic structure $X_0$. One intuitive way to obtain a new hyperbolic structure from $X_0$ is to perform a twist along a simple closed geodesic. If we do this continuously, we obtain a 1-parameter family of hyperbolic structures inducing a flow on the Teichm\"uller space of $S$. Such flows are not just elementary but fundamental as well. Wolpert \cite{wolpert}, for example, showed that they are Hamiltonian flows with respect to the Weil--Petersson symplectic form. 

Now, instead of twisting $X_0$ along a single simple closed geodesic, we apply left-twists along infinitely many disjoint simple geodesics. This procedure is much more delicate than twisting along a single geodesic; we have to control the amount of twisting along each geodesic carefully so that the total twisting across infinitely many geodesics does not diverge. This infinite twisting deformation is called an earthquake and its twisting information is packaged into an object known as a measured lamination. Just as twisting deformations induce flows on the Teichm\"uller space, earthquake deformations also give rise to flows on the Teichm\"uller space. 

Thurston first considered earthquakes and showed that any two points in the Teichm\"uller space can be joined by a flow line of an earthquake flow. Kerckhoff \cite{Kerckhoff,kerckhoff_analytic} further studied earthquake flows and showed that the hyperbolic length function is convex along an earthquake flow. This result enabled him to resolve the Nielsen realization problem. 

A cataclysm is a further generalization of an earthquake. Unlike an earthquake, a cataclysm allows both left- and right-twisting along leaves of a geodesic lamination. Now the twisting information involves negative values that represent right-twists. We will record such left- and right-twisting data as a so-called $\bR$-valued transverse cocycle.  

We can view the Teichm\"uller space as the space of Fuchsian  representations of $\pi_1(S)$ into $\PSL_2(\bR)$. For a general semisimple noncompact Lie group $\mathsf{G}$, $\Theta$-Anosov representations are representations of $\pi_1(S)$ into $\mathsf{G}$ that have dynamical and geometric properties analogous to those of Fuchsian representations.  $\mathsf{G}$-Hitchin representations for the split real form $\mathsf{G}$ of a simple complex adjoint group are the most well-known examples of Anosov representations. It is known that the space of $\Theta$-Anosov representations is open and that each $\Theta$-Anosov representation admits a well-behaved limit map from the Gromov boundary $\partial_\infty\pi_1(S)$ into a flag manifold $\mathcal{F}_\Theta$.  It turns out that these two properties are essentially sufficient to generalize cataclysm deformations to the space of $\Theta$-Anosov representations. 

Indeed, Bonahon \cite{bonahon96} considered cataclysm deformations on the space of representations of $\pi_1(S)$ into $\PSL_2(\mathbb{C})$. He introduced $\bR\times S^1$-valued transverse cocycles that encode twisting, and bending data of a quasi-Fuchsian representation and used them to give a local parametrization of the space of representations of a surface group into $\PSL_2(\mathbb{C})$. Also, McMullen \cite{mcmullen} studied complex earthquake on the space of quasi-Fuchsian representations and derived several interesting results. This theory has been successfully adapted to higher-rank Lie groups, for example, by Pfeil \cite{pfeil}, and Bonahon--Dreyer \cite{BD}. 

In this paper, we consider cataclysm deformations on the space of $\Theta$-Anosov representations into real and complex semisimple Lie groups $\mathsf{G}$. More precisely, we will deal with cataclysm deformations associated to a special class of twisted cocycles, namely twisted measured laminations with values in the Cartan subalgebra of $\mathsf{G}$.  For quasi-Fuchsian representations, it is known that the map $\mathsf{c} \mapsto \Lambda^{\mathsf{c}}$, which assigns to each complex valued measured lamination $\mathsf{c}$ the cataclysm deformation $\Lambda^\mathsf{c}$ is continuous \cite{EM06,kourouniotis}. Our main theorem generalizes this result to Anosov representations. 

\begin{theorem}\label{thm:continuity}
    Let $\mathsf{G}$ be a semisimple complex Lie group with the Cartan subalgebra $\mathfrak{h}$.  Let $\{\mathsf{c}_n\}$ be a sequence of twisted $\mathfrak{h}$-valued measured laminations that converges weakly to a twisted $\mathfrak{h}$-valued measured lamination $\mathsf{c}$. Then the sequence of cataclysm deformations $\{\Lambda^{\mathsf{c}_n}\}$ on the space of Borel Anosov representations converges to $\Lambda^\mathsf{c}$ uniformly on compact sets.
\end{theorem}

A more precise statement can be found in Theorem~\ref{thm:convergence}. Indeed, our actual main theorem deals with general $\Theta$-Anosov representations.  

Now, let us focus on  $\PSL_n(\bR)$-Hitchin representations. On the space of Hitchin representations, Goldman flows are elementary examples of cataclysm deformations, in the sense that they are the cataclysms associated with transverse cocycles supported on simple closed curves. More precisely, let  $\mathbf{a}$  be a linear functional on the vector space of the traceless real diagonal matrices. We can realize $\mathbf{a}$ as the dual of some trace-less diagonal matrix $\mathbf{a}^\vee$, that is, $\mathbf{a}(Y) = \mathrm{tr}(\mathbf{a}^\vee Y)$ for all diagonal matrices $Y$. We  assume that $\mathbf{a}^\vee=\mathrm{diag}(x_1, \cdots, x_n)$ satisfies  $x_i = -x_{n+1-i}$. Then the Goldman flow along a simple closed geodesic $\gamma$ is the cataclysm $\Lambda^{ \mathbf{a}^\vee \gamma}$, where $ \mathbf{a}^\vee \gamma$ is a transverse cocycle such that $(\mathbf{a}^\vee \gamma)(k)=n\mathbf{a}^\vee$ if a geodesic arc $k$ intersects $\gamma$ $n$-times. 

It is known that Goldman flows are Hamiltonian flows with respect to the Atiyah--Bott--Goldman symplectic form $\omega_\mathrm{ABG}$. The Hamiltonian function associated with the Goldman flow $\Lambda^{\mathbf{a}^\vee\gamma}$ is given by the length function $\ell^\mathbf{a}_\gamma:\rho\mapsto \mathbf{a}( \lambda(\rho(\gamma)))$, where $\lambda$ is the Jordan projection.  In other words, the vector field $\mathcal{H}^\mathbf{a}_\gamma = \frac{\der}{\der t}\Lambda^{t\mathbf{a}^\vee \gamma}$ satisfies 
\[
\der \ell^\mathbf{a}_\gamma = \omega_\mathrm{ABG} (\mathcal{H}_\gamma ^\mathbf{a}, -).
\]
The Poisson bracket of length functions in this particular case can be computed using the Goldman product formula \cite{Goldman86}. We paraphrase it as 
\[
\{\ell^\mathbf{a} _\gamma, \ell^\mathbf{b} _\eta\} (\rho)=\omega_\mathrm{ABG}(\mathcal{H}^\mathbf{a} _\gamma, \mathcal{H}^\mathbf{b} _\eta)|_\rho = \sum_{p\in \gamma\sharp \eta}  \mathrm{tr}^{\mathbf{a},\mathbf{b}}_\rho(\gamma_p,\eta_p),
\]
where $\gamma_p$ and $\eta_p$ are lifts of $\gamma$ and $\eta$ whose intersection descends to the intersection point $p$ and $\mathrm{tr}^{\mathbf{a},\mathbf{b}}_\rho(\gamma_p,\eta_p)$ is some function on the space $\mathcal{G}^\times$ of pairs of intersecting geodesics in $\mathbb{H}^2$. The definition of $\mathrm{tr}^{\mathbf{a},\mathbf{b}}_\rho$ can be found in Section~\ref{sec:app}.  The above formula can be written more neatly using geodesic currents:
\[
\{\ell^\mathbf{a} _\gamma, \ell^\mathbf{b} _\eta\} (\rho)= \int_{\mathcal{G}^\times(S)} \mathrm{tr}^{\mathbf{a},\mathbf{b}}_\rho \;\der(\dirac_\gamma \times \dirac_\eta)
\]
where $\dirac_\gamma$ and $\dirac_\eta$ are the Dirac delta measures on the lifts of $\gamma$ and $\eta$ respectively and $\mathcal{G}^\times(S)=\mathcal{G}^\times/\pi_1(S)$. 

Consider a measured lamination $\mu$. Then, the association $k\mapsto \mathbf{a}^\vee\mu(k)$ gives rise to a measured lamination valued in the vector space of trace-less diagonal matrices. As an application of Theorem~\ref{thm:continuity}, we know that the cataclysm flow $\Lambda^{t\mathbf{a}^\vee\mu}$ is also a Hamiltonian flow. If we denote by $\ell^\mathbf{a}_\mu$ a Hamiltonian function of the cataclysm $\Lambda^{t\mathbf{a}^\vee \mu}$, we can further obtain the following generalization of the Goldman product formula:

\begin{theorem}\label{thm:poisson} Let $S$ be a closed oriented surface of genus at least two. Let $\mathsf{G}$ be the split real form of a complex simple adjoint group with the invariant bilinear form $\kf$. Let $\mathbf{a},\mathbf{b}$ be non-trivial elements of $\mathfrak{a}^*$ that are fixed by the opposite involution. Then, for measured laminations $\mu$ and $\tau$ on $S$, and for $\rho\in \Hit_{\mathsf{G}}(S)$, we have  
\[
\{\ell^\mathbf{a}_\mu, \ell^\mathbf{b}_{\tau}\}(\rho)=\int_{\mathcal{G}^\times (S)}\kf^{\mathbf{a},\mathbf{b}}_\rho \;\der(\mu \times \tau).
\]    
\end{theorem}

Bridgeman--Labourie \cite{bridgeman2025} also obtained the same formula on the space of $\Theta$-Anosov representations when  $\mathbf{a}=\mathbf{b}=\sum_{\theta\in \Theta}\theta$. 

The second application of Theorem~\ref{thm:continuity} concerns  strongly dense representations. This notion originated from the work of Breuillard--Green--Guralnick--Tao \cite{tao} where the authors defined a subgroup $\mathsf{H}$ of an algebraic group $\mathsf{G}$ to be \emph{strongly dense} if every subgroup $\langle x,y\rangle$ generated by a non-commuting pair of elements $x,y\in \mathsf{H}$ is Zariski dense in $\mathsf{G}$. Following this definition, we say that a representation $\rho:\Gamma\to \mathsf{G}$ is strongly dense if for every non-commuting pair $x,y\in \Gamma$, the subgroup $\langle \rho(x),\rho(y)\rangle$ generated by $\rho(x)$ and $\rho(y)$ is Zariski dense in $\mathsf{G}$. This is much stronger than a representation being Zariski dense, which only requires the image $\rho(\Gamma)$ to be Zariski dense. 

The existence of strongly dense free groups \cite{tao}, surface groups \cite{long,breuillard}, some hyperbolic 3-manifold groups \cite{lee}, and 2-orbifold groups \cite{jung2025} in $\mathsf{G}$ has been established under certain conditions on $\mathsf{G}$. 

When $\mathsf{G}$ is a split real form, strongly dense surface groups in $\mathsf{G}$ are abundant in the sense that generic $\mathsf{G}$-Hitchin representations are strongly dense. On the other hand,  it is not even known whether  this space is open in the $\mathsf{G}$-Hitchin component. 

As an application of Theorem~\ref{thm:continuity}, we show that, unlike the Zariski dense condition, being strongly dense is not an open condition. 

\begin{theorem}\label{thm:main} Let $\mathsf{G}$ be a split real form of a complex simple Lie group of adjoint type. Assume that the longest element of the Weyl group of $\mathsf{G}$ is $-1$.  Fix a symmetric generating set for $\pi_1(S)$ and let $|\cdot|$ be the associated word metric. There exist a sequence of $\mathsf{G}$-Hitchin representations $\{\rho_n\}$ and constants $A,B>0$ that satisfy the following properties: 
 \begin{enumerate}
      \item each $\rho_n$ is not strongly dense,
     \item  $\langle\rho_n(x),\rho_n(y)\rangle$ is Zariski dense for all non-commuting  $x,y\in \pi_1(S)$ with $|x|,|y|<An-B$,
     \item $\{\rho_n\}$ converges to a strongly dense $\mathsf{G}$-Hitchin representation.
 \end{enumerate}
\end{theorem}

In particular, we can immediately conclude the following:

\begin{corollary}
    Let $\mathsf{G}$ be as in Theorem~\ref{thm:main}. The set of strongly dense $\mathsf{G}$-Hitchin representations is not open.
\end{corollary}

We remark that Theorem~\ref{thm:main} applies to the Lie groups $\mathsf{PSO}(n,n+1)$, $\mathsf{PSp}(2n)$, $\mathsf{PSO}(2n,2n)$, $\mathsf{E}_7$, $\mathsf{E}_8$, $\mathsf{F}_4$, and $\mathsf{G}_2$. Although we believe that Theorem~\ref{thm:main} holds for other Lie groups such as $\mathsf{PSL}(n,\bR)$ and $\mathsf{E}_6$, we could not cover these cases with our current techniques.

Theorem~\ref{thm:main} also shows that it is impossible to decide whether a given representation is strongly dense by testing Zariski density of some finite numbers of subgroups.

Let $\mathcal{SD}$ be the set of strongly dense $\mathsf{G}$-Hitchin representations. Our proof of Theorem~\ref{thm:main} actually shows that there exists a 1-parameter family $l$ of Hitchin representations contained in the topological boundary of $\mathcal{SD}$ such that  $l\cap \mathcal{SD}$ is dense in $l$.

\subsection*{Organization} In Section~\ref{sec:lie}, we quickly review relevant Lie theory, Anosov representations and Hitchin components.

We introduce twisted cocycles and measured laminations in Section~\ref{sec:cocycle} and explore their elementary properties. Most estimates in this section will be used in the proof of Theorem~\ref{thm:continuity}. 

Theorem~\ref{thm:continuity} will be shown in Section~\ref{sec:cataclysm}.  We also present the construction of cataclysm deformations on the space of Anosov representation into a complex Lie group. 

In Section~\ref{sec:app} we prove Theorem~\ref{thm:poisson} and Theorem~\ref{thm:main}. 

\subsection*{Acknowledgement} The author was supported by a KIAS Individual Grant (MG105301) at Korea Institute for Advanced Study.

\section{Hitchin and Anosov representations}
\subsection{Lie theory}\label{sec:lie}
Let $\mathfrak{g}$ be a real semisimple Lie algebra. For such a Lie algebra, the Killing form $\kf$ is nondegenerate and there exists a Cartan involution $\tau:\mathfrak{g}\to \mathfrak{g}$ characterized by the property that $\kf_\tau(\cdot,\cdot):=-\kf(\cdot,\tau(\cdot))$ is positive definite. The Lie algebra $\mathfrak{g}$ can be decomposed into the $\pm 1$-eigenspaces of the Cartan involution $\tau$, giving rise to a Cartan decomposition $\mathfrak{g} = \mathfrak{k}\oplus \mathfrak{p}$. A maximal abelian subalgebra $\mathfrak{a}$ of $\mathfrak{p}$ together with the set $\Sigma$ of restricted roots forms a root system. Additionally, $\mathfrak{g}$ admits the restricted root space decomposition
\[
\mathfrak{g} = \mathfrak{g}_0 \oplus \bigoplus_{\alpha\in \Sigma}\mathfrak{g}_\alpha,
\]
where $\mathfrak{g}_\alpha = \{X\in \mathfrak{g}\mid  [A,X] = \alpha(A)X\text{ for all }A\in \mathfrak{a}\}$ is the restricted root space. We fix an ordering on $\Sigma$ and denote by $\Sigma^+$ the set of positive roots. Let $\Delta\subset \Sigma^+$ be the set of simple roots. Then the \emph{positive Weyl chamber} $\mathfrak{a}^+$ is given by $\{X\in \mathfrak{a}\mid \alpha(X)>0\text{ for all }\alpha\in \Delta\}$. 

The Weyl group of $\mathfrak{g}$ is a finite group generated by root reflections. This group contains the longest element $w_0$ whose action on $\mathfrak{a}$ is an involution. Define the \emph{opposite involution} $\iota:\mathfrak{a}\to \mathfrak{a}$ by $\iota(X) = -w_0\cdot X$.  The opposite involution preserves the positive Weyl chamber $\mathfrak{a}^+$. 

Let $\mathsf{G}$ be a connected real semisimple Lie group with Lie algebra $\mathfrak{g}$.  The Weyl group can be identified with $N_K(\mathfrak{a})/Z_K(\mathfrak{a})$. We sometimes lift $w_0$ to an element $\widehat{w_0}$ of $N_K(\mathfrak{a})$. Then, $\widehat{w_0}$ acts on $\mathsf{G}$ as the conjugation. 

Any element $g\in \mathsf{G}$ can be written as a commuting product $g=g_Kg_A g_U$, where $g_K$, $g_A$, and $g_U$ are elements in $\mathsf{G}$ such that $\Ad_{g_K}$, $\Ad_{g_A}$ and $\Ad_{g_U}$ are elliptic, hyperbolic, and unipotent matrices, respectively.  Moreover, the conjugacy class of $g_A$ intersects $\exp \overline{\mathfrak{a}}^+$ at a single element,  $\exp \lambda(g)$. The map $g\mapsto \lambda(g)$ so-obtained is called the \emph{Jordan projection}. Let 
\[
\mathsf{L}= \bigcup_{g\in \mathsf{G}}g(\exp \mathfrak{a}^+)g^{-1}.
\]
Elements in $\mathsf{L}$ are called \emph{purely loxodromic}.  It is known that $\mathsf{L}$ is an open set and  $\lambda:\mathsf{L}\to \mathfrak{a}^+$ is a real analytic map. We refer to, for instance, \cite[Lemma~A.1]{jung2025}.

Now let $\mathfrak{g}$ be a complex semisimple Lie algebra and let $\mathfrak{h}$ be its Cartan subalgebra. Let $\kf$ be the Killing form of $\mathfrak{g}$. Then,  $\mathfrak{g}$ admits the root space decomposition
\[
\mathfrak{g} = \mathfrak{h} \oplus \bigoplus_{\alpha\in \Sigma}\mathfrak{g}_\alpha
\]
where $\Sigma$ is the set of roots.

We are concerned with two real Lie algebras associated to a complex semisimple Lie algebra $\mathfrak{g}$. One is the Lie algebra $\mathfrak{g}^\bR$ obtained by simply regarding $\mathfrak{g}$ as a real Lie algebra. The other one is the split real form $\mathfrak{g}_{\mathrm{split}}$ of $\mathfrak{g}$. Both real Lie algebras are known to be semisimple. 

First, we consider $\mathfrak{g}^\bR$. Being a semisimple real Lie algebra, $\mathfrak{g}^\bR$ admits the Cartan decomposition $\mathfrak{g}^\bR= \mathfrak{k}^\bR\oplus \mathfrak{p}^\bR$. 
More concretely, $\mathfrak{g}$ admits a compact real form $\mathfrak{u}$ so that  $\mathfrak{k}^\bR$ and $\mathfrak{p}^\bR$ are given by $\mathfrak{u}$ and $\ui \mathfrak{u}=\sqrt{-1} \mathfrak{u}$, respectively. We also have the restricted root space decomposition
\[
\mathfrak{g}^\bR = (\mathfrak{g}^\bR) _0 \oplus \bigoplus_{\alpha\in \Sigma}(\mathfrak{g}^\bR)_\alpha.
\]
The restricted root space decomposition of $\mathfrak{g}^\bR$ and the root space decomposition of $\mathfrak{g}$ are closely related. The maximal abelian subalgebra $\mathfrak{a}^\bR$ of $\mathfrak{p}^\bR$ satisfies $\mathfrak{h}  = \mathfrak{a}^\bR\oplus \ui \mathfrak{a}^\bR=(\mathfrak{g}^\bR) _0$ as real vector spaces. The root system of $\mathfrak{g}$ and the restricted root system of $\mathfrak{g}^\mathbb{R}$ are of the same type; if $\theta$ is a root of $\mathfrak{g}$ then $\theta|\mathfrak{a}^\bR$ is a restricted root of $\mathfrak{g}^\mathbb{R}$. For each $\alpha\in \Sigma\cup\{0\}$, we have $\mathfrak{g}_\alpha=(\mathfrak{g} ^\bR)_\alpha$, as real vector spaces.

The \emph{split real form} of a complex semisimple Lie algebra $\mathfrak{g}$, denoted by $\mathfrak{g}_{\mathrm{split}}$, is a real form such that, for a Cartan decomposition $\mathfrak{g}_\mathrm{split}=\mathfrak{k}\oplus \mathfrak{p}$, a maximal abelian subalgebra $\mathfrak{a}$ of $\mathfrak{p}$ satisfies $\mathfrak{h} =\mathfrak{a}\otimes \mathbb{C}= \mathfrak{a}\oplus \ui\mathfrak{a}$. It follows that $\mathfrak{a}$ is also a maximal abelian subalgebra of $\mathfrak{p}^\bR$. Moreover, the restricted root systems of $\mathfrak{g}_{\mathrm{split}}$, and that of $\mathfrak{g}^\bR$ are of the same type.

Let $\mathsf{G}$ be a complex semisimple connected Lie group with Lie algebra $\mathfrak{g}$. If we regard $\mathsf{G}$ as a real Lie group, then its Lie algebra is given by $\mathfrak{g}^\bR$. For the split real form $\mathfrak{g}_{\mathrm{split}}$, we denote by $\mathsf{G}_{\mathrm{split}}$ the Lie subgroup associated to $\mathfrak{g}_{\mathrm{split}}$.

Now, we define a \emph{complex Jordan projection} $\lambda^\mathbb{C}$ for a complex connected simple Lie group $\mathsf{G}$.  Unlike the usual Jordan projection, we define $\lambda^\mathbb{C}$ only on a small neighborhood of the set of purely loxodromic elements. From now on, we will use $\mathfrak{a}$, instead of $\mathfrak{a}^\bR$, to denote the maximal abelian subalgebra of $\mathfrak{p}^\bR$. Recall that the Cartan subgroup $\mathsf{H}$ of $\mathsf{G}$ is given by $\mathsf{M}\times \mathsf{A}$ where $\mathsf{A}=\exp \mathfrak{a}$ and $\mathsf{M}= \exp (\ui \mathfrak{a})$.  

Let 
\[
\mathsf{H}^+:=\mathsf{M}\times \exp\mathfrak{a}^+.
\]
Observe that each $\mathsf{G}$-conjugacy class intersects $\mathsf{H}^+$ at most once. Let
\[
\mathsf{K} := \bigcup _{h\in \mathsf{G}} h \mathsf{H}^+ h^{-1}
\]
be the set of elements that are conjugate into $\mathsf{H}^+$. 

The exponential map $\exp:\ui \mathfrak{a} \to \mathsf{M}$ is analytic and has a local inverse $\log_M:V\to \ui\mathfrak{a}$ defined an open set $V$ containing the identity such that $\log_M(\mathrm{Id}) = 0$. Define the open set
\[
U=\bigcup_{h\in \mathsf{G}}h(V\times \exp \mathfrak{a}^+)h^{-1}.
\]
Observe that $U$ is open, invariant under conjugation, and contains all purely loxodromic elements. For any element $g\in U$, its $\mathsf{G}$-conjugacy class meets $V\times \exp\mathfrak{a}^+$ at a unique point, say $g'$. Write $g'=g_{M}\cdot g_{A}$ where $g_M\in V$ and $g_A\in \exp \mathfrak{a}^+$. Then, $\lambda^\mathbb{C}(g)$ is defined to be 
\[
\lambda^\mathbb{C}(g)=\log(g_A)+ \log_M(g_M)\in \mathfrak{a}^+ \oplus \ui \mathfrak{a}\subset \mathfrak{h}.
\]
Note that $\lambda^\mathbb{C}$ is well-define only on $U$. Furthermore, $\lambda^\mathbb{C}(g)$ coincides with the usual Jordan projection of $g$ when $g$ is purely loxodromic. We regard $\lambda^\mathbb{C}$ as a $\mathfrak{h}$-valued function defined on the open set $U\subset \mathsf{K}$.   

We claim that $\lambda^\mathbb{C}$ is holomorphic. Indeed, one observes that the map $\mathsf{H}^+\times \mathsf{G}\to \mathsf{K}$ defined by  $(a,g) \mapsto gag^{-1}$ is holomorphic and has a holomorphic local section $\varphi$ near the image of $(a_0,1)$ for each $a_0\in \exp \mathfrak{a}^+$. The logarithm of the $\mathsf{H}^+$-factor of the local section $\varphi$ is, by definition, the complex Cartan projection, showing that it is holomorphic. Moreover, the same argument shows the following.

\begin{lemma}\label{lem:holomorphicconj}
    Let $\mathsf{G}$ be a complex connected simple Lie group. Let $f:D\to \mathsf{G}$ be a holomorphic map defined on a small disk $D\subset \mathbb{C}$ such that $f(0)\in \mathsf{H}^+$ and $f(z)\in \mathsf{K}$ for all $z\in D$. Then there is a holomorphic map $g:D'\to \mathsf{G}$ defined on a smaller open disk $D'\subset D$ with $g(0) = 1$ such that $g(z) f(z) g(z)^{-1}\in \mathsf{H}^+$. 
\end{lemma}

Let $\mathbf{a} :  \mathfrak{h}\to \mathbb{C}$ be a linear functional. As $\lambda^\mathbb{C}$ is holomorphic, the function $f:=\mathbf{a}\circ \lambda^\mathbb{C}$ is also holomorphic on $U$. Moreover, $f$ is invariant under the conjugation action, that is,  $f(ghg^{-1}) =f(h)$. For such a function, we can associate the \emph{Goldman function} $\widehat{f}:U \to \mathfrak{g}$ defined by the property that  $\kf(\widehat{f}(g), X) = \frac{\der}{\der t} f(g \exp tX)$ for all $X\in \mathfrak{g}$. By following the computation in \cite[Lemma~3.4]{jung2025}, we can compute the Goldman function of $f$. For this, let $\mathbf{a}^\vee$ be the dual of $\mathbf{a}$ in the sense that $\mathbf{a}(Y) = \kf(\mathbf{a}^\vee,Y)$ for all $Y\in \mathfrak{h}$.
\begin{lemma}\label{lem:goldmanfunction}
Let $\mathbf{a}\in \mathfrak{h}^*$ and let $f = \mathbf{a}\circ \lambda^\mathbb{C}$. Then 
\[
\widehat{f}(g) = \mathbf{a}^\vee,\qquad \widehat{f}(g^{-1}) = w_0 \mathbf{a}^\vee
\]
for $g\in U \cap \mathsf{H}^+$.
\end{lemma}

\subsection{Anosov and Hitchin representations} An Anosov representation is a higher rank generalization of a convex cocompact representation of $\pi_1(S)$ into $\PSO(3,1)$. When discussing Anosov representations into $\mathsf{G}$, we will treat $\mathsf{G}$ as a real Lie group even if it is a complex Lie group. 

Let $\mathsf{G}$ be a real semisimple noncompact Lie group. Let $\Theta$ be a subset of simple roots of the Lie algebra $\mathfrak{g}$ of $\mathsf{G}$ such that $-w_0 (\Theta) = \iota(\Theta) = \Theta$. We define 
\[
\mathfrak{p}_\Theta:=\mathfrak{z}_{\mathfrak{g}}(\mathfrak{a})\oplus \sum_{\theta\in \mathrm{Span}(\Delta\setminus \Theta)} \mathfrak{g}_{\theta}\oplus\bigoplus_{\theta\in \Sigma^+\setminus \mathrm{Span}(\Delta\setminus \Theta)}\mathfrak{g}_\theta
\]
where $\mathfrak{z}_{\mathfrak{g}}(\mathfrak{a}) = \{X\in \mathfrak{g}\mid [X,A]=0\text{ for all }A\in \mathfrak{a}\}$. The normalizer $\mathsf{P}_\Theta$ of $\mathfrak{p}_\Theta$ in $\mathsf{G}$ is called the $\Theta$-\emph{parabolic subgroup}.  Note that the flag manifold $\mathcal{F}_\Theta:=\mathsf{G}/\mathsf{P}_\Theta$  can be identified with the space of $\Ad_\mathsf{G}$ orbits of $\mathfrak{p}_\Theta$. For this reason, $\mathfrak{p}_\Theta$ is often called the \emph{standard flag}.  We also have the \emph{standard opposite flag} $\mathfrak{p}_\Theta ^{\mathrm{op}} := \Ad_{\widehat{w_0}} \mathfrak{p}_{\Theta}$. The opposite $\Theta$-parabolic subgroup $\mathsf{P}_\Theta ^{\mathrm{op}}$ is the normalizer of $\mathfrak{p}_\Theta ^\mathrm{op}$. It gives the opposite flag manifold $\mathcal{F}_\Theta^\mathrm{op} = \mathsf{G}/\mathsf{P}_\Theta^\mathrm{op}$.

Now we define a $\Theta$-Anosov representation. In fact, there are several equivalent characterizations for Anosov representation. The original definition, due to Labourie \cite{L06}, deals with representations of closed surface groups into a noncompact semisimple Lie group. His notion of Anosov representation is now referred to as Borel Anosov representation. Guichard--Wienhard \cite{GW12} expand Labourie's definition to representations of Gromov hyperbolic groups and introduce various modes of Anosov-ness. A more geometric approach can be found in Porti--Leeb--Kapovich \cite{KLP3}. Bochi--Potrie--Sambarino \cite{BPS} provides a simple characterization for Anosov representation only using uniform singular value gap property.  In this paper, we present the definition by  Gu\'eritaud--Guichard--Kassel--Wienhard \cite{gueritaud} which addresses the existence of limit maps. 

We say a representation $\rho:\pi_1(S) \to \mathsf{G}$ is $\Theta$-\emph{Anosov} if the following properties are satisfied:
\begin{itemize}
\item For any $\theta\in \Theta$, we have $\theta\circ\lambda(\rho(\gamma))\to \infty$ as $|\gamma|_\infty \to \infty$, where $|\gamma|_\infty$ is the stable length, $|\gamma|_\infty = \lim_{n\to \infty} |\gamma^n|/n$. 
    \item There exists a pair of $\rho$-equivariant continuous \emph{limit maps} $\xi_\rho:\partial_\infty \pi_1(S) \to \mathcal{F}_\Theta$ and $\xi_\rho^\mathrm{op}:\partial_\infty \pi_1(S) \to \mathcal{F}_\Theta^\mathrm{op}$ such that
    \begin{itemize}
    \item $\xi_\rho$ and $\xi_\rho^\mathrm{op}$ are dynamics preserving: for any infinite order element $\gamma\in \pi_1(S)$ with its attracting fixed point $\gamma^+$ in $\partial_\infty\pi_1(S)$, the corresponding points $\xi_\rho(\gamma^+)$ and $\xi_\rho^\mathrm{op}(\gamma^+)$ are attracting points of $\rho(\gamma)$. 
    \item the pair $(\xi_\rho,\xi_\rho^\mathrm{op})$ is transverse (or opposite): for any distinct $x,y\in \partial_\infty\pi_1(S)$, $(\xi_\rho(x),\xi_\rho^\mathrm{op}(y))$ is in the $\Ad_\mathsf{G}$-orbit of $(\mathfrak{p}_\Theta, \mathfrak{p}_\Theta^{\mathrm{op}})$
     \end{itemize}
\end{itemize}

In fact, if $\iota(\Theta) = \Theta$, one limit map $\xi_\rho$ determines the other limit map $\xi_\rho^\mathrm{op}$. More precisely, Define $P:\mathcal{F}_\Theta \to \mathcal{F}_\Theta^\mathrm{op}$ by $\Ad_g \mathfrak{p}_\Theta \mapsto \Ad_{g\widehat{w_0}}\mathfrak{p}_\Theta^\mathrm{op}$. Since $\widehat{w_0}\mathsf{P}_\Theta \widehat{w_0}^{-1} = \mathsf{P}_{\iota(\Theta)} ^\mathrm{op} = \mathsf{P}_\Theta^\mathrm{op}$, this map $P$ is well-define.  Then, one can observe that $\xi_\rho ^\mathrm{op}(x) = P(\xi_\rho(x))$. 

The set $\Hom_\Theta(\pi_1(S),\mathsf{G})$ of $\Theta$-Anosov representations is open in $\Hom(\pi_1(S), \mathsf{G})$.  Therefore, if $\mathsf{G}$ is a complex semisimple Lie group,  $\Hom_\Theta(\pi_1(S),\mathsf{G})$ inherits the complex structure. 

Let $\mathsf{G}$ be a semisimple complex Lie group and $\mathsf{G}_{\mathrm{split}}$ its split real form. If $\rho:\pi_1(S)\to \mathsf{G}_{\mathrm{split}}$ is $\Theta$-Anosov then the induced representation $\pi_1(S)\to \mathsf{G}_{\mathrm{split}} \to \mathsf{G}$ is also $\Theta$-Anosov. This follows from the fact that the root structures of $\mathfrak{g}_{\mathrm{split}}$, $\mathfrak{g}$, and $\mathfrak{g}^\bR$ coincide. 

Observe that for any $\gamma\in \pi_1(S)\setminus \{1\}$ and a $\Theta$-Anosov representation $\rho$, the image $\rho(\gamma)$ is always conjugate into the subgroup $\mathsf{P}_\Theta\cap \mathsf{P}_\Theta^\mathrm{op}$. In particular if $\rho$ is $\Delta$-Anosov into a simple complex Lie group $\mathsf{G}$, $\rho(\gamma)$ is conjugate into $\mathsf{P}\cap \mathsf{P}^\mathrm{op}$. 

Consider the complex simple adjoint Lie group $\mathsf{G}$ and a closed orientable surface $S$ with genus $g \ge 2$. There exists preferred representation $\tau_\mathsf{G}:\PSL_2(\bR) \to \mathsf{G}_\mathrm{split}$ obtained by taking a regular $\mathfrak{sl}_2$-triple. We call a representation of $\pi_1(S)$ into $\mathsf{G}_\mathrm{split}$ \emph{Fuchsian} if it is of the form $\tau_\mathsf{G}\circ \rho$ where $\rho:\pi_1(S)\to \PSL_2(\bR)$ is the holonomy of a hyperbolic structure on $S$. The \emph{Hitchin component} is the connected component of $\Hom(\pi_1(S), \mathsf{G}_{\mathrm{split}})$ containing a Fuchsian representation. It is known \cite{L06,guichard_wienhard_labourie} that Hitchin representations are $\Theta$-Anosov with $\Theta=\Delta$. We denote by $\Hit_{\mathsf{G}_\mathrm{split}}(S)$, or just simply $\Hit_\mathsf{G}(S)$, the space of $\mathsf{G}_\mathrm{split}$-conjugacy classes of Hitchin representations.  Hitchin~\cite{hitchin92} showed that  $\Hit_{\mathsf{G}_\mathrm{split}}(S)$ is a cell of dimension $(2g-2)\cdot \dim \mathsf{G}_{\mathrm{split}}$. 

\section{Geodesic laminations and transverse cocycles}\label{sec:cocycle}
Throughout this section, $S$ will denote a closed orientable surface of genus $g>1$ equipped with an auxiliary hyperbolic structure.

\subsection{Notations and hyperbolic geometry estimates}
We clarify conventions and notations that will be used throughout this paper. Since we deal with closed geodesics possibly with self-intersections, notions involving intersections require special care. 

We also collect several technical hyperbolic geometry estimates for the sake of completeness. Although many results in this section are perhaps well-known to experts, we could not locate a proper reference. 

All geodesics in this paper are regarded as subsets of $S$ or $\mathbb{H}^2$. A geodesic  $\gamma$ in $S$ lifts to a subset of the projectivized tangent bundle $\PT S$; a point $x\in \gamma$ lifts to points $(x,l_1),\cdots, (x,l_m)\in\PT_x S\subset \PT S$ where $l_i$ are the tangent lines of $\gamma$ at $x$.  We denote this lift of $\gamma$ to $\PT S$ by $\widehat{\gamma}$. For a simple point $x\in \gamma$, $\widehat{\gamma}\cap \PT_x S$ consists of a single element. We will denote this unique element by $\widehat{\gamma}_x$.

An \emph{oriented geodesic} is a geodesic with a prescribed lift to the unit tangent bundle $\UT S$.  We sometimes write $\vec{\gamma}$ to emphasize that $\gamma$ is oriented. The same notation $\vec{\gamma}$ is used to denote the lifted image of $\gamma$ to $\UT S$. For a simple point $x\in \gamma$, $\vec{\gamma}_x\in \UT_x S$ is the unit tangent vector at $x$ with respect to the orientation of $\gamma$. 

 Given two geodesics $\gamma$ and $\eta$, a pair $p=(v,w)$ with $v\in \widehat{\gamma}$ and $w\in \widehat{\eta}$ is called an \emph{intersection} if $v$ and $w$ lie in the same fiber of the fiber bundle $\PT S\to S$.  An intersection $p=(v,w)\in \gamma\sharp \eta$ is \emph{transverse} if $v\ne w$. Let $\gamma\sharp \eta$ be the set of transverse intersections of $\widehat{\gamma}$ and $\widehat{\eta}$. Notice that $\gamma\sharp \eta$ differs from the usual intersection $\gamma\cap \eta$, especially when $\gamma$ and $\eta$ intersect multiple times at a single point. 
 
 If $\gamma$, $\eta$ and $S$ are oriented, we say $(v,w)$, $v\in \vec{\gamma}$, $w\in \vec{\eta}$ is an intersection if $v,w$ belong to the same fiber of $\UT S\to S$. We use the same notation $\gamma\sharp \eta$ to denote the set of intersections of $\vec{\gamma}$ and $\vec{\eta}$. We say that $p=(v,w)\in \gamma\sharp \eta$ is \emph{positive} (\emph{negative}, respectively) if $(v,w)$, as an ordered tuple of elements in $\mathrm{T}_x S$,  forms a positive (negative, respectively) basis for $\mathrm{T}_x S$.

We will often work in the universal cover $\mathbb{H}^2$ rather than $S$. In this case, geodesics have no self-intersections, and any two distinct geodesics intersect at most once. Therefore, for geodesics $\gamma$ and $\eta$ in  $\mathbb{H}^2$, we do not distinguish between $\gamma\sharp \eta$ and $\gamma\cap \eta$. We say $\vec{\gamma}$ and $\vec{\eta}$ \emph{intersect positively} if $(\vec{\gamma}_x, \vec{\eta}_x)\in \gamma \sharp \eta$ is positive. If $\vec{\gamma}$ intersects  $\vec{\eta}$ positively then, $\vec{\eta}$ intersects $\vec{\gamma}$ negatively. This notion is well-defined since $\gamma$ and $\eta$ meet at most at one point in $\mathbb{H}^2$.

If we are given two distinct points $v,w\in \PT_x S$,  we can measure the \emph{counter-clockwise angle}  $\angle^X (v,w)$  in the counter-clockwise direction from $v$ to $w$ with respect to the hyperbolic structure $X$ on $S$. For convenience, we sometimes write $\angle (v,w)$ if $X$ is clear from the context. Note that $0< \angle (v,w)< \pi$ and $\angle(v,w) = \pi-\angle(w,v)$. To avoid ambiguity, we formally declare $\angle(v,w)=0$ if $v=w$.   The counter-clockwise angle also induces a metric $\der_{\mathbb{P}\bR^1}$ on the projectivized tangent space $\PT_xS$. More precisely, for $v,w\in \PT_xS$ define,
\[
\der_{\mathbb{P}\bR^1}(v, w) = \min\{\angle(v,w),\angle(w,v)\}.
\]

For two geodesics $\gamma$ and $\eta$ intersecting at $p=(v,w)\in \gamma\sharp \eta$ with $v\in \widehat{\gamma}\cap \mathrm{T}_xS$, and $w\in \widehat{\eta}\cap \mathrm{T}_x S$, the counter-clockwise intersection angle $\angle^X_p(\gamma, \eta)$ at $p$ is defined to be $\angle^X (v,w)$. If $\gamma$ and $\eta$ are geodesics in $\mathbb{H}^2$, we may omit the intersection point and simply write $\angle(\gamma, \eta)$ to denote their intersection angle.

For a compact geodesic arc $\gamma$ in $\mathbb{H}^2$, the restricted projectivized tangent bundle $\PT\mathbb{H}^2|\gamma$ will appear frequently. This bundle can be trivialized, $\PT\mathbb{H}^2|\gamma=\gamma\times \mathbb{PR}^1$, using the parallelism. Being compact, all metrics on $\PT\mathbb{H}^2|\gamma$ are bi-Lipschitz equivalent. Therefore, if we are insensitive to multiplicative errors, we may choose a convenient metric. One such example is the product metric $\der((x,l_x),(y,l_y))=\der_\gamma(x,y) + \der_{\mathbb{PR}^1}(l_x,l_y)$ on $\gamma\times \mathbb{PR}^1$, where $\der_\gamma$ is the arc metric on $\gamma$ induced from the hyperbolic metric on $S$.

We also consider the unit tangent bundle $\UT\mathbb{H}^2|\gamma$ restricted to a compact geodesic arc $\gamma$ in  $\mathbb{H}^2$. As in the projective tangent bundle case, we choose a convenient metric on $\UT\mathbb{H}^2|\gamma$ as follows.  Identify $\UT\mathbb{H}^2$ with $(\partial_\infty \mathbb{H}^2)^{(3)}$, the set of distinct triples of $\partial_\infty \mathbb{H}^2$. The metric $\dUT$ on $\UT\mathbb{H}^2|\gamma$ is the restriction of the product metric induced from the metric $\der_\infty$ on $\partial_\infty\mathbb{H}^2$. If $x$ and $y$ belong to $\UT\mathbb{H}^2|\gamma$, and if $x^+$ and $y^+$ denote the forward endpoints of the oriented geodesics generated by $x$ and $y$, we have $\der_\infty (x^+,y^+)\le \dUT(x,y)$.

Suppose that $S$ is oriented. If $v,w\in \UT_x\mathbb{H}^2$, and if $\der_{S^1}$ denotes the angular metric on the unit tangent space $\UT_x\mathbb{H}^2$, we have
\[
\der_{S^1} (v,w) = \begin{cases}
    \angle([v],[w]) & \text{if }  (v,w) \text{ is a positive basis or }v=w\\
    \pi-\angle([v],[w])& \text{if }  (v,w)\text{ is a  negative basis or }v=-w
\end{cases}
\]
where $[v]$ and $[w]$  denote the lines generated by $v$ and $w$ respectively. For a compact geodesic arc $\gamma$, we frequently use the product metric $\der((x,v_x),(y,v_y))=\der_\gamma(x,y)+\der_{S^1}(v_x,v_y)$ on $\UT\mathbb{H}^2|\gamma=\gamma\times S^1$. This metric is again bi-Lipschitz equivalent to the metric inherited from $\der_{\UT}$.

For a given geodesic arc $\gamma\subset S$ or $\mathbb{H}^2$, and a  constant $\theta$ with $0<\theta<\pi/2$, define 
\[
\mathfrak{A}(\gamma, \theta) = \left\{\begin{aligned}\text{geodesics } \eta \text{ intersecting the interior of }\gamma \text{  such that } \\\frac{\pi}{2}-\theta<\angle_p(\gamma,\eta)<\frac{\pi}{2}+\theta\text{ for all }p\in \eta\sharp\gamma \end{aligned}\right\}.
\]
If $\vec{\gamma}$ is a oriented geodesic in $\mathbb{H}^2$, we let 
\[
\mathfrak{A}^+(\vec{\gamma}, \theta) = \{\vec{\eta}\mid  \eta\in \mathfrak{A}(\gamma, \theta) \text{ and } \vec{\gamma} \text{ intersect }\vec{\eta} \text{  positively}\}.
\]

The following lemma is elementary but will be used in several places. 
\begin{lemma}\label{lem:lengthheight}
    Let $\epsilon>0$, and $d>0$ be given numbers. Let $\gamma_1$ and $\gamma_2$ be vertical geodesics in the upper half plane $\mathbb{H}^2$ with base points $(0,0)$ and $(0,d)$, respectively.  Then, there exists a constant $C>0$ depending on $\epsilon$, and $d$ such that for any geodesic arc $\eta$  joining $(0,a)$ and $(d,b)$ with $\epsilon<\angle(\eta,\gamma_i)<\pi-\epsilon$ for $i=1,2$, we have
    \[
    \frac{d}{C\cdot \max\{a,b\}}\le \frac{d}{C\cdot \min\{a,b\}} < \ell(\eta) <\frac{Cd}{\max\{a,b\}}\le \frac{Cd}{\min\{a,b\}}.
    \]
\end{lemma}

\begin{lemma}\label{lem:parallel}
    Let $\vec{\gamma}$ be an oriented compact geodesic arc on $\mathbb{H}^2$ and let $\theta$ be a given real number with $0<\theta<\pi/2$. There is a constant $C>1$ depending on $\gamma$ and $\theta$ such that for any oriented parallel geodesics $\vec{g}$ and $\vec{h}$ in $\mathfrak{A}^+(\vec{\gamma}, \theta)$, we have
    \[
    \frac{1}{C}\cdot\der_\gamma (p,q) < \dUT(\vec{g}_p,\vec{h}_q)<C \cdot \der _\gamma(p,q).
    \]
    where $p=\gamma\cap g$ and $q=\gamma\cap h$ are intersection points.
\end{lemma}
\begin{proof}
    Since $\UT \mathbb{H}^2|\gamma$ is compact, all metrics on $\UT \mathbb{H}^2|\gamma$ are bi-Lipschitz equivalent. Thus, we prove the lemma using a convenient metric. Trivialize $\UT \mathbb{H}^2|\gamma=\gamma\times S^1$ using the parallel transport along $\gamma$.  We give $\gamma\times S^1$ the product metric, where $S^1$ and $\gamma$ are equipped with the angular and the path metrics, respectively.  Let $\dUT$ be the associated distance function. It is clear that $\der_\gamma(p,q)<\dUT(\vec{g}_p,\vec{h}_q)$ for any oriented geodesics  $\vec{g}$ and $\vec{h}$ intersecting $\vec{\gamma}$ at $p$ and $q$. 

    We claim that if $g$ and $h$ are parallel, $|\angle_p (\gamma,g)-\angle_q (\gamma,h)|<C_1\cdot \der_\gamma(p,q)$ for some constant $C_1$ depending only on $\gamma$. To show this, assume first that $\angle(\gamma,g)\le\angle(\gamma,h)$. We apply an isometry so that $g$ is the geodesic joining $0$ and $\infty$ in the upper half-plane model and $p= \sqrt{-1}$ and $q= a+b\sqrt{-1}$ for some $a>0$, $b>0$. Consider a vertical geodesic $h'$ passing through $q$ as in Figure~\ref{fig:angle} so that
    \[
    |\angle (\gamma,g)-\angle (\gamma,h)|=\angle (\gamma,h) - \angle (\gamma,g) \le \angle  (\gamma,h')-\angle (\gamma,g).
    \]
    The geodesics $g$, $\gamma$ and $h'$ bound a geodesic triangle $\triangle$ with one ideal vertex. Observe that the total area of the $\triangle$ is given by $\angle  (\gamma,h')-\angle (\gamma,g)$. Draw a horizontal line segment $l$ from $q$  to the imaginary axis if $b\le 1$ or from $p$ to $h'$ if $b>1$. Then $g$, $l$ and $h'$ forms an ideal (not geodesic) triangle  $\triangle'$ containing $\triangle$. The area of $\triangle'$ is 
    \[
    \operatorname{Area}(\triangle')=\int_{\min(b,1)}^\infty \int_0 ^a \frac{\der x \, \der y}{y^2} = \frac{a}{\min(1,b)}.
    \]
    Since $\gamma$ is compact, there exists $\theta'$ depending on $\theta$ and $\gamma$ such that  $g,h'\in \mathfrak{A}^+(\gamma, \theta')$. Thus, Lemma~\ref{lem:lengthheight} shows that exists a constant $C_1>0$ depending on $\gamma$ and $\theta$ such that 
    \[
    \operatorname{Area}(\triangle')=\frac{a}{\min(1,b)}< C_1\cdot \der(p,q).
    \]
    Consequently, we obtain  
    \[
     |\angle_p (\gamma,g)-\angle_q (\gamma,h')|=\operatorname{Area}(\triangle)\le\operatorname{Area}(\triangle')=\ell(l)<C_1\cdot \der_\gamma(p,q).
    \]
    A similar argument applies when $\angle_p(\gamma,g)>\angle_q(\gamma,h)$.
    
    Since $\gamma$ intersects $g$ and $h$ positively, we obtain $\angle_p (\gamma, g)=\der_{S^1}(\vec{\gamma}_p, \vec{g}_p)$ and $\angle_q(\gamma,h) = \der_{S^1}(\vec{\gamma}_q,\vec{h}_q)$. It follows that
    \begin{align*}
    \dUT(\vec{g}_p,\vec{h}_q)&= \der_\gamma(p,q)+|\der_{S^1}(\vec{\gamma}_p,\vec{g}_p)-\der_{S^1}(\vec{\gamma}_q, \vec{h}_q)|\\&<\der_\gamma(p,q)+C_1\cdot \der_\gamma (p,q)\\
    &= (1+C_1)\cdot \der_\gamma (p,q).
    \end{align*}
Setting $C=1+C_1$ completes the proof.
\end{proof}
    \begin{figure}[!htb]
        \centering
        \includegraphics[width=\linewidth]{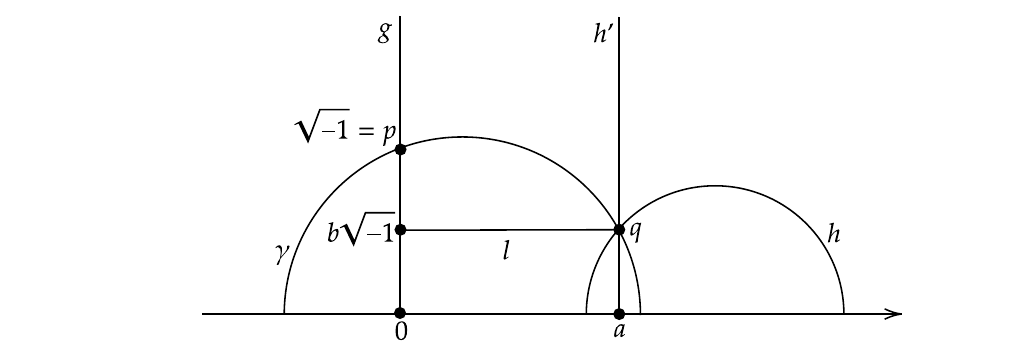}
        \caption{Configuration for the proof of Lemma~\ref{lem:parallel} with  $b<1$.}
        \label{fig:angle}
    \end{figure}

\subsection{Geodesic laminations}
A \emph{geodesic lamination} is a closed subset of $S$ (or $\mathbb{H}^2$) that can be decomposed as a union of disjoint geodesics. Given a geodesic lamination $\mu$, the connected components of $S\setminus \mu$ are called \emph{complementary regions}. A geodesic lamination is \emph{maximal} if all complementary regions are ideal triangles. Since a geodesic lamination itself has zero hyperbolic area, a maximal geodesic lamination always has $4(g-1)$ complementary ideal triangles, where $g$ is the genus of $S$.

Let $p:\mathbb{H}^2\to S$ be the covering map. Let $\gamma$ be a compact geodesic transverse to a maximal geodesic lamination $\mu$ of $S$. Choose a constant speed parametrization $f_\gamma:[0,1]\to \gamma\subset S$ of $\gamma$. Define $\mathcal{C}_\gamma(\mu)$ to be the set of connected components of $[0,1]\setminus f_\gamma^{-1}(\mu)$. To understand $\mathcal{C}_\gamma (\mu)$ geometrically, let $\widetilde{f_\gamma}:[0,1]\to \mathbb{H}^2$ be the lift of $f_\gamma$. Then, the image of $\widetilde{f_\gamma}$ is a compact geodesic arc in $\mathbb{H}^2$ that is transverse to the lift $\widetilde{\mu}$. We will identify $\mathcal{C}_\gamma(\mu)$ with the set of connected component of $\widetilde{f_\gamma}([0,1])\setminus \widetilde{\mu}$. For each $J\in \mathcal{C}_\gamma(\mu)$, there exists a unique ideal triangle in $S\setminus \mu$ that contains $J$. 

If a transverse geodesic $\gamma\subset S$ is oriented, we choose the orientation-preserving constant-speed parametrization $f_\gamma$. As before let $\widetilde{f_\gamma}$ be the lift. Then $\widetilde{f_\gamma}$ induces a total ordering $<$ on the lifted geodesic arc $\widetilde{f_\gamma}([0,1])$ in $\mathbb{H}^2$. Given $J\in \mathcal{C}_\gamma(\mu)$, let $J^0$ and $J^1$ be the left and right endpoints of $J$, respectively. That is, with respect to the induced ordering on $\widetilde{f_\gamma}([0,1])$, the interval $J$ can be written as $J=(J^0,J^1):=\{t\in \widetilde{f_\gamma}([0,1])\mid J^0<t<J^1\}$.

A sequence (or an enumerated subset) $\{J_i\}_{i\in \mathbb{Z}}$ of $\mathcal{C}_{\gamma}(\mu)$ is \emph{ordered} if $J_i<J_{i+1}$ for all $i$. Given any subset $\mathcal{J}\subset \mathcal{C}_{\gamma}(\mu)$, we can form the sequence $\vec{\mathcal{J}}=\{J_i\}_{i\in \mathbb{Z}}$ by enumerating the elements of $\mathcal{J}$ so that $\vec{\mathcal{J}}$ is ordered. 

The set $\mathcal{C}_\gamma(\mu)$ can be partitioned into finitely many parallel classes. We say that $J_1$ and $J_2$ are in the same \emph{parallel class} if there is an ideal triangle $T\subset S\setminus \mu$ that contains both $J_1$ and $J_2$ and they separate the same ideal vertex of $T$ from the other ideal vertices. See Figure~\ref{fig:divrad} for an example.  Notice that there are at most $12(g-1)$ parallel classes of $\mathcal{C}_\gamma(\mu)$.

For an element $J\in \mathcal{C}_{\gamma}(\mu)$, let  $\ell(J)$ be the arc-length of $J$ with respect to the hyperbolic metric on $S$.

Let $\mathcal{GL}(S)$ be the set of geodesic laminations. We equip $\mathcal{GL}(S)$ with a topology as follows. Let  $\mathcal{C}(\PT S)$ be the set of closed subsets in  the projective tangent bundle $\PT S$. Each geodesic leaf of a geodesic lamination $\mu$ in $S$ can be naturally lifted to $\PT S$. This enables us to lift $\mu$ to $\PT S$. By the same convention, let $\widehat{\mu}$ be the lifted image of $\mu$ in $\PT S$. Via this lifting, $\mathcal{GL}(S)$ can be regarded as a subspace of  $\mathcal{C}(\PT S)$.  For any point $x\in \mu$, $\widehat{\mu}_x$ denotes the unique element $\widehat{\mu}\cap \mathrm{T}_x S$. This is well defined as leaves of $\mu$ do not intersect. 

We equip $\PT S$ with any Riemannian metric and let  $\dPT$ be the associated distance. This $\dPT$ induces the Hausdorff distance $\mathrm{D_H}$ on  $\mathcal{C}(\PT S)$, turning it into a metric space. In this topology, it is well-known that $\mathcal{GL}(S)$ is a compact subset of  $\mathcal{C}(\PT S)$. As its meaning is clear, we write simply $\mathrm{D_H}(\mu, \tau)$ instead of $\mathrm{D_H}(\widehat{\mu}, \widehat{\tau})$. 

Let $\mathcal{C}(S)$ be the set of closed subsets in $S$ equipped with the Hausdorff metric $\mathrm{d_H}$ induced by a hyperbolic structure on $S$. The set of geodesic laminations $\mathcal{GL}(S)$ can be viewed as a compact subset of $\mathcal{C}(S)$ as well. It is known that the identity map $(\mathcal{GL}(S),\mathrm{D_H})\to (\mathcal{GL}(S),\mathrm{d_H})$ is a homeomorphism.

\begin{lemma}\label{lem:approx}
Let $\{\mu_n\}$ be a sequence of geodesic laminations in $S$ that converges to a geodesic lamination $\mu$. Let $\gamma$ be a geodesic arc transverse to $\mu$ and $\mu_n$ for all $n$. Let 
\begin{align*}
\Gamma_n &:= \{\widehat{(\mu_n)}_x\mid x\in \mu_n \cap \gamma\}\\
\Gamma&:=\{\widehat{\mu}_x\mid x\in \mu\cap \gamma\}
\end{align*}
be closed subsets in $\PT S$. Then  $\{\Gamma_n\}$ converges to $\Gamma$ in the Hausdorff topology.  
\end{lemma}
\begin{proof}
Lift everything to the universal cover with the projective model, where all geodesics are straight lines. We introduce affine coordinates on a relatively compact neighborhood $U$ of $\gamma$ in $\mathbb{H}^2$ so that $\gamma = \{(t,0)\mid 0\le t\le l\}$. We introduce the flat Euclidean metric on $U$. Trivialize $\PT\mathbb{H}^2|U=U\times \mathbb{PR}^1$ using the parallelism with respect to the Euclidean metric. We will use the product distance $\der_{\mathrm{Euc}}+\der_{\RP^1}$ on $U\times \mathbb{PR}^1$ which is bi-Lipschitz equivalent to the metric induced by $\dPT$. 

 Note that the Euclidean counter-clockwise angles $\angle^{\mathrm{Euc}}(\gamma, g)$ \footnote{This is not equal to the hyperbolic angle $\angle(\gamma, g)$ since the projective model is not conformal to the Euclidean space.} from $\gamma$ to geodesic leaves $g$ of $\mu$ that intersect $\gamma$ are bounded between $\frac{\pi}{2}-\theta_0$ and $\frac{\pi}{2}+\theta_0$ for some $\theta_0$ with $0<\theta_0<\frac{\pi}{2}$. Since $\{\mu_n\}$ converges to $\mu$, we can find $N_0$ such that whenever $n>N_0$, $\frac{\pi}{2}-\theta_0 < \angle^\mathrm{Euc}(\gamma,g)<\frac{\pi}{2}+\theta_0$ for all leaves $g$ of $\mu_n$  intersecting $\gamma$.

Let $g$ and $h$ be geodesics intersecting $\gamma$ at $x$ and $y$ respectively. Suppose that the Euclidean counter-clockwise angles from $\gamma$ to $g$ and to $h$ are bounded between $\frac{\pi}{2}-\theta_0$ and $\frac{\pi}{2}+\theta_0$. Choose $\delta_0$ such that $\theta_0 + \delta_0<\frac{\pi}{2}$ and that the $\delta_0$-Euclidean balls around each point in $\gamma$ stay in $U$. By  Euclidean geometry as in Figure~\ref{fig:euclidean}, if $z\in h\cap U$ satisfies $\dPT( \widehat{g}_x,\widehat{h}_z)<\delta_0$ then we have
\[
\dPT(\widehat{g}_x, \widehat{h}_y) < 2\delta_0+\delta_0\cdot\tan (\theta_0 + \delta_0).
\]

\begin{figure}[!htb]
    \centering
    \includegraphics[width=\linewidth]{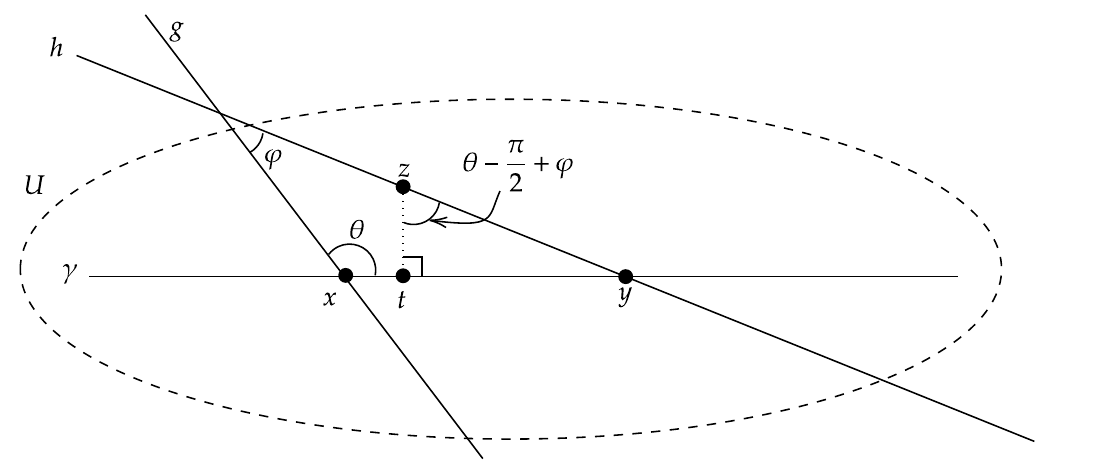}
    \caption{Since $\dPT(\widehat{g}_x,\widehat{h}_z)=\der_{\mathrm{Euc}}(x,z)+\varphi<\delta_0$, we have $\varphi< \delta_0$. The length of the segment $\overline{xz}$ is bounded from above by $\delta_0$ as well. Hence, the length of the line segment $\overline{xt}$ is at most $\delta_0$ and the length of $\overline{ty}$ is at most $|\overline{zt}|\cdot \tan (\theta-\frac{\pi}{2}+\varphi) < \delta_0 \cdot \tan (\theta_0 +\delta_0)$. Consequently, we obtain $\dPT(\widehat{g}_x,\widehat{h}_y)=\der_{\mathrm{Euc}}(x,y)+\varphi< 2\delta_0+\delta_0\cdot\tan (\theta_0 + \delta_0)$.}
    \label{fig:euclidean}
\end{figure}

Given $\epsilon>0$, let $0<\delta<\delta_0$ be a real number satisfying $2\delta+\delta\cdot\tan (\theta_0 + \delta_0)<\epsilon$. Let $N>N_0$ be an integer such that $\mathrm{D_H}(\mu_n, \mu)<\delta$ for all $n>N$. Then the estimate explained in the above paragraph shows that the set $\Gamma$ is contained in the $\epsilon$-neighborhood of $\Gamma_n$ for each $n>N$. Since the estimate is symmetric, we conclude that each $\Gamma_n$ sits inside the $\epsilon$-neighborhood of $\Gamma$ whenever $n>N$. Therefore, $\mathrm{D_H}(\Gamma_n, \Gamma)<\epsilon$ for all $n>N$. 
\end{proof}

We can approximate a geodesic lamination using a train-track.  Formally, a \emph{train-track} is a collection of smoothly embedded  images of $[0,1]\times[0,1]$ in $S$, called \emph{branches} satisfying 
\begin{itemize}
    \item two branches are either disjoint or intersect along their horizontal edges $[0,1]\times \{0,1\}$,
    \item a horizontal edge of a branch is covered by horizontal edges of other branches. 
\end{itemize}
The images of $[0,1]\times\{y\}$ for $0<y<1$ are called ties and  the images of $[0,1]\times\{0,1\}$ are called \emph{switches}.  A geodesic lamination is \emph{carried by} a train-track if it is contained in the union of branches of the train-track and is transverse to ties and switches.

If we orient a switch $s$, the set of branches that intersect $s$ can be decomposed into two subsets; a branch meeting $s$ from the left is called an \emph{inbound} branch. An \emph{outbound} branch is one that intersects $s$ from the right. 

For any $\delta$-neighborhood $\mathcal{N}_\delta(\mu)=\{x\in S \mid \der(x,\mu)<\delta\}$ of a geodesic lamination $\mu$, we can find a train-track $\mathcal{T}$ contained in $\mathcal{N}_\delta(\mu)$ that carries $\mu$  \cite[Theorem~1.6.5]{penner}.

The following result is also well-known; see for example \cite[Lemma~4.2.1]{otal}. 
\begin{lemma}\label{lem:traintracknbd}
    Given a geodesic lamination $\mu$ carried by a train-track $\mathcal{T}$, there is a constant $\epsilon>0$ such that any geodesic lamination $\lambda$ with $\mathrm{D_H}(\mu, \lambda)<\epsilon$ is carried by  $\mathcal{T}$.  
\end{lemma}

The set $\mathcal{N}_\delta(\gamma\cap \mu)=\mathcal{N}_\delta(\mu) \cap \gamma$ is a countable union of disjoint  open intervals, say $\mathcal{N}_\delta(\gamma\cap \mu) = \bigcup_i J_i$. The total length of $\mathcal{N}_\delta(\gamma\cap \mu)$ is given by $\sum_i \ell(J_i)$. We need the following:

\begin{lemma}\label{lem:hdimzero}
 Let $\mu$ be a geodesic lamination in $\mathbb{H}^2$ and let $\gamma$ be a geodesic arc in $\mathbb{H}^2$. The total arc-length of $\mathcal{N}_\delta(\gamma\cap \mu)$ converges to 0 as $\delta\to 0$.
\end{lemma}
\begin{proof}
Let $\epsilon>0$ be given. Since the Hausdorff dimension of $\mu\cap \gamma$ equals 0, we can find an open covering $\{I_i\}$ of $\gamma\cap \mu$ such that $\sum \ell(I_i) <\epsilon$. Let $J= \gamma\setminus \mu$ be an open subset of $\gamma$.  Observe that $\{I_i\}\cup \{J\}$ is an open covering of the compact set $\gamma$.   Let $\delta_1>0$ be a Lebesgue number for this covering  and let $\delta = \delta_1/2$. Then, for any $x\in \gamma\cap \mu$, the $\delta$-neighborhood of $x$ in $\gamma$ is contained in some $I_i$. In particular, this implies that $\mathcal{N}_\delta(\gamma \cap \mu)\subset \bigcup_{i}I_i$. Therefore, the total length of $\mathcal{N}_\delta(\gamma\cap \mu)$ is at most $\epsilon$. Since $\epsilon$ was arbitrary, the lemma follows. 
\end{proof}

Suppose that we are given a maximal geodesic lamination $\mu$ of $S$ and a transverse geodesic arc $\gamma$. Now we define the \emph{divergence radius} $\dr_\mu:\mathcal{C}_\gamma(\mu)\to \mathbb{N}$ with respect to $\mu$. The role of $\dr_\mu(J)$ is to measure how deep $J$ lies inside a cusp neighborhood of a complementary triangle. 
This number can be defined in various ways as long as it satisfies properties (i) and (ii) stated below. We present one  possible construction. 

Choose a train-track $\mathcal{T}$ that carries $\mu$. We may further assume that $\gamma$ does not pass through any switch of $\mathcal{T}$. For a branch $e$ of $\mathcal{T}$,  $e\setminus \mu$ consists of countably many connected components which we call rectangles. The set of rectangles can be decomposed into two classes. One is finitely many rectangles $R_1, \cdots, R_m$ that meet $\partial \mathcal{T}$. The other class consists of $12(g-1)$ infinite chains $\mathfrak{C}^i=\{C^i_{1}, C^i_2, \cdots\}$ of rectangles, for $i=1,2,\cdots, 12g-12$,  where $C^i_j$ shares an edge with $C^i_{j+1}$ and where $\bigcup_j C^i_j$ forms a spike. Now let $J\in \mathcal{C}_\gamma(\mu)$ be given. Since $\gamma$ does not pass through switches, $J$ is either entirely contained in a rectangle $C^i_j$ of an infinite chain or does not intersect any rectangles of infinite chains. We set  
\[
\dr_{\mu}(J)=\begin{cases}
    n&\text{if }J\text{ is contained in } C^i_n\\
    0& \text{otherwise}
\end{cases}.
\]
Figure~\ref{fig:divrad} illustrates how $\dr_\mu$ is defined.

As alluded, it enjoys the following properties  \cite{bonahon96,BD,pfeil}:
    \begin{enumerate}
        \item  There exists a uniform constant $N$ depending only on $\gamma$ and the topological type of $S\setminus\mu$ such that for any $n\ge 0$,
        \[
        \sharp\{J\in \mathcal{C}_\gamma(\mu)\mid\dr_\mu(J)=n\}\le N.
        \]
        \item There exist positive constants $B$, $C$, $B'$ and $C'$ such that
        \[
        B \e^{-C\dr_\mu(J)}\le \ell(J)\le B' \e^{-C'\dr_\mu(J)}.
        \]
    \end{enumerate}

\begin{figure}[!htb]
    \centering
    \includegraphics[width=\linewidth]{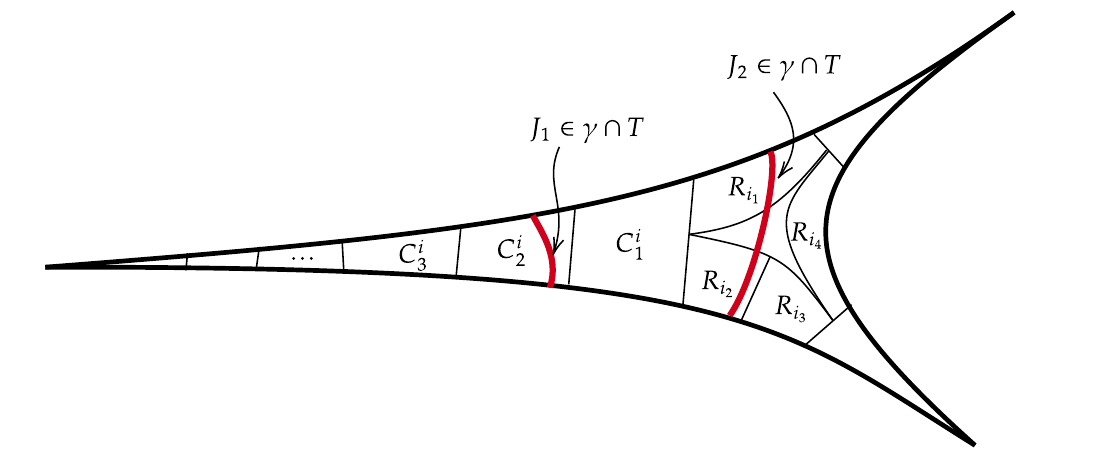}
    \caption{A complementary ideal triangle $T$ and its train-track neighborhood. The intervals $J_1$ and $J_2$ are in the same parallel class. The interval $J_1\in \mathcal{C}_\gamma(\mu)$ in the figure has divergence radius 2. The other interval $J_2$, intersecting $R_{i_1}$ and $R_{i_2}$, has divergence radius 0.}
    \label{fig:divrad}
\end{figure}
    
The constants $B,C,B',C'$ and $N$ in the above (i) and (ii) are chosen for each given $\mu$. We observe that, to some extent, the same set of constants can be used for any geodesic laminations in a small neighborhood $U$ of $\mu$.

\begin{lemma}\label{lem:divrad}
    Let $\mu$ be a maximal geodesic lamination of $S$ and let $\gamma$ be a geodesic transverse to $\mu$. Let $\theta$ be a number such that every leaf of $\mu$ belongs to $\mathfrak{A}(\gamma, \theta)$. Then, there exist a neighborhood $U\subset \mathcal{GL}(S)$  of $\mu$ and constants $B,C,B',C', N$, and $\theta'$ with $C'<C$ depending on $\gamma$ and $\theta$ such that
        \begin{enumerate}
        \item\label{enum:transverse}  For $\tau\in U$, every leaf of $\tau$ belongs to $\mathfrak{A}(\gamma, \theta')$. In particular, $\gamma$ is transverse to every $\tau\in U$.
        \item\label{enum:finite} For any $n\ge0$, and maximal $\tau\in U$,
        \[
        \sharp\{J\in \mathcal{C}_\gamma(\tau)\mid\dr_\tau(J)=n\}\le N.
        \]
        \item\label{enum:compatible} For any maximal $\tau\in U$, and any $J\in \mathcal{C}_\gamma(\tau)$, we have
        \[
         B\cdot \ell(J_\mathrm{min})\cdot \e^{-C\dr_\tau(J)}\le \ell(J)\le  B' \cdot \ell(J_\mathrm{min})\cdot \e^{-C'\dr_\tau(J)}
        \]
        where $J_\mathrm{min}\in \mathcal{C}_\gamma(\tau)$ is the one that has the smallest divergence radius among all elements that are in the same parallel class as $J$.  
    \end{enumerate}
\end{lemma}
\begin{definition}
    Call a neighborhood $U$ obtained in Lemma~\ref{lem:divrad} an \emph{admissible neighborhood} of $\mu$ with respect to $\gamma$. 
\end{definition}

\begin{proof}[Proof of Lemma~\ref{lem:divrad}] Choose a train-track $\mathcal{T}$ that carries $\mu$. We may further assume that $\gamma$ does not pass through any switch of $\mathcal{T}$ and that $\gamma$ intersects $\mathcal{T}$ along ties. Using Lemma~\ref{lem:traintracknbd}, we find a neighborhood $U$ of $\mu$ where all $\tau\in U$ are carried by $\mathcal{T}$. Suppose that there is a sequence of geodesic laminations $\tau_i$ in $U$ containing a leaf $l_i$ with $\angle_{p_i}(\gamma, l_i)\to 0$ as $i\to \infty$. We may assume that the intersection points $p_i$ are contained in a single branch, say $e$. Lift everything to the universal cover $\mathbb{H}^2$ and fix a smooth embedding $\psi:[0,1]\times[0,1]\to S$ for the branch $e$. Consider $g_i:=\psi^{-1}(l_i)$. Since $\tau_i$ are carried by $\mathcal{T}$, $g_i$ are arcs in $[0,1]\times[0,1]$ transverse to all horizontal lines $[0,1]\times\{t\}$ for $0\le t\le 1$. Since $\gamma$ and $l_i$ are geodesics, the condition $\angle_{p_i}(\gamma, l_i)\to 0$ implies that the sequence $\{g_i\}$ must converge to $\psi^{-1}(\gamma)$ in the Hausdorff topology with respect to any metric on $[0,1]\times [0,1]$. This is impossible because $g_i$ intersects the switch $[0,1]\times\{1\}$ whereas $\psi^{-1}(\gamma)$ does not. We reach a similar contradiction if we assume that there exists a sequence of leaves $l_i$ with $\angle_{p_i}(\gamma, l_i)\to \pi$.  Therefore, for some $\theta'$ depending on $\mathcal{T}$, every $\tau\in U$ has property~(\ref{enum:transverse}). 

Both $\tau$ and $\mu$ are maximal laminations carried by $\mathcal{T}$. The following argument is a refinement of the proof of \cite[Lemma~8.7]{BD}. 

Choose a parametrization $f_\gamma:[0,1]\to \gamma\subset S$ for $\gamma$ and consider its lift $\widetilde{f_\gamma}:[0,1]\to \mathbb{H}^2$. Let $\widetilde{\gamma} :=\widetilde{f_\gamma}([0,1])$ and let $\widetilde{\mathcal{T}}$ be the full lift of $\mathcal{T}$ to $\mathbb{H}^2$. 

For each connected component $C$ of $\widetilde{\gamma} \setminus \widetilde{\mathcal{T}}$, there exists a unique element $J_C\in \mathcal{C}_\gamma(\tau)$ with $\dr_{\tau}(J_C)=0$ such that $C\subset J_C$. This correspondence $C\mapsto J_C$ is a surjective map from the set of connected components of $\widetilde{\gamma} \setminus \widetilde{\mathcal{T}}$ onto \[
\{J\in \mathcal{C}_\gamma(\tau)\mid\dr_\tau(J)=0\}.
\]
It follows that, as long as $\tau$ is carried by $\mathcal{T}$, the quantity
\[
\sharp\{J\in \mathcal{C}_\gamma(\tau)\mid\dr_\tau(J)=0\}
\]
is bounded from above by the number of connected components of $\widetilde{\gamma} \setminus \widetilde{\mathcal{T}}$. 

Let $\mathcal{B}(\mathcal{T})$ be the set of branches of $\mathcal{T}$ and let  
\[
M=\max_{e\in \mathcal{B}(\mathcal{T})} \sharp\{\text{connected components of }\gamma \cap e\}.
\]
Then each element $C^i_j$ of an infinite chain contains at most $M$ elements of $\mathcal{C}_\gamma(\tau)$. Therefore, for any $n>0$ and any $\tau\in U$, we conclude that 
\[
\sharp\{J\in \mathcal{C}_\gamma(\tau)\mid\dr_\tau(J)=n\}
\]
is bounded from above by $M$ times the number of spikes, or $M\cdot (12g-12)$, where $g$ is the genus of $S$. This bound depends only on $\mathcal{T}$ and the topology of $S$.  Therefore, for (\ref{enum:finite}), we choose $N$ to be the maximum of $M\cdot(12g-12)$ and the number of connected components of $\widetilde{\gamma} \setminus \widetilde{\mathcal{T}}$.

Now we prove (\ref{enum:compatible}). First observe that, using the surjection from the set of connected components of $\widetilde{\gamma} \setminus \widetilde{\mathcal{T}}$ onto $\{J\in \mathcal{C}_\gamma(\tau)\mid\dr_\tau(J)=0\}$, one can find  $B_1 $ and $B_1'$ depending only on $U$ such that whenever $\tau\in U$ and $J\in \mathcal{C}_{\gamma}(\tau)$ with $\dr_\tau(J)=0$, we have $B_1<\ell(J)<B_1'$. 

If $J\in \mathcal{C}_{\gamma}(\tau)$ satisfies $\dr_\tau(J)=n>0$, we can find an infinite chain $\mathfrak{C}^i=\{C^i_1,C^i_2, \cdots\}$  such that  $J$ is contained in $C^i_{\dr_\tau(J)}$. Then $J$,  $J_{\mathrm{min}}$ and segments of geodesic leaves of a complementary ideal triangle bound a geodesic quadrilateral $Q$ embedded in $S$. Applying a suitable isometry, we may assume that $Q$ lifts to the universal cover as in Figure~\ref{fig:divrad2}. As we showed in (\ref{enum:transverse}), all leaves of $\tau\in U$ belong to $\mathfrak{A}(\gamma, \theta')$. Using hyperbolic geometry and Lemma~\ref{lem:lengthheight} as explained in Figure~\ref{fig:divrad2}, we obtain constants $A$ and $l$ and $L$, $l<L$, depending on $\theta'$, $\mathcal{T}$ and $\gamma$ such that
\[
\frac{\ell(J_\mathrm{min})}{A^2}\e^{-L\cdot \dr_\tau(J)}<\ell(J)<A^2\cdot \ell(J_{\mathrm{min}})\cdot \e^{-l\cdot (\dr_\tau(J)-\dr_\tau(J_{\mathrm{min}})-1)}.
\]
Therefore, we obtain (\ref{enum:compatible}).
\begin{figure}[!hbt]
    \centering
    \includegraphics[width=\linewidth]{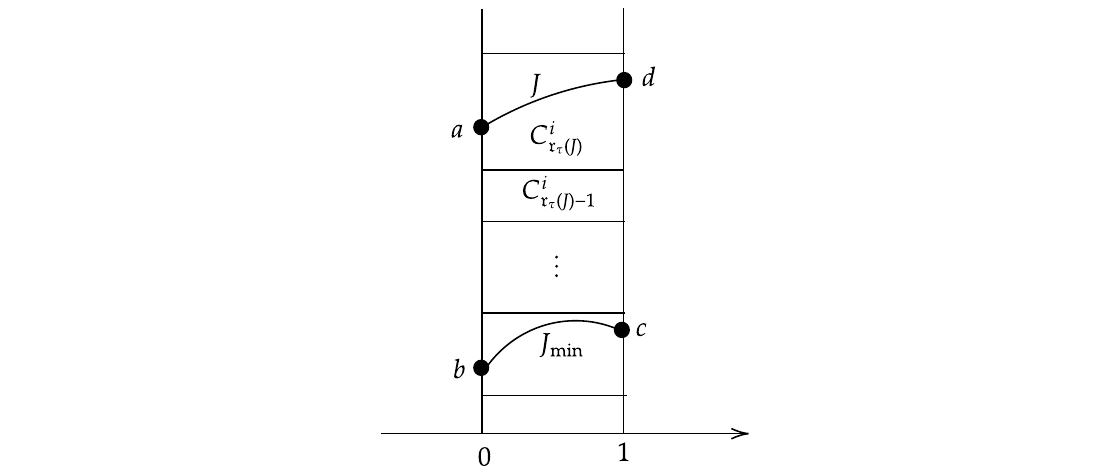}
    \caption{Two vertical lines are geodesic leaves of $\tau$. Let $B(\mathcal{T})$ be the set of branches of $\mathcal{T}$. Let $L=\max_{e\in B(\mathcal{T})} \{H(e)\}$ and $l=\min_{e\in B(\mathcal{T})}\{h(e)\}$, where $H(e)$ and $h(e)$ are the maximum and the minimum hyperbolic distances between two switches of a branch $e$ respectively. The hyperbolic length of the vertical geodesic segment $\overline{ab}$ is given by $\ln\frac{a}{b}$. Since there are  $\dr_\tau(J)-\dr_\tau(J_\mathrm{min})$ branches between $J$ and $J_{\mathrm{min}}$, we obtain  $l\cdot(\dr_\tau(J)-\dr_\tau(J_\mathrm{min})-1) <\ln\frac{a}{b}<L \cdot \dr_\tau(J)$. The same bound holds for $\ln\frac{d}{c}$. By Lemma~\ref{lem:lengthheight}, we have $\frac{\ell(J_\mathrm{min})}{A^2}\e^{-L\cdot \dr_\tau(J)}<\ell(J)<A^2\cdot \ell(J_{\mathrm{min}})\cdot \e^{-l\cdot (\dr_\tau(J)-\dr_\tau(J_{\mathrm{min}})-1)}$ for some constant $A$ depending only on $\theta'$.}
    \label{fig:divrad2}
\end{figure}
\end{proof}

For each fixed $\tau$, the bound (\ref{enum:compatible}) in Lemma~\ref{lem:divrad} can be written as the following form:
\[
         \min_{J\in \mathcal{C}_\gamma(\tau)}\{\ell(J_\mathrm{min})\} \cdot B \e^{-C\dr_\tau(J)}\le \ell(J)\le  \max_{J\in \mathcal{C}_\gamma(\tau)}\{\ell(J_\mathrm{min})\} \cdot B' \e^{-C'\dr_\tau(J)}.
\]
Since there are only finitely many $J_\mathrm{min}$'s, we recover the familiar estimation. 

Since $\ell(J_\mathrm{min})\le \ell(\gamma)$, we  conclude that there exist uniform constants $B''$ and $C'$ depending only on $U$ and $\gamma$ such that
\[
\ell(J) \le B'' \e^{-C'\dr_\tau(J)}
\]
for all $\tau\in U$ and all $J\in \mathcal{C}_\gamma(\tau)$. However, we are not able to rule out the possibility that the quantity $\ell(J_\mathrm{min})$ converges to zero as $\tau$ converges to $\mu$, while $\dr_\tau(J_\mathrm{min})$ remains bounded. For this reason, it may not be possible to find uniform constants $B$ and $C$ for which
\[
B \e^{-C\dr_\tau(J)}\le \ell(J)
\]
holds for all $J\in \mathcal{C}_\gamma(\tau)$ and all $\tau$ near $\mu$.

The divergence radius plays a role in the proof of the following useful lemma. 

\begin{lemma}\label{lem:exponentcase}
Let $\mu$ be a maximal geodesic lamination of $S$ and let $\gamma$ be a transverse oriented geodesic arc. Let $U$ be an admissible neighborhood of $\mu$  with respect to $\gamma$. Let $\nu>0$. Then there exist constants $Q$, and $\nu'$ depending on $U$, $\nu$, and $\gamma$ with the following properties: 

\begin{enumerate}
    \item\label{eq:lengthbound1} For every maximal $\tau\in U$ and every subset $\mathcal{J}$ of $\mathcal{C}_\gamma(\tau)$,    we have
    \[
    \sum_{J\in\mathcal{J}}\ell(J)^\nu \le Q \cdot \left(\sum_{J\in\mathcal{J}} \ell(J)\right)^{\nu'}
    \]
    \item\label{eq:lengthbound2} For every maximal $\tau\in U$ and every ordered subset $\vec{\mathcal{J}}=\{J_i\}_{i\in \mathbb{Z}}$ of $\mathcal{C}_\gamma(\tau)$ with
        \[
    \sum_{L\in \mathcal{C}_\gamma(\tau)\setminus \vec{\mathcal{J}}} \ell(L)\le 1
    \]
    we have
    \[
    \sum_{i} \mathrm{d}(J_i^1, J_{i+1}^0)^\nu \le Q\cdot\left(\sum_{L\in \mathcal{C}_\gamma(\tau)\setminus \mathcal{J}} \ell(L)\right)^{\nu'}.
    \]
\end{enumerate}
\end{lemma}
\begin{proof} By the property (\ref{enum:compatible}) of Lemma~\ref{lem:divrad}, 
\[
\sum_{J\in \mathcal{J}}\ell(J)^\nu < \sum_{J\in \mathcal{J}} \ell(J_\mathrm{min})^\nu (B')^\nu \e^{-C'\nu \dr_\tau(J)}.
\]
Let  $A_1,\cdots, A_{12g-12}$ be  the parallel classes of $\mathcal{C}_\gamma(\tau)$ and let $\mathcal{J}_i=\{J\in \mathcal{J}\mid J\in A_i\}$. Then,
\begin{align*}
\sum_{J\in \mathcal{J}}\ell(J)^\nu &< \sum_{i=1}^{12g-12}\sum_{J\in \mathcal{J}_i} \ell(M_i)^\nu (B')^\nu \e^{-C'\nu \dr_\tau(J)}\\
&\le (B')^\nu\cdot \sum_{i=1}^{12g-12}\sum_{r=\dr(M_i)}^\infty N\cdot \ell(M_i)^\nu  \e^{-C'\nu r}\\
&= \frac{N\cdot (B')^\nu}{1-\e^{-C'\nu}}\cdot\sum_{i=1}^{12g-12}\ell(M_i)^\nu  \e^{-C'\nu\dr(M_i)}
\end{align*}
where $N$ is the constant from (\ref{enum:finite}) of Lemma~\ref{lem:divrad} and where $M_i$ is the element of $A_i$ with the smallest $\dr_\tau(M_i)$. That is,  $\dr_\tau(M_i) = \min_{J\in A_i} \{\dr_\tau(J)\}$. 

On the other hand, for  $\nu'=\frac{C'}{C}\nu$, we have
\begin{align*}
\left(\sum_{J\in \mathcal{J}} \ell(J)\right)^{\nu'}&=\left( \sum_{i=1}^{12g-12} \sum_{J\in \mathcal{J}_i}\ell(J)\right)^{\nu'}\\
&\ge G\cdot \sum_{i=1}^{12g-12}\left(\sum_{J\in \mathcal{J}_i} \ell(J)\right)^{\nu'} \\
&\ge B^{\nu'}G\cdot\sum_{i=1}^{12g-12} \ell(M_i)^{\nu'} \left( \sum_{J\in \mathcal{J}_i} \e^{-C\dr_\tau(J)}\right)^{\nu'}\\
&\ge B^{\nu'}G\cdot\sum_{i=1}^{12g-12} \ell(M_i)^{\nu'}  \e^{-C\nu'\dr_\tau(M_i)}\\
&=B^{\nu'}G\cdot\sum_{i=1}^{12g-12} \ell(M_i)^{\nu'}  \e^{-C'\nu\dr_\tau(M_i)}
\end{align*}
where $G= 1$ if $\nu'\ge  1$ or $G=(12g-12)^{\nu'-1}$ if $0<\nu'<1$.  Here, we used Jensen's inequality and the convexity of the exponential function. Since $C'<C$, we obtain $\nu'<\nu$. Thus, we have  
\[
\ell(M_i)^\nu<\max(1,\ell(\gamma)^\nu)\cdot \ell(M_i)^{\nu'}
\] 
for all $i$. Therefore,
\begin{equation}\label{eq:lengthbound3}
\sum_{J\in \mathcal{J}}\ell(J)^\nu< H\cdot \left(\sum_{J\in \mathcal{J}}\ell(J)\right)^{\nu'},
\end{equation}
where
\[
H:=\frac{(B')^\nu\cdot \max\{1,\ell(\gamma)^\nu\}\cdot N}{(1-\e^{-C\nu})\cdot B^{C'\nu/C}G},\qquad \nu'=\frac{C'}{C}\nu.
\]

Observe that for a sequence of positive numbers $\{a_i\}$, 
\begin{equation}\label{eq:nu>1}
\sum_i a_i^\nu \le \left(\sum_i a_i\right)^\nu
\end{equation}
if $1\le \nu$ and 
\begin{equation}\label{eq:nu<1}
\sum_i a_i^\nu \ge \left(\sum_i a_i\right)^\nu
\end{equation}
if $0<\nu<1$.

Suppose that $1\le \nu $. From (\ref{eq:nu>1}), we have
\[
\sum_{i} \mathrm{d}(J_i^1, J_{i+1}^0)^\nu \le \left(\sum_{i} \mathrm{d}(J_i^1, J_{i+1}^0)\right)^\nu\le \left(\sum_{L\in \mathcal{C}_\gamma(\tau)\setminus \vec{\mathcal{J}}}\ell(L)\right)^\nu.
\]

On the other hand, if $0<\nu<1$, we   note that
\begin{equation}\label{eq:between}
\mathrm{d}(J_i^1, J_{i+1}^0)= \sum_{\substack{L\in \mathcal{C}_\gamma(\tau)\\ J_i^1<L< J_{i+1}^0}}\ell(L).
\end{equation}
Thus, we deduce that
\[
        \sum_{i} \mathrm{d}(J_i^1, J_{i+1}^0)^\nu \le \sum_{L\in \mathcal{C}_\gamma(\tau)\setminus \vec{\mathcal{J}}}\ell(L)^{\nu}<H\cdot\left(\sum_{L\in \mathcal{C}_\gamma(\tau)\setminus \vec{\mathcal{J}}} \ell(L)\right)^{\nu'}
\]
where the first inequality is a consequence of (\ref{eq:nu<1}) and (\ref{eq:between}), and the second inequality follows from (\ref{eq:lengthbound3}). 

Since 
\[
    \sum_{L\in \mathcal{C}_\gamma(\tau)\setminus \vec{\mathcal{J}}} \ell(L)\le 1,
\]
we obtain (\ref{eq:lengthbound1}) and (\ref{eq:lengthbound2}) by taking $Q= \max(1,H)$.
\end{proof}

\subsection{Twisted transverse cocycles}

In this subsection, we define twisted transverse cocycles and their weak convergence. 

In this article, we assume that geodesic arcs are compact unless otherwise stated. Let  $\gamma$ be an oriente geodesic arc. If $\eta$ is a subarc of $\gamma$, we will always assume that $\eta$ is given the oriented compatible to the orientation of $\gamma$.

\begin{definition}
 Let $V$ be a finite dimensional vector space with an involutive automorphism $\iota:V\to V$. A \emph{twisted $V$-valued transverse cocycle} $\mathsf{c}$ consists of an underlying maximal geodesic lamination $|\mathsf{c}|$ and an assignment mapping each oriented simple transverse arc $\gamma$ to an element $\mathsf{c}(\gamma)\in V$ satisfying the following properties:
\begin{itemize}
    \item If $\gamma_1$ and $\gamma_2$ are isotopic as oriented arcs relative to $|\mathsf{c}|$, then $\mathsf{c}(\gamma_1) = \mathsf{c}(\gamma_2)$,
    \item Any point $x$ in the interior of $\gamma\setminus |\mathsf{c}|$ splits $\gamma$ into two subarcs $\gamma_1$ and $\gamma_2$. We request that $\mathsf{c}(\gamma) = \mathsf{c}(\gamma_1)+\mathsf{c}(\gamma_2)$,
    \item $\mathsf{c}(\overline{\gamma}) = \iota \mathsf{c}(\gamma)$, where $\overline{\gamma}$ denotes the oriented geodesic arc with the same underlying geodesic arc as $\gamma$ but with the opposite orientation. 
\end{itemize}
\end{definition}

Let $\HD(S;V)$ be the set of twisted $V$-valued transverse cocycles on $S$. Given a maximal geodesic lamination $\mu$,  $\HD(\mu;V)$ denotes the set of twisted $V$-valued transverse cocycles whose underlying lamination is $\mu$.

Given a train-track $\mathcal{T}$, denote by $B(\mathcal{T})$ the set of branches of $\mathcal{T}$. We define the set of weights
\[
\mathcal{WT}(\mathcal{T};V) =\{ \alpha: B(\mathcal{T})\to V\mid\alpha\text{ satisfies the switch condition}\}.
\]
Here the switch condition means that if $e_1, \cdots, e_k$ are incoming branches and $f_1, \cdots, f_r$ are outgoing branches along a switch, we have $\sum\alpha(e_i)  = \sum \alpha(f_j)$. 

For a given maximal geodesic lamination $\lambda$ carried by a train-track $\mathcal{T}$, we consider the orientation double cover $\lambda^{\mathrm{or}}\to \lambda$ which extends to the double cover of the train-track $p:\mathcal{T}^{\mathrm{or}}\to \mathcal{T}$. We remark that, since $\lambda$ is maximal, $\lambda^\mathrm{or}$ and $\mathcal{T}^\mathrm{or}$ are both connected. Then we can orient the ties and switches of $\mathcal{T}^\mathrm{or}$ consistently. The covering involution $R:\mathcal{T}^\mathrm{or}\to \mathcal{T}^\mathrm{or}$ flips the orientation of $\mathcal{T}^\mathrm{or}$.

Let $\mathsf{c}$ be a twisted $V$-valued transverse cocycle with the underlying lamination $\lambda$. We assign to each branch $e$ of $\mathcal{T}^\mathrm{or}$ the value $\mathsf{c}(p(\gamma_e))$, where $\gamma_e$ is an oriented generic tie of $e$ such that $\gamma_e$ intersects $\lambda^\mathrm{or}$ positively.  This assignment, denoted by $\alpha_\mathsf{c}$, gives rise to an element $\alpha_\mathsf{c}\in \WT(\mathcal{T}^\mathrm{or};V)$ satisfying $\alpha_\mathsf{c} (R (e)) = \iota \alpha_\mathsf{c} (e)$. 

Define
\[
\WT(\mathcal{T};V) = \{\alpha \in \mathcal{WT}(\mathcal{T}^\mathrm{or};V)\mid\alpha(R(e))= \iota(\alpha(e))\}.
\]
Then via the association $\mathsf{c}\mapsto \alpha_\mathsf{c}$,  $\HD(|\mathsf{c}|;V)$ can be identified with a subspace of the finite dimensional vector space $\WT(\mathcal{T};V)$.

We can lift a $V$-valued twisted transverse cocycle $\mathsf{c}$ of $S$ to a $V$-valued transverse cocycle  $\widetilde{\mathsf{c}}$ of the universal cover $\mathbb{H}^2$. For any oriented geodesic arc $\gamma$ in $\mathbb{H}^2$ transverse to the lift of $|\mathsf{c}|$, we take a subdivision $\gamma = \gamma_1\cup \cdots \cup \gamma_l$ such that each projected image $p(\gamma_i)$ is simple in $S$, where $p:\mathbb{H}^2\to S$ is the covering map. Then , the value $\widetilde{\mathsf{c}}(\gamma)$ is obtained by computing $\widetilde{\mathsf{c}}(\gamma) = \sum_i \mathsf{c}(p(\gamma_i))$. One can check that this definition does not depend on the choice of partition of $\gamma$. By abuse of notation, we will write $\mathsf{c}(\gamma)$ instead of $\widetilde{\mathsf{c}}(\gamma)$. 

Using this lift, we may define $\mathsf{c}(\gamma)$ for a non-simple transverse geodesic in $S$: lift $\gamma$ to a simple transverse geodesic arc $\widetilde{\gamma}$ in $\mathbb{H}^2$ and set $\mathsf{c}(\gamma) = \widetilde{c}(\widetilde{\gamma})$. Equivalently, we may split $\gamma$ into simple transverse oriented arcs $\gamma=\gamma_1\cup \gamma_2 \cup \cdots \cup \gamma_l$ and define $\mathsf{c}(\gamma) = \sum \mathsf{c}(\gamma_i)$.

A \emph{measured lamination} $\mu$ of $S$ is an example of $\bR$-valued transverse cocycles. More precisely, it consists of an underlying maximal geodesic lamination, denoted by $|\mu|$, and a locally finite Borel measure $\mu_\gamma$ attached to each simple arc $\gamma$ transverse to $|\mu|$ satisfying homotopy invariance and compatibility in a sense that if $\gamma' \subset \gamma$ is  a subarc then $\mu_{\gamma'}=\mu_\gamma|\gamma'$. By abuse of notation, $\mu(\gamma)$ will often denote the total $\mu_\gamma$-mass $\mu_\gamma(\gamma)$ of the transverse arc $\gamma$. With this convention, $\gamma\mapsto \mu(\gamma)$ is a $\bR$-valued transverse cocycle. Let $\ML(S)$ be the space of measured geodesic laminations of $S$.

Similarly, for a finite dimensional vector space $V$ with involution $\iota:V\to V$,  we can define a \emph{twisted $V$-valued measured lamination} $\mathsf{c}$ to be a pair consisting of an underlying maximal geodesic lamination $|\mathsf{c}|$ and a map $\gamma\mapsto \mathsf{c}_\gamma$ that assigns a countably additive locally finite measure with values in $V$ to each oriented geodesic arc $\gamma$ such that 
\begin{itemize}
    \item $f_*\mathsf{c}_\gamma=\mathsf{c}_\eta$ where $f$ is an isotopy preserving $|\mathsf{c}|$ with $f(\gamma)=\eta$ as oriented arcs,
    \item $\mathsf{c}_\eta = \mathsf{c}_\gamma|\eta$ if $\eta$ is an oriented subarc of $\gamma$ 
    \item $\mathsf{c}_{\overline{\gamma}} = \iota \mathsf{c}_\gamma$ if $\overline{\gamma}$ is a geodesic arc with the same underlying geodesic as $\gamma$ but with the reverse orientation.
\end{itemize} 
One can regard a twisted $V$-valued measured lamination as a twisted $V$-valued cocycle satisfying countable additivity. 

Let $\mathsf{c}$ be a twisted $V$-valued measured lamination and let $\gamma$ be an oriented transverse arc. Since $V$ is finite dimensional, we can choose a basis $\mathbf{e}_1, \cdots, \mathbf{e}_n$ of $V$ and write $\mathsf{c}_\gamma$ as a linear combination $\mathsf{c}_\gamma = \sum_{i=1} ^n\mathsf{c}_i\mathbf{e}_i$ of complex (if $V$ is a complex vector space), or signed (if $V$ is a real vector space) measures $\mathsf{c}_i$. Since $\gamma$ is compact and each $\mathsf{c}_i$ is locally finite, the total variation of $\mathsf{c}_i$ is also finite. Hence, there exists a uniform bound $L_0>0$ such that $|\mathsf{c}_i (I)|<L_0$ for all measurable subset $I\subset \gamma$. It follows that one can find a constant $L>0$ such that $\|\mathsf{c}_\gamma(I)\|<L$ for all measurable $I\subset \gamma$.

By lifting $\gamma$ to the universal cover, we can assign a locally finite Borel measure on a non-simple transverse geodesic $\gamma$ as we did in the transverse cocycle case. More precisely, to each open set $U\subset \gamma$ we assign a real number
\[
\int_U\der\mu_\gamma=\sup \sum_{i} \mu_\gamma(U_i)
\]
where the supremum is taken over all partitions of $U$ into simple subarcs. We sometimes simply write $\der \mu$ and $\mu(U)$ instead of $\der\mu_\gamma$ and $\int_U\der\mu_\gamma$ when their meanings are clear from context. 

We equip $\ML(S)$ with the weak-* topology; a sequence $\{\mu_i\}$ in $\ML(S)$ converges weakly to $\mu$ if for every simple geodesic arc $\gamma$ transverse to $|\mu|$ and any continuous function $f:\gamma \to \bR$,
\[
\int_\gamma f\; \der\mu_i \to \int_\gamma f\;\der\mu
\]
as $i\to \infty$. 

 Let $\mathcal{S}$ be the set of simple closed geodesics in $S$. The set $\mathcal{S}\times\bR_+$ of \emph{weighted simple closed geodesics} can be regarded as a subset of $\mathcal{ML}(S)$. That is, an element $(\gamma, r)\in \mathcal{S}\times \bR_+$ represents a measured geodesic lamination whose support is the single geodesic $\gamma$ and whose transverse measure is the $r$ times the intersection number. It is known that $\mathcal{S}\times\bR_+$ is dense in $\ML(S)$.

We can also define the weak convergence and support for twisted transverse cocycles. 

\begin{definition}A sequence $\{\mathsf{c}_n\}$ of twisted transverse cocycles \emph{converges weakly} to $\mathsf{c}$ if for every oriented simple geodesic arc $\gamma$ transverse to $\supp(\mathsf{c})$, $\gamma$ is transverse to $\supp(\mathsf{c}_n)$ for all sufficiently large $n$ and $\mathsf{c}_n(\gamma) \to \mathsf{c}(\gamma)$ as $n\to \infty$. 
\end{definition}

\begin{definition} The \emph{support} $\supp(\mathsf{c})$ of a twisted $V$-valued transverse cocycle $\mathsf{c}$ is the set of points $x\in S$ such that there exists a nesting sequence of arcs $\gamma_i$ transverse to $|\mathsf{c}|$ satisfying  $\bigcap_i \gamma_i = \{x\}$ and $\mathsf{c}(\gamma_i)\ne 0$ for all $i$.
\end{definition}

\begin{remark}
In the literature, measured laminations are usually assumed to have full support. That is, $\mu(\gamma)\ne 0$ when $\gamma\cap |\mu|\ne \emptyset$. In this paper, we use a slightly different convention. We do not assume that measured laminations are fully supported. Consequently, given a measured lamination $\mu$, the underlying (maximal) geodesic lamination $|\mu|$ is not necessarily identical to its support $\supp(\mu)$.
\end{remark}

Let $\mathsf{c}$ be a twisted $V$-valued cocycle and let $\gamma$ be an oriented arc $\gamma$ transverse to $|\mathsf{c}|$. Unlike measured laminations, even if $\mathsf{c}(\gamma)=0$  there might exist a subarc $\eta\subset \gamma$ with $\mathsf{c}(\eta)\ne 0$. Therefore, the following lemma deserves a proof although it is  clear for measured laminations. 

\begin{lemma}\label{lem:support}
     $\supp(\mathsf{c})$ is a closed subset of $|\mathsf{c}|$. In particular, $\supp(\mathsf{c})$ is a sublamination of $|\mathsf{c}|$.
\end{lemma}
\begin{proof}
Let $x\in |\mathsf{c}|\setminus \supp(\mathsf{c})$ and let $\gamma$ be an oriented arc transverse to $|\mathsf{c}|$ that contains $x$ in its interior. Let 
\[
I:=\overline{\bigcap \{\eta\mid \eta\text{ are subarcs of }\gamma \text{ containing }x\text{ in the interior and }\mathsf{c}(\eta)\ne0\}}.
\]
Since $x\notin\supp(\mathsf{c})$, $I\supsetneq\{x\}$. In particular, the interior $\mathrm{int}(I)$ is non-empty. 

Observe that $\mathsf{c}(I)=0$. To this end, let $\mathsf{c}(I)=v\in V$. We can split $I$ into subarcs $I_-\cup I_0\cup I_+$ with $x\in I_0$. By the construction of $I$, we must have $\mathsf{c}(I_-\cup I_0) = \mathsf{c}(I_0\cup I_+)=0$. By finite additivity, we obtain $\mathsf{c}(I_-) = \mathsf{c}(I_+)= v$. But, since $\mathsf{c}(I) = \mathsf{c}(I_-)+\mathsf{c}(I_0)+\mathsf{c}(I_+)$, we have $\mathsf{c}(I_0) = -v$. By the construction of $I$, we obtain $\mathsf{c}(I_0) = v = 0$. 

Now, we claim that every subarc $\eta'\subset \mathrm{int}(I)$ whose endpoints do not lie on $|\mathsf{c}|$ satisfies $\mathsf{c}(\eta') = 0$. From this claim, we conclude that $\supp(c)\cap \gamma$ is closed.  Indeed, if $\eta'$ contains $x$ in its interior then, by the construction of $I$, we must have $\mathsf{c}(\eta')=0$. Thus, we may assume that $x\notin \mathrm{int}(\eta')$. If $I\setminus \eta'$ consists of two components, we may split $I$ into three subarcs $I_-\cup \eta' \cup I_+$ with $x\in I_-$, say.  Since $\mathsf{c}(I)= \mathsf{c}(I_-) = 0$, we observe that $I_-\cup \eta'$ is a proper subarc of $I$ containing $x$ in its interior with $\mathsf{c}(I_-\cup \eta') \ne 0$, violating the construction of $I$.  Therefore, $I \setminus \eta'$ is connected. Then since $\mathsf{c}(I)=0$, we must have $\mathsf{c}(\eta')=0$.
\end{proof}

The following lemma relates weak limits of transverse cocycles and Hausdorff limits of their underlying geodesic laminations. This result is also a folklore for measured laminations. We refer to \cite[Proposition A.3.1]{otal}.  
\begin{lemma}\label{lem:limit}
    Let $\{\mathsf{c}_n\}$ be a sequence in $\HD(S;V)$ converging to $\mathsf{c}\in \HD(S;V)$ weakly. Then, for every $\delta>0$, there exist $N$ and finitely many maximal geodesic laminations $\tau_1,\cdots, \tau_k$ with the following properties:
    \begin{enumerate}
        \item For each $i\in\{1,2,\cdots,k\}$, $\supp(\mathsf{c})\subset \tau_i$.
        \item For each $n>N$, there exists $i\in \{1,2,\cdots,k\}$ such that $\mathrm{D_H}(\tau_i,|\mathsf{c}_n|)<\delta$ and $|\mathsf{c}_n|$ is contained in an admissible neighborhood of $\tau_i$. 
    \end{enumerate}
\end{lemma}
\begin{proof}
    Let $\mathcal{A}$ be the set of accumulation points of the sequence $\{|\mathsf{c}_n|\}_{n\in \mathbb{N}}$. Since $\mathcal{GL}(S)$ is compact, $\mathcal{A}$ is also compact and $\mathcal{A}\ne\emptyset$. Let $\tau\in \mathcal{A}$. Since $|\mathsf{c}_n|$ are maximal, so is $\tau$. We claim that $\supp(\mathsf{c})\subset \tau$. Indeed, if there were a leaf $l$ of $\supp(\mathsf{c})$ that is not a leaf of $\tau$, there would exist an oriented arc $\gamma$ transverse to  $\supp(\mathsf{c})$ such that $\mathsf{c}(\gamma)\ne 0$, but $\gamma\cap \tau = \emptyset$. Choose a subsequence $n_j$ such that $\lim_{j\to\infty}|\mathsf{c}_{n_j}|=\tau$. Since $\gamma\cap \tau=\emptyset$, we have $\mathsf{c}_{n_j}(\gamma)\to 0$ as $j\to  \infty$. On the other hand, we had $\mathsf{c}(\gamma) \ne 0$,  a contradiction. Therefore, every accumulation point of $\{|\mathsf{c}_n|\}$ contains $\supp(\mathsf{c})$.

 Now for each $\tau\in \mathcal{A}$, choose an open ball $B_\tau$ of radius $r_\tau$ centered at $\tau$. We assume that $r_\tau$ is small enough such that $r_\tau<\delta/2$ and that $B_\tau$ is contained in an admissible neighborhood of $\tau$.  Since $\mathcal{A}$ is compact,  we can find a finite collection $B_{\tau_1},\cdots, B_{\tau_k}$ such that $\mathcal{A}\subset  \bigcup_{i=1}^k B_{\tau_i}$. Let $U:=\bigcup_{i=1}^k B_{\tau_i}$ and let $N:=\max\{n\mid|\mathsf{c}_n|\notin U\}$. This is a finite number since  $\{|\mathsf{c}_n|\}\setminus U$ is a finite set. By construction, whenever $n>N$, $|\mathsf{c}_n|$ is $\delta$-close to some $\tau_i$ , $i\in \{1,2,\cdots, k\}$ and is in the admissible neighborhood of $\tau_i$. As observed above, $\tau_i$ contains $\supp(\mathsf{c})$. 
\end{proof}

Since $\supp(\mathsf{c})\subset \tau_i$ for all $i$, we may regard $\mathsf{c}$ as a twisted transverse cocycle with underlying lamination $\tau_i$ for all $i$. Let $\mathcal{T}_i$ be a train-track that carries $\tau_i$. Lemma~\ref{lem:traintracknbd} and Lemma~\ref{lem:limit} together show that there exists $N$ such that $\mathsf{c}_n$ is carried by $\mathcal{T}_i$ for some $i$ whenever $n>N$. Hence, we may regard $\mathsf{c}_n$ as a member of $\WT(\mathcal{T}_i;V)$ for some $i$ if $n>N$. Consequently, we can compare $\mathsf{c}$ and $\mathsf{c}_n$ in the finite dimensional vector space $\WT(\mathcal{T}_i;V)$.

\begin{lemma}\label{lem:transverse}
    Let $\{\mathsf{c}_n\}$ be a sequence in $\HD(S;V)$ converging weakly to $\mathsf{c}$. Then for any subsequence $\{\mathsf{c}_{n_i}\}$ of $\{\mathsf{c}_n\}$ such that $\lim_{i\to \infty} |\mathsf{c}_{n_i}|=\tau_m$, the sequence $\{\mathsf{c}_{n_i}\}$ converges to $\mathsf{c}$ in $\WT(\mathcal{T}_m;V)$.

\end{lemma}
\begin{proof}
We have to show that $\alpha_{\mathsf{c}_{n_i}}(e)\to \alpha_\mathsf{c}(e)$ as $i\to\infty$ for every branch $e$ of $\mathcal{T}_m^\mathrm{or}$. Since $\mathsf{c}_{n_i}$ converges to $\mathsf{c}$ weakly, we have that $\mathsf{c}_{n_i}(\gamma_e) \to \mathsf{c}(\gamma_e)$ where $\gamma_e$ is the projection of an oriented generic tie in the branch $e$. Therefore, by definition, $\lim_{i\to\infty}\alpha_{\mathsf{c}_{n_i}}(e)= \alpha_\mathsf{c}(e)$ follows.
\end{proof}

\subsection{Geodesic currents}

Theorem~\ref{thm:poisson} requires interpreting measured laminations as geodesic currents in the sense of  Bonahon \cite{bonahoncurrent}. We briefly review the basic theory of geodesic currents. Denote by $\mathcal{G}$ the set of unoriented geodesics in $\mathbb{H}^2$. If the surface $S$ is given a hyperbolic structure, the fundamental group $\pi_1(S)$ acts  on $\mathcal{G}$. A \emph{geodesic current} is a Borel measure on $\mathcal{G}$ that is invariant under the $\pi_1(S)$-action. The set of geodesic currents $\mathcal{C}(S)$ is equipped with the weak* topology. That is, $\mu_i \to \mu$ if and only if $\int_{\mathcal{G}} f\;\der \mu_i \to \int_{\mathcal{G}} f \; \der \mu$ for all compactly supported continuous functions $f$. 

Any weighted simple closed curve $(\gamma,t)\in \mathcal{S}\times \mathbb{R}_+$ can be interpreted as $t$ times the Dirac delta measure $t\cdot \dirac_\gamma$ supported on the full lift of $\gamma$. It is also known that $\ML(S)$ can be continuously embedded into $\mathcal{C}(S)$. Every $\mu\in \ML(S)$, regarded as a member of  $\mathcal{C}(S)$, satisfies 
\[
\int_{\mathcal{G}^\times(S)} \der(\mu \times \mu) = 0,
\]
where $\mathcal{G}^\times$ is the set of pairs of intersecting geodesics in $\mathbb{H}^2$ and $\mathcal{G}^\times (S) = \mathcal{G}^\times/\pi_1(S)$.

\begin{lemma}\label{lem:continuity}
    Let $\{\mu_n\}$ be a sequence in $\ML(S)\subset \mathcal{C}(S)$ converging weakly to $\mu$. Let $\{g_n: V \to\bR\}_{n\in \mathbb{N}}$ be a sequence of continuous functions defined on an open set $V\subset \mathcal{G}$. Suppose that there exists a compact set $K$ containing  $V\cap \supp(\mu)$ and $V \cap \supp(\mu_n)$ for all $n$. If the sequence $\{g_n\}$ converges to a continuous function $g:V \to \bR$ uniformly on $K$, then for any $\epsilon>0$, there exists $N$ such that
    \[
    \left|\int_{V} g_n \;\der\mu_n - \int_{V} g \;\der\mu\right|<\epsilon
    \]
    for all  $n>N$.
\end{lemma}

\begin{proof} Choose a compactly supported continuous function $s:\mathcal{G}\to \mathbb{R}$ such that $s =1$ on $K$ and $s=0$ on $\mathcal{G}\setminus V$. Then, $s\cdot g$ and $s\cdot g_n$ are compactly supported continuous functions on $\mathcal{G}$, and
\[
\int_V g\;\der\mu = \int_{\mathcal{G}} s\cdot g\;\der\mu,\quad\text{ and }\quad \int_V g_n \;\der\mu_n=\int_{\mathcal{G}} s\cdot g_n\;\der\mu_n
\]
for all $n$. Choose $N$ so that whenever $n>N$, $|\mu_n(K)-\mu(K)|<1$, 
\[
\left| \int_{V}  g \;\der\mu_n-\int_{V} g  \;\der \mu \right|=\left| \int_{\mathcal{G}}  s\cdot g \;\der\mu_n-\int_{\mathcal{G}} s\cdot g  \;\der \mu \right|<\frac{\epsilon}{2+\mu(K)}
\]
and 
\[
\sup_{x\in K}|g_n(x) -g(x)|<\frac{\epsilon}{2+\mu(K)}
\]
hold. Then, whenever $n>N$, we obtain
\begin{align*}
\left| \int_{V} g_n \;\der \mu_n -\int_{V}  g \;\der \mu \right| &\le \left| \int_{V}  g \;\der\mu_n-\int_{V} g  \;\der \mu \right|+\int_{V} |g_n - g|\;\der \mu_n\\
&< \frac{\epsilon}{2+\mu(K)} + \int_{V} |g_n - g| \;\der \mu_n\\
&< \frac{\epsilon}{2+\mu(K)}+ \frac{\epsilon}{2+\mu(K)} \cdot \mu_n (K)\\
&< \frac{\epsilon}{2+\mu(K)}+ \frac{\epsilon}{2+\mu(K)} \cdot( 1+\mu(K))\\
&= \epsilon.\qedhere
\end{align*}
\end{proof}

\section{Cataclysm deformations}\label{sec:cataclysm}
\subsection{General construction}

We begin with a slightly generalized definition of cataclysm deformations. While cataclysm deformations are typically defined for $\Theta$-Anosov representations into a real semisimple Lie group, it is more convenient to work with complex Lie groups for our purposes. The construction in \cite{pfeil} remains largely applicable with only minor modifications.

Let $\mathsf{G}$ be a semisimple connected complex Lie group with Lie algebra $\mathfrak{g}$. As in Section~\ref{sec:lie}, let $\mathfrak{a}=\mathfrak{a}^\bR$ be a maximal abelian subalgebra of $\mathfrak{p}^\bR$ and let $\Sigma$ be the set of restricted roots. The Weyl group $W(\Sigma)$ has the longest element $w_0$ acting  on $\mathfrak{a}$ as an involution. Since the Cartan subalgebra $\mathfrak{h}$ of $\mathfrak{g}$ can be written as $\mathfrak{a}\oplus \ui \mathfrak{a}$, $w_0$ also acts naturally on $\mathfrak{h}$.  Define the \emph{opposite involution} $\iota:\mathfrak{h}\to\mathfrak{h}$ by the map $X\mapsto -w_0\cdot X$. 

Let $\Theta$ be a subset of simple roots such that $\iota(\Theta) = \Theta$. Let 
\[
\mathfrak{a}_\Theta := \bigcap _{\theta\in \Delta\setminus \Theta} \ker \theta.
\]
and
\[
\mathfrak{h}_\Theta := \mathfrak{a}_\Theta  \oplus \ui \mathfrak{a}_\Theta.
\]
Note that $\iota(\mathfrak{h}_\Theta) = \mathfrak{h}_{\iota(\Theta)}=\mathfrak{h}_\Theta$. Therefore, it makes sense to consider $\HD(S;\mathfrak{h}_\Theta)$. From now on we fix a non-zero element $\mathsf{c}\in \HD(S;\mathfrak{h}_\Theta)$.

Let $\mu$ be the lift of the underlying geodesic lamination $|\mathsf{c}|$. Given an oriented compact geodesic arc $\gamma$ in $\mathbb{H}^2$ transverse to a geodesic lamination $\mu$ in $\mathbb{H}^2$, $\mathcal{C}_\gamma(\mu)$ denotes the set of connected components of $\gamma\setminus \mu$. As before, we sometimes assume that $\gamma$ is given a total ordering $<$ compatible with its orientation. Let $x_0$ and $x_1$ be the endpoints of $\gamma$ with $x_0<x_1$. We can always assume that $x_0$ and $x_1$ do not lie on leaves of $\mu$. The set $\mathcal{C}_\gamma(\mu)$ has two special elements $E$ and $F$ containing $x_0$ and $x_1$, respectively. Again, for $I\in \mathcal{C}_\gamma(\mu)$, the endpoints of $I$ are denoted by $I^0$ and $I^1$ with $I^0<I^1$.

We orient each leaf $g$ of $\mu$ so that $\gamma$ intersects $g$ positively. 
Let $g^+$ and $g^-$ be the forward and backward endpoints of $g$. Since a tangent vector determines a unique geodesic, the notation $\vec{\mu}_p$, $p\in \mu\cap \gamma$ is used to denote the (oriented) geodesic leaf of $\mu$ passing through $p$.  Given $I\in \mathcal{C}_\gamma(\mu)$, there are geodesic leaves of $\mu$ passing through $I^0$ and $I^1$. As we defined above, we may denote by $\vec{\mu}_{I^0}$ and $\vec{\mu}_{I^1}$ the oriented geodesic leaves of $\mu$ passing through $I^0$ and $I^1$. 

Recall that a pair of flags $ \mathcal{F}_\Theta\times \mathcal{F}_\Theta^\mathrm{op}$ is transverse if it is in the $\Ad_\mathsf{G}$-orbit of $(\mathfrak{p}_\Theta, \mathfrak{p}_\Theta ^\mathrm{op})$. Let $\mathfrak{T}\subset\mathcal{F}_\Theta\times \mathcal{F}_\Theta^\mathrm{op}$ be the set of transverse flags.  For each $(\mathfrak{x},\mathfrak{y})\in \mathfrak{T}$, there exists an element $\mathsf{g}(\mathfrak{x}, \mathfrak{y})\in \mathsf{G}$ such that  
\begin{equation}\label{eq:sligher}
\Ad_{\mathsf{g}(\mathfrak{x},\mathfrak{y})}(\mathfrak{p}_\Theta, \mathfrak{p}_\Theta^\mathrm{op}):=(\Ad_{\mathsf{g}(\mathfrak{x},\mathfrak{y})}\mathfrak{p}_\Theta,\Ad_{\mathsf{g}(\mathfrak{x},\mathfrak{y})}\mathfrak{p}_\Theta^\mathrm{op}) =(\mathfrak{x},\mathfrak{y})
\end{equation}
Note that $\mathsf{g}(\mathfrak{x},\mathfrak{y})$ is not uniquely determined. If we find another element $\mathsf{g}'$ with $\Ad_{\mathsf{g}'}(\mathfrak{p}_\Theta,\mathfrak{p}_\Theta^\mathrm{op}) =(\mathfrak{x},\mathfrak{y})$, we have $\mathsf{g}' = \mathsf{g}(\mathfrak{x},\mathfrak{y})\cdot \mathsf{h}$ for some $\mathsf{h}\in \mathsf{P}_\Theta \cap \mathsf{P}_\Theta^\mathrm{op}$. However, we will see in Lemma~\ref{lem:welldef} that this ambiguity does not affect our discussion.

For an element $I\in \mathcal{C}_\gamma(\mu)$, we define 
\[
m_{\vec{\mu}_{I^0}}= m_{\vec{\mu}_{I^0}}(\rho):=\mathsf{g}(\xi_\rho(\vec{\mu}_{I^0}^+),\xi^\mathrm{op}_\rho(\vec{\mu}_{I^1}^-)).
\]
Similarly, we define $m_{\vec{\mu}_{I^1}}$.

By convention, for $I\in \mathcal{C}_\gamma(\mu)$, write 
\[
\mathsf{c}(I):=\mathsf{c}([E,I])=\mathsf{c}([x_0,I]) := \mathsf{c}([x_0, x_I])\in \mathfrak{h}_\Theta, 
\]
where $[x_0,x_I]$ is the geodesic subsegment of $\gamma$ oriented from $x_0$ to any point $x_I\in I$.  Choosing a different point in the same interval does not affect the value since the cocycle is supported on leaves of the underlying lamination.

Let $X\in \mathfrak{h}_\Theta$ and let $\rho$ be a  $\Theta$-Anosov representation  with the limit map $\xi_\rho$. Define 
\[
\mathrm{T}^{X}_{I^0} (\rho) := m_{\vec{\mu}_{I^0}}(\rho)\exp (X)m_{\vec{\mu}_{I^0}}(\rho)^{-1}
\]
We define $\mathrm{T}^{X}_{I^1}$ similarly. 

As we mentioned above, the element $m_{\vec{\mu}_{I^0}}=\mathsf{g}(\xi_\rho(\vec{\mu}_{I^0}^+),\xi^\mathrm{op}_\rho(\vec{\mu}_{I^1}^-))$ is determined up to  multiplication of $\mathsf{P}_\Theta \cap \mathsf{P}_\Theta^\mathrm{op}$. Nevertheless we observe the following.
\begin{lemma}\label{lem:welldef}
    $\mathrm{T}^{X}_{I^0} (\rho) $ is well-defined in the sense that the choice of $m_{\vec{\mu}_{I^0}}(\rho)$ does not affect the result.
\end{lemma}
\begin{proof}
We claim that every element $g\in \mathsf{P}_\Theta\cap \mathsf{P}_\Theta^\mathrm{op}$ centralizes $\mathfrak{h}_\Theta$. Indeed, every element $g\in \mathsf{P}_\Theta\cap \mathsf{P}_\Theta^\mathrm{op}$ normalizes the identity component $(\mathsf{P}_\Theta\cap \mathsf{P}_\Theta^\mathrm{op})^0$ and therefore, it normalizes the center $Z((\mathsf{P}_\Theta\cap \mathsf{P}_\Theta^\mathrm{op})^0)$ of the identity component. It follows that $\Ad_g(\mathfrak{h}_\Theta) = \mathfrak{h}_\Theta$.  Since $g$ normalizes both $\mathfrak{p}_\Theta$ and $\mathfrak{p}_\Theta^\mathrm{op}$, we have $\Ad_{g}|\mathfrak{g}_{\pm\theta}=\mathrm{Id}$ for all $\theta\in \Sigma^+\setminus \mathrm{Span}(\Delta\setminus\Theta)$. Choose any $W\in \mathfrak{h}$, and  $H\in \mathfrak{h}_\Theta$. Write $W=\sum_{\beta\in\Delta} c_\beta \beta^\vee$, for  $c_\beta\in \mathbb{C}$.  As we observed above, we have that $\Ad_{g^{-1}} \beta^\vee =\beta^\vee$ for $\beta\in \Theta$. If $\beta\in\Delta\setminus\Theta$, we can write
\[
\Ad_{g^{-1}}\beta^\vee=\sum_{\alpha\in \Delta\setminus \Theta}d_\alpha\alpha^\vee+ \sum_{\alpha\in \mathrm{Span}(\Delta\setminus \Theta)}E_\alpha
\]
for some $E_\alpha\in \mathfrak{g}_{\alpha}$ and $d_\alpha\in \mathbb{C}$. Hence, 
\[
\Ad_{g^{-1}}\left(\sum_{\beta\in \Delta\setminus\Theta}c_\beta\beta^\vee\right) = \sum_{\alpha\in \Delta\setminus \Theta}d'_\alpha\alpha^\vee+ \sum_{\alpha\in \mathrm{Span}(\Delta\setminus \Theta)}E_\alpha'
\]
for $d'_\alpha\in \mathbb{C}$ and $E_\alpha'\in \mathfrak{g}_\alpha$. Note that
$\kf(H, \alpha^\vee)=\alpha(H)= 0$ for any $\alpha\in \Delta\setminus \Theta$. Since $\kf(\mathfrak{h}, \mathfrak{g}_\alpha) = 0$ for all roots $\alpha$, we have
\begin{align*}
\kf(\Ad_g H , W)&= \kf(H, \Ad_{g^{-1}} W)\\
&=\kf \left(H, \sum_{\beta\in \Theta}c_\beta \beta^\vee+\sum_{\alpha\in \Delta\setminus \Theta}d'_\alpha\alpha^\vee+\sum_{\alpha\in \mathrm{Span}(\Delta\setminus \Theta)}E'_\alpha\right)\\
&=\kf \left(H, \sum_{\beta\in \Theta}c_\beta \beta^\vee\right)\\
&=\kf(H, W).
\end{align*}
Since the Killing form $\kf$ is nondegenerate on $\mathfrak{h}$, and since $W$ is arbitrary, we conclude $H = \Ad_g H$ for any $H\in \mathfrak{h}_\Theta$. 
\end{proof}

For simplicity, write
\begin{align*}
\mathrm{T}^{\pm \mathsf{c}}_{I^0} (\rho)&=\mathrm{T}^{\pm \mathsf{c}(I)}_{I^0} (\rho) = m_{\vec{\mu}_{I^0}}(\rho)\exp (\pm\mathsf{c}(I))m_{\vec{\mu}_{I^0}}(\rho)^{-1}\\
\mathrm{T}^{\pm \mathsf{c}}_{I^1} (\rho)&=\mathrm{T}^{\pm \mathsf{c}(I)}_{I^1} (\rho)=m_{\vec{\mu}_{I^1}}(\rho)\exp (\pm\mathsf{c}(I))m_{\vec{\mu}_{I^1}}(\rho)^{-1}.
\end{align*}

The shearing map is given by
\[
\varphi^{\mathsf{c}} _{\gamma}(\rho):=\lim_{\mathcal{C}\to \mathcal{C}_k(\mu)}\left(\vec{\prod_{J\in \mathcal{C}}}\mathrm{T}_{J^0}^{\mathsf{c}}(\rho)\mathrm{T}_{J ^1}^{-\mathsf{c}}(\rho) \right)\mathrm{T}_{F^0}^{\mathsf{c}}(\rho),
\]
where the limit is taken over the directed system of subsets of $\mathcal{C}_\gamma(\mu)$. The arrow over the product notation indicates that it is an ``ordered product.'' More precisely,  given  $\mathcal{J}\subset \mathcal{C}_\gamma(\mu)$, we enumerate the elements of $\mathcal{J}$ in such a way that $J_i<J_{i+1}$ for all $i$.  Then 
\[
\vec{\prod_{J\in \mathcal{J}}} \mathrm{T}^{\mathsf{c}}_{J^0} \mathrm{T}^{-\mathsf{c}}_{J^1}=\left(\mathrm{T}^{\mathsf{c}}_{J_1^0} \mathrm{T}^{-\mathsf{c}}_{J_1^1}\right)\cdot \left(\mathrm{T}^{\mathsf{c}}_{J_2^0} \mathrm{T}^{-\mathsf{c}}_{J_2^1}\right)\cdots \left(\mathrm{T}^{\mathsf{c}}_{J_k^0} \mathrm{T}^{-\mathsf{c}}_{J_k^1}\right).
\]

The key ingredient of the construction is that the above limit eventually converges for all small enough cocycles $\mathsf{c}$. We state this in the following modified form. The proof is essentially a combination of the proof of \cite[Proposition 5.1]{pfeil} (see also \cite{BD,bonahon96}), our observation Lemma~\ref{lem:divrad} and the fact that the limit map depends continuously on representations \cite{BCLS2015}.

\begin{lemma}\label{lem:domain}
    Let $\rho\in \Hom_\Theta(\pi_1(S), \mathsf{G})$ and let $\mathsf{c}\in \HD(S;\mathfrak{h}_\Theta)$. Let $\gamma$ be a compact geodesic arc in $\mathbb{H}^2$ transverse to the lift of $|\mathsf{c}|$. Choose a train-track $\mathcal{T}$ carrying $|\mathsf{c}|$ and an admissible neighborhood $U$ of $|\mathsf{c}|$. Then, there exist constant $B$ and an open neighborhood $W$ of $\rho $ such that $\varphi^{\mathsf{c}'} _\gamma (\rho')$ exists for any $\rho'\in W$ and any $\mathsf{c}'$ with $|\mathsf{c}'|\in U$ and $\|\mathsf{c}'\|<B$. Here, the norm $\| \cdot \|$ is measured in $\WT(\mathcal{T};\mathfrak{h}_\Theta)$. 
\end{lemma}

The \emph{cataclysm deformation} of $\rho$ with respect to $\mathsf{c}\in \HD(S;\mathfrak{h}_\Theta)$ is a new representation $\Lambda^\mathsf{c}_{x_0}\rho$ defined by
\[
(\Lambda^{\mathsf{c}}_{x_0}\rho)(\gamma) := \varphi_{\widetilde{\gamma}}^{\mathsf{c}}(\rho) \rho(\gamma),
\]
where $x_0\in \mathbb{H}^2\setminus \mu$ is a reference point and $\widetilde{\gamma}$ is an oriented geodesic segment joining $x_0$ and $\gamma x_0$. One can check that a different choice of $x_0$ results in a conjugation of the cataclysm deformation. We therefore have the map
\[
\Lambda^{\mathsf{c}} : \Hom_\Theta(\pi_1(S),\mathsf{G})/\mathsf{G} \to \Hom_\Theta(\pi_1(S),\mathsf{G})/\mathsf{G}
\]
for a sufficiently small twisted $\mathfrak{h}_\Theta$-valued transverse cocycle $\mathsf{c}$.

\subsection{Convergence of cataclysm deformations}

In this subsection, we show that cataclysm deformations behave continuously as we vary transverse cocycles. This result is closely related to \cite[Theorem~II.3.11.5]{EM06} as well as the original definition of earthquake deformation by Kerckhoff \cite{Kerckhoff}. As a corollary, we obtain that cataclysm deformations are analytic on Hitchin components. Note that analyticity is known for earthquake deformations \cite{kerckhoff_analytic}.

Before we start our main proof, we give a preferred  metric on $\mathsf{G}$ as follow. We fix a norm on the Lie algebra $\mathfrak{g}$. This will give rise to a left invariant metric $\dG$ on $\mathsf{G}$ by setting $\dG(g,h)= \inf_{\alpha} \int \|\alpha'(t)\|\;\der t$ over all smooth paths $\alpha$ joining $g$ and $h$. We observe that $\dG$ satisfies ``almost $\Ad$-invariance'' in the sense that
\[
\dG(gk,hk) \le \|\Ad_{k^{-1}}\|\cdot \dG(g,h),
\]
where $\|\Ad_{k^{-1}}\|_{\mathrm{op}}$ is the operator norm of the adjoint action with respect to the norm on $\mathfrak{g}$. By abusing notation, we let $\|g\|_{\mathrm{op}}:= \|\Ad_g\|_{\mathrm{op}}$ for $g\in \mathsf{G}$.

Straightforward computations show that
\begin{align}
    \dG(1,ghg^{-1})&\le \|g\|_{\mathrm{op}}\cdot \dG(1,h) \label{eq:distanceconj}\\
    \dG(1,ghg^{-1}h^{-1})&\le \|h\|_{\mathrm{op}}\cdot(1+\|g\|_{\mathrm{op}})\cdot \dG(1,h)\label{eq:distancecommutator}\\
    \dG(1,g_1g_2 \cdots g_l)&\le\sum_{i=1}^l \dG(1,g_i)\label{eq:product}\\
    \dG(g_1g_2g_3,g_1g_3)&\le \|g_3^{-1}\|_{\mathrm{op}}\cdot \dG(1,g_2)\label{eq:deletemiddle}.
\end{align}

\begin{lemma}\label{lem:TTest}
Let $K$ be a compact subset in $\Hom_\Theta(\pi_1(S),\mathsf{G})$ and let $0<\theta<\frac{\pi}{2}$.  Let $\gamma$ be a compact geodesic arc in $\mathbb{H}^2$. There exist constants $A,B>0$ and $\nu>0$ depending on $K$, $\theta$, and $\gamma$ such that for any geodesics $g$ and $h$ in $\mathfrak{A}(\gamma, \theta)$ which transversely intersect $\gamma$ at $x,y$, any $X,Y\in \mathfrak{h}_\Theta$, and any $\rho\in K$ we have
\[
     \dG(\mathrm{T}^X_g(\rho), \mathrm{T}^Y_h(\rho))<A \cdot\left(\left(1+B\cdot \max_{\beta\in \Sigma }(|\e^{-\beta(X)}|)\right) \dPT(\widehat{g}_x,\widehat{h}_y)^\nu +\|X-Y\|\right).
     \]
\end{lemma}
\begin{proof}
    We cannot use the slithering map construction \cite{BD,pfeil} directly since $g$ and $h$ may not be leaves of a single geodesic lamination. We rather adopt their proofs. 

    Orient $g$ and $h$ so that $\gamma$ intersects $g$ and $h$ positively. Draw a geodesic $l$ joining the forward endpoint $g^+$ of $g$ and the backward endpoint, $h^-$ of $h$. For a given $\rho\in K$, there is a unique unipotent $U_1^\rho$ fixing $\xi_\rho(g^+)$ that sends $\xi^\mathrm{op}_\rho(g^-)$ to $\xi^\mathrm{op}_\rho(h^-)$. By \cite[Theorem~6.1(3)]{BCLS2015}, there exist constants $C_1>0$ and $\nu>0$ depending on $K$ such that $\dG(1,U_1^\rho)<C_1\cdot \mathrm{d}_\infty(g^-,h^-)^\nu$ where $\der_\infty$ is a metric on the ideal boundary. We similarly construct an unipotent $U_2^\rho$ fixing $\xi^\mathrm{op}_\rho(h^-)$ that maps $\xi_\rho(g^+)$ to $\xi_\rho(h^+)$. We have $\dG(1,U_2^\rho)<C_1\cdot \mathrm{d}_\infty(g^+,h^+)^\nu$. The constants $C_1>0$ and $\nu>0$ depend on the set $K$ and do not depend on $g$, $h$ or $\rho$.  Let $U_{gh}^\rho = U_2^\rho U_1^\rho$.

    It follows that for all $\rho\in K$,
    \begin{align*}
    \dG(1,U_{gh}^\rho)&\le C_1\cdot(\mathrm{d}_\infty(g^+,h^+)^\nu+\mathrm{d}_\infty(g^-,h^-)^\nu) \\
    &\le C_2 \cdot (\dUT(\vec{g}_x,\vec{h}_y)^\nu +\dUT(\vec{g'}_x,\vec{h'}_y)^\nu).
    \end{align*}
    Here, $g'$ and $h'$ are geodesics with reversed orientation. The constants $C_2$ and $\nu$ depend on $K$, $\theta$ and $\gamma$. In particular, for all $\rho\in K$, and all $g$ and $h$ in $\mathfrak{A}(\gamma,\theta)$, 
    \[
    \|U_{gh}^\rho\|_{\mathrm{op}}<C_3
    \]
    for some constant $C_3$ depending on $K$, $\theta$ and $\gamma$. 
    
    Moreover, if $g$ and $h$ belong to $\mathfrak{A}(\gamma,\theta)$, one can find a constant $C_4>0$ depending on $\gamma$ and $\theta$ such that
    \begin{equation}\label{eq:tangentvsprojective}
    \dUT(\vec{g}_x,\vec{h}_y)+\dUT(\vec{g'}_x,\vec{h'}_y)<C_4\cdot \dPT(\widehat{g}_x,\widehat{h}_y).
    \end{equation}
    It follows that
    \begin{align*}
    \dUT(\vec{g}_x,\vec{h}_y)^\nu +\dUT(\vec{g'}_x,\vec{h'}_y)^\nu & < 2\cdot \left( \dUT(\vec{g}_x,\vec{h}_y) +\dUT(\vec{g'}_x,\vec{h'}_y)\right)^\nu \\
    &< 2 C_4^\nu \cdot \dPT(\widehat{g}_x,\widehat{h}_y)^\nu.
    \end{align*}

    Further, we estimate
    \begin{align*}
    \|\mathrm{T}^X_g\|_{\mathrm{op}}&\le\max(\|U_{g_0g}^\rho\|_{\mathrm{op}},\|(U_{g_0g}^\rho)^{-1}\|_{\mathrm{op}}) \cdot \|\Ad_{\exp X}\|_{\mathrm{op}}\\
    &<C_5\ \cdot\max_{\beta\in \Sigma}(|\e^{\beta(X)}|),
    \end{align*}
    where $g_0$ is an oriented geodesic satisfying $\xi_\rho(g_0^+)= \mathfrak{p}_\Theta$, $\xi^\mathrm{op}_\rho (g_0^-) = \mathfrak{p}^{\mathrm{op}}_\Theta$.  The constant $C_5$ depends on $K$, $\theta$, and $\gamma$. 

    We have
    \[
     \dG(\mathrm{T}^X_g, \mathrm{T}^Y_h)\le \dG(\mathrm{T}^X_g, \mathrm{T}^X_h)+\dG(\mathrm{T}^X_h,\mathrm{T}^Y_h).
     \]
     By (\ref{eq:distancecommutator}), 
     \begin{align*}
         \dG(\mathrm{T}_g^X, \mathrm{T}_h^X)&= \dG(1,\mathrm{T}_g^{-X}\mathrm{T}_h^{X})\\
         &=\dG(1,\mathrm{T}_g^{-X}U_{gh}^\rho \mathrm{T}_g^{X} (U_{gh}^\rho)^{-1})\\
         &<\|U_{gh}^\rho\|_{\mathrm{op}}\cdot (1+\|\mathrm{T}_g^{-X}\|_{\mathrm{op}}) \cdot \dG(1,U_{gh}^\rho)\\
         &<2C_3C_2C_4^\nu\cdot \dPT(\widehat{g}_x,\widehat{h}_y)^\nu\cdot \left(1+C_5\cdot\max_{\beta\in \Sigma}(|\e^{-\beta(X)}|)\right).
     \end{align*}

     We also have, by (\ref{eq:distanceconj}),
     \[
     \dG(\mathrm{T}_h^{X},\mathrm{T}_h^Y)< \|U_{g_0h}^\rho\|_{\mathrm{op}}\cdot\dG(\exp X,\exp Y)<C_3\cdot \|X-Y\|,
     \]
     where the constant $C_3$ depends on $K$, $\theta$ and $\gamma$. 
     Putting all together, we obtain
     \[
     \dG(\mathrm{T}_g^X,\mathrm{T}_h^Y)<A \cdot\left( \dPT(\widehat{g}_x,\widehat{h}_y)^\nu\cdot \left(1+B\cdot\max_{\beta\in \Sigma}(|\e^{-\beta(X)}|)\right) +\|X-Y\|\right),
     \]
     where $A=\max(2C_2C_3 C_4^\nu,C_3)$ and $B=C_5$ are constants depending on $K$, $\theta$, and  $\gamma$. 
\end{proof}

A similar estimate holds when geodesics $g$ and $h$ are parallel.  

\begin{lemma}\label{lem:TTestparallel}
Let $K$ be a compact subset in $\Hom_\Theta(\pi_1(S),\mathsf{G})$. Let $\gamma$ be a compact geodesic arc in $\mathbb{H}^2$. There exist constants $A,B>0$ and $\nu>0$ depending on $K$, and $\gamma$ such that for any parallel geodesics $g$ and $h$ which transversely intersect $\gamma$ at $x,y$, any $X,Y\in \mathfrak{h}_\Theta$, and any $\rho\in K$ we have
\[
     \dG(\mathrm{T}^X_g(\rho), \mathrm{T}^Y_h(\rho))<A \cdot\left(\left(1+B\cdot \max_{\beta\in \Sigma }(|\e^{-\beta(X)}|)\right) \der(x,y)^\nu +\|X-Y\|\right).
     \]
\end{lemma}
\begin{proof}
    The proof is identical to that of Lemma~\ref{lem:TTest} except for using the inequality from Lemma~\ref{lem:parallel} instead of (\ref{eq:tangentvsprojective}).
\end{proof}

We introduce some notion. For intervals $I,J\subset \widetilde{\gamma}$,  we write $I<J$ if  $I^1<x$ for all $x\in J$. Given two collections of intervals $\mathcal{P}$ and $\mathcal{Q}$, write $\mathcal{P}<\mathcal{Q}$ if $I<J$ for all $I\in \mathcal{P}$ and  $J\in \mathcal{Q}$. We say an interval $J$ is bounded by $I$ and $K$ if $I^0<J < K^1$. 

Now we state and prove one of the main theorems of this paper. 

\begin{theorem}\label{thm:convergence}  Let $\mathsf{G}$ be a semisimple complex connected Lie group and Let $\{\mathsf{c}_n\}$ be a sequence of twisted $\mathfrak{h}_\Theta$-valued measured laminations converging weakly to  a twisted $\mathfrak{h}_\Theta$-valued measured lamination $\mathsf{c}$.  Let $Z$ be an open set of $\mathbb{C}\times \Hom_\Theta(\pi_1(S),\mathsf{G})$ such that $\Lambda^{z\mathsf{c}_n}_{x_0}\rho$ is well-defined for all $(z,\rho)\in Z$ and all $n$. Then, the sequence of cataclysm deformations $(z,\rho)\mapsto \Lambda^{z\mathsf{c}_n}(\rho)$ converges to $(z,\rho)\mapsto \Lambda^{z\mathsf{c}}(\rho)$ uniformly on each compact set of $Z$.
\end{theorem}

\begin{proof}
Fix the universal covering $\mathbb{H}^2\to S$. Let $\mu$ and $\mu_n$ be the lifts of $|\mathsf{c}|$, and $|\mathsf{c}_n|$ to $\mathbb{H}^2$ respectively.  We also fix a base-point $x_0\in \mathbb{H}^2\setminus\left(\mu\cup\bigcup_{n>0}\mu_n\right)$. Let 
\[
\Lambda:Z\to \Hom_\Theta(\pi_1(S), \mathsf{G})
\]
be the map given by $\Lambda(z, \rho) = \Lambda^{z\mathsf{c}}_{x_0} (\rho)$ and similarly define $\Lambda_n(z,\rho) = \Lambda^{z\mathsf{c}_n}_{x_0}(\rho)$. We will show that the sequence $\{\Lambda_n\}$ converges to $\Lambda$ uniformly on compact subsets of $Z\subset \mathbb{C}\times \Hom_\Theta(\pi_1(S),\mathsf{G})$. 

To this end, let $\{\gamma_1,\cdots, \gamma_{2g}\}$ be a standard generating set for $\pi_1(S)$ and let $C$ be a compact set in $Z$. We may cover $C$ by finitely many compact sets of the form $D\times K\subset Z$. Thus, it suffices to show that for any $\epsilon>0$ there exists $N$ such that  $\mathrm{d}(\Lambda_n(z,\rho),\Lambda(z, \rho))<\epsilon$ for all $z\in D$, $\rho \in K$, and $n>N$, where $\mathrm{d}$ is any metric on $\Hom(\pi_1(S),\mathsf{G})$, which may be chosen to be 
\[
\mathrm{d}(\rho_1,\rho_2) = \sup_{i=1,2,\cdots,2g} \dG(\rho_1(\gamma_i),\rho_2(\gamma_i))
\]
with respect to the metric $\dG$ on $\mathsf{G}$.

Let $\widetilde{\gamma}_i$ be the geodesic arc joining $x_0$ and $\gamma_ix_0$ oriented  from $x_0$ to $\gamma_i x_0$.  Let $\mathcal{C}=\mathcal{C}_{\widetilde{\gamma}_i}(\mu)$ and $\mathcal{C}_n = \mathcal{C}_{\widetilde{\gamma}_i}(\mu_n)$ be the sets of complementary intervals of $\widetilde{\gamma}_i \setminus \mu$ and $\widetilde{\gamma}_i\setminus \mu_n$, respectively. By definition, we have
\[
\Lambda_n(z,\rho) (\gamma_i)= \varphi_{\widetilde{\gamma_i}}^{z\mathsf{c}_n}(\rho)\cdot \rho(\gamma_i).
\]
Thus, the proof reduces to showing that for each $i=1,2,\cdots,2g$, the sequence $\{\varphi_{\widetilde{\gamma_i}} ^{z\mathsf{c}_n}(\rho)\}$ converges to $\varphi_{\widetilde{\gamma_i}}^{z\mathsf{c}}(\rho)$ uniformly on $D\times K$.

By Lemma~\ref{lem:limit}, there are finitely many maximal geodesic laminations $\tau_1, \cdots, \tau_k$ to which a subsequence of $\{|\mathsf{c}_n|\}$ converges. Therefore, by passing to a subsequence, we may assume that $\{|\mathsf{c}_n|\}$ converges to $\tau_1$, say. Since $\tau_1$ contains $\supp(\mathsf{c})$, the underlying lamination of $|\mathsf{c}|$ can be replaced with $\tau_1$. In other words, we may assume that $\mu$ is the lift of $\tau_1$ to the universal cover $\mathbb{H}^2$ and $\{\mu_n\}$ converges to $\mu$ in the Hausdorff topology.

 For convenience, we will omit the index for the generators of $\pi_1(S)$ and simply write $\widetilde{\gamma}$ instead of $\widetilde{\gamma}_i$.  

\begin{setup}\label{set}
Given $\epsilon>0$, we prepare a number $0<\delta<1$, and an integer $N$ together with finite subsets $\mathcal{I}\subset\mathcal{C}$, and $\mathcal{I}_n\subset\mathcal{C}_n$. The constants $\delta$ and $N$ will depend only on $\epsilon$, $D$, $K$, and $\widetilde{\gamma}$. On the other hand,  $\mathcal{I}$ and $\mathcal{I}_n$ will be chosen for each $(z,\rho)\in D\times K$. 

\begin{enumerate}

\item\label{set:angle} There is a constant $\theta_0$ such that $0<\frac{\pi}{2}-\theta_0<\angle(\widetilde{\gamma},g)<\frac{\pi}{2}+\theta_0<\pi$ for all leaves $g$ of $\mu$. Let $\theta$ be a number satisfying $\theta_0<\theta<\frac{\pi}{2}$. Choose an integer $N_1$ such that $g\in \mathfrak{A}(\widetilde{\gamma},\theta)$ for all leaves $g$ of $\mu_n$ with $n>N_1$. This number $N_1$ depends only on $\widetilde{\gamma}$ and $\mu$.

 Due to Lemma \ref{lem:transverse},  $\{\mathsf{c}_n\}$ converges to $\mathsf{c}$ as a sequence of elements in $\WT(\mathcal{T}_0;\mathfrak{h}_\Theta)$ for some train-track $\mathcal{T}_0$ carrying $\mu$. We regard $\mathsf{c}_n$ and $\mathsf{c}$ as elements of $\WT(\mathcal{T}_0;\mathfrak{h}_\Theta)$.  

\item\label{set:shearing} Since $D$ is compact, there exists a bound $L$ such that $\|z\mathsf{c}([I,J])\|<L$ for all $I,J\in \mathcal{C}$, and $z\in D$ and that $\|z\mathsf{c}_n([I,J])\|<L$ for all $I,J\in \mathcal{C}_n$ with $n>N_1$, and $z\in D$. By  Lemma~\ref{lem:TTest}, Lemma~\ref{lem:TTestparallel} and Set-up~\ref{set}(\ref{set:angle}), one can find constants $A>0$, and $\nu>0$ such that 
\[
     \dG(\mathrm{T}^{X}_g(\rho), \mathrm{T}^{Y}_h(\rho))<A \cdot\left( \der(g\cap \widetilde{\gamma},h\cap \widetilde{\gamma})^\nu + \|X-Y\|\right)
\]
for all $\rho\in  K$, all leaves $g$ and $h$ of $\mu_n$ with $n>N_1$ and all $X,Y\in \mathfrak{h}_\Theta$ with $\|X\|,\|Y\|<L$
\[
    \dG(\mathrm{T}^{X}_g(\rho), \mathrm{T}^{Y}_h(\rho))<A \cdot\left( \der(g\cap \widetilde{\gamma},h\cap \widetilde{\gamma})^\nu + \|X-Y\|\right)
\]
for all $\rho\in  K$, all leaves $g$ and $h$ of $\mu$, and all $X,Y\in \mathfrak{h}_\Theta$ with $\|X\|,\|Y\|<L$ and
\[
\dG(\mathrm{T}^{X}_g(\rho), \mathrm{T}^{Y}_h(\rho))<A \cdot\left( \dPT(\widehat{g}_x,\widehat{h}_y)^\nu + \|X-Y\|\right)
\]
for all $\rho\in  K$, all leaves $g$ of $\mu_n$ with $n>N_1$, all leaves $h$ of $\mu$  and all $X,Y\in \mathfrak{h}_\Theta$ with $\|X\|,\|Y\|<L$. Here, $x=g\cap \widetilde{\gamma}$ and $y=h\cap \widetilde{\gamma}$.

\item\label{set:admissible} Choose an admissible neighborhood $U$ of $\mu$. Find $N_2$ such that $\mu_n\in U$ for all $n>N_2$. Notice that $N_2$ depends only on  $U$. Let $Q$, and $\nu'$ be constants from Lemma~\ref{lem:exponentcase} with respect to $U$ and the exponent $\nu$ from Set-up~\ref{set}(\ref{set:shearing}). They depend only on $U$, $\nu$, and $\widetilde{\gamma}$. 

\item \label{set:operatornorm} Let $\mathcal{J}_n^k=\{J\in \mathcal{C}_n\mid \dr_{|\mathsf{c}_n|} (J)\le k \}$ and let
\[
M= \sup_{\substack{(z,\rho)\in D\times K \\ n=1,2,\cdots,\infty\\k=1,2,\cdots}}\left\| \left(\vec{\prod}_{J\in \mathcal{J}_n ^k}\mathrm{T}^{z\mathsf{c}_n}_{J^0}(\rho)\mathrm{T}^{-z\mathsf{c}_n}_{J^1}(\rho)\right) \mathrm{T}^{z\mathsf{c}_n}_{F_n^0}(\rho)\right\|_{\mathrm{op}},
\]
with conventions $\mathsf{c}_\infty=\mathsf{c}$ and $\mathcal{C}_\infty = \mathcal{C}$. This number is finite since the infinite product converges to $\varphi^{z\mathsf{c}_n}_{\widetilde{\gamma}}(\rho)$ and since $D\times K$ is compact. $M$ depends on $D$, $K$ and $\widetilde{\gamma}$. 

\item\label{set:delta} Choose $0<\delta<1$ so that
\[
2\delta + 2AQM\delta^{\nu'} + AM\delta+2AQM \delta^{\nu'}+4AM\delta <\epsilon.
\]
As all coefficients depend only on $D$, $K$ and $\widetilde{\gamma}$, $\delta$ also depends only on these data.

\item \label{set:totallength} We choose $0<\delta_0<\min\{1, \delta/2\}$ such that the total arc-length of $\mathcal{N}_{\delta_0}(\mu\cap \widetilde{\gamma})$ is at most $\delta$. This is possible by Lemma~\ref{lem:hdimzero}.  Let $\mathcal{T}$ be a train-track carrying $|\mathsf{c}|$ that is contained in $\mathcal{N}_{\delta_0}(|\mathsf{c}|)$.  Let $\widetilde{\mathcal{T}}$ be the lift of $\mathcal{T}$ to $\mathbb{H}^2$. We may choose $\mathcal{T}$ so that the boundary of $\widetilde{\gamma}\cap \widetilde{\mathcal{T}}$ does not meet $\mu_i$ for all $i$ and that $\gamma$ intersects $\mathcal{T}$ along generic ties.

\item \label{set:traintrack}
Lemma~\ref{lem:traintracknbd} guarantees that there exists $N_3$  such that whenever $n>N_3$, each $|\mathsf{c}_n|$ is carried by $\mathcal{T}$.

\item\label{set:partition} Note that $\widetilde{\gamma}\cap \widetilde{\mathcal{T}}$ is a finite union of intervals. Let $\mathfrak{U}=\{U_1, \cdots, U_d\}$ be the connected components of $\widetilde{\mathcal{T}}\cap \widetilde{\gamma}$, enumerated from left to right. Let
\begin{align*}
    \mathfrak{V}&=\{V\in \mathfrak{U}\mid V \cap \supp(\mathsf{c}) =\emptyset\}\\
    \mathfrak{W}&=\{W\in \mathfrak{U}\mid W \cap \supp(\mathsf{c}) \ne\emptyset\}
\end{align*}
so that $\mathfrak{U}=\mathfrak{V}\cup \mathfrak{W}$. Enumerate $\mathfrak{V}=\{V_1,\cdots, V_e\}$ and $\mathfrak{W}=\{W_1, \cdots, W_f\}$ from left to right. 

\item \label{set:I} For each consecutive pair $(U_i,U_{i+1})$ of elements in $\mathfrak{U}$, there exists $I_i\in \mathcal{C}$ intersecting both $U_i$ and $U_{i+1}$. Choose a finite subset containing all of these $I_i$, $i=1,2,\cdots, d-1$ and extend it to a finite set  $\mathcal{I}=\{I\in \mathcal{C}\mid \dr_{|\mathsf{c}|}(I)\le k\}$ for some large enough $k$  such that 
\[
\mathrm{d}_\mathsf{G} \left(\vec{\prod}_{J\in\mathcal{I}}\left(\mathrm{T}^{z\mathsf{c}}_{J^0}(\rho)\mathrm{T}^{-z\mathsf{c}}_{J^1}(\rho)\right) \mathrm{T}^{z\mathsf{c}}_{F^0}(\rho),\varphi^{z\mathsf{c}}_{\widetilde{\gamma}}(\rho)\right)<\delta.
\]
As before, $E$ and $F$ are the elements of $\mathcal{C}$ containing $x_0$ and $\gamma_ix_0$. See Figure~\ref{fig:setup}.

\item \label{set:totalcocycle} Choose $N_4$ such that 
\[
\sum_{j=1}^e\| {z\mathsf{c}}_n(V_j)\|<\delta
\]
and that 
\[
\|z\mathsf{c}_n([x_0,U_i^0])-z\mathsf{c}([x_0, U_i^0])\|<\frac{\delta}{f}
\]
for all $n>N_4$, all $z\in D$ and all $i=1,2,\cdots,d$. This is possible by Lemma~\ref{lem:transverse}.  

\item \label{set:dPTbound} For $n>N_1$, leaves of $|\mathsf{c}_n|$ and leaves of $|\mathsf{c}|$ are members of $\mathfrak{A}(\widetilde{\gamma},\theta)$, where $\theta$ is obtained in Set-up~\ref{set}(\ref{set:angle}). We find a constant $\delta_1<\delta_0$ such that whenever $g,h\in \mathfrak{A}(\widetilde{\gamma},\theta)$ satisfy $\mathrm{D_H}(\widehat{g}, \widehat{h})<\delta_1$, we have $\dPT(\widehat{g}_x,\widehat{h}_y)<(\delta /f)^{1/\nu}$. This is possible by Lemma~\ref{lem:approx}. Choose $N_5$ such that  $\mathrm{D_H}(|\mathsf{c}_n|,|\mathsf{c}|)<\delta_1$ for all $n>N_5$.  Set $N=\max(N_1,N_2,N_3,N_4,N_5)$.

\item \label{set:In} Assume $n>N$. For each $i=1,2,\cdots, d-1$, let $B_i\in \mathcal{C}_n$ be the one that intersects both $U_i$ and $U_{i+1}$. As we did in Set-up~\ref{set}(\ref{set:I}),   we  form a set  containing all of these $B_1, \cdots, B_{d-1}$  and extend it  to a finite subset $\mathcal{I}_n=\{I\in  \mathcal{C}_n\mid \dr_{|\mathsf{c}|}(I)\le k\}$ for some large $k$ such that 
    \[
\dG\left(\vec{\prod_{J\in\mathcal{I}_n}} \left(\mathrm{T}^{z\mathsf{c}_n}_{J^0} \mathrm{T}^{-z\mathsf{c}_n}_{J^1} \right)\mathrm{T}_{F_n^0}^{z\mathsf{c}_n}, \varphi^{z\mathsf{c}_n}_{\widetilde{\gamma}}\right)<\delta
\]
where $E_n$ and $F_n$ denote the elements of $\mathcal{C}_n$ containing $x_0$ and $\gamma_i x_0$, respectively. 
\end{enumerate}
\end{setup}

\begin{figure}[!htb]
    \centering
    \includegraphics[width=\linewidth]{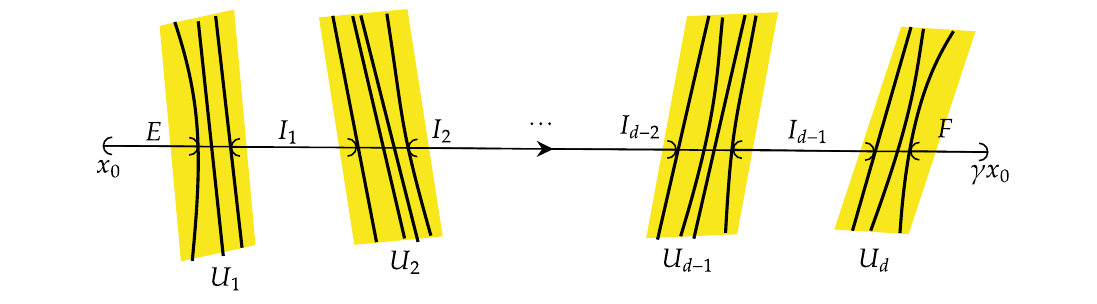}
    \caption{Configuration for Set-up~\ref{set}. The horizontal line denotes $\widetilde{\gamma}$.  Yellow strips are parts of $\widetilde{\mathcal{T}}$. Thick vertical lines are leaves of the lamination $\mu$.}
    \label{fig:setup}
\end{figure}

Now under the assumption $n>N$, we will make several estimates. Let 
\[
\mathcal{Q}_n^i :=\{J\in \mathcal{I}_n\mid J\subset U_i\}
\]
for each $U_i\in \mathfrak{U}$. Let
\[
\mathcal{Q}^1_n\cup\mathcal{P}^1_n\cup\mathcal{Q}^2_n\cup\mathcal{P}^2_n\cup\mathcal{Q}^2_n\cup \cdots\cup \mathcal{P}^{d-1}_n\cup \mathcal{Q}^{d}_n
\]
be a partition of $\mathcal{I}_n$ such that $\mathcal{Q}_n ^1<\mathcal{P}^1_n <  \mathcal{Q}_n^2< \mathcal{P}^2_n<\cdots<\mathcal{Q}_n^{d}$. We formally allow that some $\mathcal{Q}_n^i$ are empty. Also notice that each $\mathcal{P}^i_n$ consists of a single element bounded by $U_i$ and $U_{i+1}$. 

Similarly, we define 
\[
\mathcal{Q}^i=\{J\in \mathcal{I}\mid J\subset U_i\}
\]
and form a partition 
\[
\mathcal{I} = \mathcal{Q}^1\cup\mathcal{P}^1\cup\mathcal{Q}^2\cup\mathcal{P}^2\cup\mathcal{Q}^2\cup \cdots\cup \mathcal{P}^{d-1}\cup \mathcal{Q}^{d}
\]
with $\mathcal{Q}^1<\mathcal{P}^1<\mathcal{Q}^2<\cdots<\mathcal{P}^{d-1}<\mathcal{Q}^d$.

\begin{claim}\label{est1}
    For each $\mathcal{Q}^{i}_n$, we have the estimate
    \[
    \dG\left(\mathrm{Id},\vec{\prod_{J\in \mathcal{Q}^i_n}}\mathrm{T}^{z\mathsf{c}_n}_{J^0}\mathrm{T}^{-z\mathsf{c}_n}_{J^1}\right)< A\cdot\sum_{J\in \mathcal{Q}_n^i} \ell(J)^\nu
    \]
    where $A>0$ and $\nu>0$ are constants depending only on the compact sets $D$, $K$, and the transverse arc $\gamma$. The same estimate holds for each $\mathcal{Q}^i$. 
\end{claim}
\begin{proof}
Due to Set-up~\ref{set}(\ref{set:shearing}), we have that
\[
 \dG\left(\mathrm{Id}, \mathrm{T}^{z\mathsf{c}_n}_{J^0}\mathrm{T}^{-z\mathsf{c}_n}_{J^1}\right) < A\cdot \mathrm{d}(J^0,J^1)^\nu
\]
for each $J\in \mathcal{Q}_n^i$, where $A$ and $\nu$ depend on $D$, $K$ and  $\gamma$. Hence, by (\ref{eq:product}, 
    \[
    \dG\left(\mathrm{Id},\vec{\prod_{J\in \mathcal{Q}^i_n}}\mathrm{T}^{z\mathsf{c}_n}_{J^0}\mathrm{T}^{-z\mathsf{c}_n}_{J^1}\right)\le A\cdot \sum_{J\in \mathcal{Q}_n^i} \ell(J)^\nu.\qedhere
    \]
    \end{proof}

For each $\mathcal{Q}_n^i$ and $\mathcal{Q}^i$, we can replace the product
\[
\vec{\prod_{J\in \mathcal{Q}^i_n}}\mathrm{T}^{z\mathsf{c}_n}_{J^0}\mathrm{T}^{-z\mathsf{c}_n}_{J^1}\quad\text{and}\quad \vec{\prod_{J\in \mathcal{Q}^i}}\mathrm{T}^{z\mathsf{c}}_{J^0}\mathrm{T}^{-z\mathsf{c}}_{J^1}
\]
with the identity with errors of
\[
AM\cdot \sum_{J\in \mathcal{Q}_n^i} \ell(J)^\nu, \quad \text{and} \quad AM\cdot\sum_{J\in \mathcal{Q}^i} \ell(J)^\nu,
\]
respectively. Here, we applied (\ref{eq:deletemiddle}) inductively. The constant $M$ comes from Set-up~\ref{set}(\ref{set:operatornorm}). Observe that, by Set-up~\ref{set}(\ref{set:totallength}),  the elements in $\mathcal{Q}_n^i$ and $\mathcal{Q}^i$ are contained in $\mathcal{N}_{\delta_0}(\mu\cap \widetilde{\gamma})$. By Set-up~\ref{set}(\ref{set:totallength}),  the total arc length of $\mathcal{N}_{\delta_0}(\mu\cap \widetilde{\gamma})$ is less than $\delta$. Combining this with Set-up~\ref{set}(\ref{set:admissible}), it follows that the total error for these replacements is bounded from above by 
\[
2 AQ M \cdot \delta^{\nu'}.
\]
The constants $M$, $Q$, and $\nu'$ come from Set-up~\ref{set}(\ref{set:admissible}) and they depend on $D$, $K$ and $\gamma$. 

After this process, the original products may be reduced to
\[
\vec{\prod_{J\in \mathcal{P}^1\cup\mathcal{P}^2\cup\cdots\cup \mathcal{P}^{d-1}}} (\mathrm{T}^{z\mathsf{c}}_{J^0}\mathrm{T}^{-z\mathsf{c}}_{J^1})\mathrm{T}^{z\mathsf{c}}_{F^0}
\]
and to
\[
\vec{\prod_{J\in \mathcal{P}_n^1\cup\mathcal{P}^2_n\cup\cdots\cup \mathcal{P}_n^{d-1}}} (\mathrm{T}^{z\mathsf{c}_n}_{J^0}\mathrm{T}^{-z\mathsf{c}_n}_{J^1})\mathrm{T}^{z\mathsf{c}_n}_{F_n^0}.
\]
According to Set-up~\ref{set}(\ref{set:In}), each $\mathcal{P}_n^i$ consists of a single element $B_i\in \mathcal{I}_n$ that is bounded between $U_i$ and $U_{i+1}$. Thus, a consecutive pair $(B_i, B_{i+1})$ bounds $U_{i+1}\in \mathfrak{U}$. We call a pair $(B_i, B_{i+1})$ of elements in $\mathcal{I}_n\cup\{F_n\}$ negligible if they bound $U_{i+1}\in \mathfrak{V}$. Let $\mathfrak{N}_n$ be the set of negligible pairs. We similarly define the set of negligible pairs $\mathfrak{N}$ for $\mathcal{I}\cup\{F\}$.

\begin{claim}\label{est2} For each  $(B_i,B_{i+1})\in \mathfrak{N}_n$ bounding $V_j\in \mathfrak{V}$, we have
    \[
        \dG\left(\mathrm{T}^{-z\mathsf{c}_n}_{B_i^1}\mathrm{T}^{z\mathsf{c}_n}_{B_{i+1}^0},\,\mathrm{Id}\right)< A\cdot(\|z\mathsf{c}_n(V_j)\|+\der (B^1_i, B^0 _{i+1})^\nu)
    \]
    for some constants $A>0$ and $\nu>0$ depending on $D$, $K$, and  $\gamma$. The same estimate holds for each pair in $\mathfrak{N}$.
\end{claim}
\begin{proof} By Set-up~\ref{set}(\ref{set:shearing}), we have
\begin{align*}
        \dG\left(\mathrm{T}^{-z\mathsf{c}_n}_{B_i^1}\mathrm{T}^{z\mathsf{c}_n}_{B_{i+1}^0},\,\mathrm{Id}\right)&=
        \dG\left(\mathrm{T}^{z\mathsf{c}_n}_{B_{i+1}^0},\,\mathrm{T}^{z\mathsf{c}_n}_{B_i^1}\right)\\
        &< A\cdot(\der (B^1_i, B^0 _{i+1})^\nu+\|z\mathsf{c}_n([x_0,B_{i+1}])-z\mathsf{c}_n([x_0,B_i])\|)\\
        &= A\cdot(\der (B^1_i, B^0 _{i+1})^\nu+\|z\mathsf{c}_n(V_j)\|)
\end{align*}
where $A$ and $\nu$ depend on $D$, $K$ and $\gamma$.
\end{proof}

By Claim~\ref{est2}, whenever we see a term of the form
\[
\mathrm{T}^{-z\mathsf{c}_n}_{B_i^1}\mathrm{T}^{z\mathsf{c}_n}_{B_{i+1}^0},
\]
with $(B_i,B_{i+1})\in \mathfrak{N}_n$, in the product 
\[
\vec{\prod_{J\in \mathcal{P}_n^1\cup\mathcal{P}^2_n\cup\cdots\cup \mathcal{P}_n^{d+1}}} (\mathrm{T}^{z\mathsf{c}_n}_{J^0}\mathrm{T}^{-z\mathsf{c}_n}_{J^1})\mathrm{T}^{z\mathsf{c}_n}_{F_n^0}
\]
we can remove it with an error of
\[
AM\cdot(\|z\mathsf{c}_n(V_j)\|+\der (B^1_i, B^0 _{i+1})^\nu).
\]
The same replacement for the product 
\[
\vec{\prod_{J\in\mathcal{I}}}\left(\mathrm{T}^{z\mathsf{c}}_{J^0}\mathrm{T}^{-z\mathsf{c}}_{J^1}\right) \mathrm{T}^{z\mathsf{c}}_{F^0}
\]
costs
\[
AM\cdot \der (B^1_i, B^0 _{i+1})^\nu.
\]

Therefore,  the total error of such modifications is at most
\[
AM \cdot \sum_{i=1}^e ( \|z\mathsf{c}_n(V_i)\|+ 2\cdot \ell(V_i)^\nu)<AM\delta+2AQM \delta^{\nu'} 
\]
where we used Set-up~\ref{set}(\ref{set:totalcocycle}) and Set-up~\ref{set}(\ref{set:admissible}). 

Now enumerate $X_0,X_1, \cdots, X_{f-1},X_f$ from left to right the connected components of $\widetilde{\gamma}\setminus (W_1\cup W_2\cup \cdots \cup W_f)$. By our construction of $\mathcal{I}$ and $\mathcal{I}_n$ (Set-up~\ref{set}(\ref{set:I}), and  Set-up~\ref{set}(\ref{set:In})), we have $P_{i,\pm}\in \mathcal{I}\cup\{F\}$ and $P_{n,i,\pm}\in \mathcal{I}_n\cup\{F_n\}$ such that
\begin{multline*}
X_0<P_{n,0,+}^1,\,P_{0,+}^1< P_{n,1,-}^0,\,P_{1,-}^0<X_1<P_{n,1,+}^1,\,P_{1,+}^1\\<P_{n,2,-}^0,\,P_{2,-}^0<X_2<P_{n,2,+}^1,\,P_{2,+}^1<\cdots\\
<P_{n,f-1,+}^1,\,P_{f-1,+}^1<P_{n,f,-}^0,\,P_{f,-}^0<X_f.
\end{multline*}
See Figure~\ref{fig:W}. Note that we may have $P_{i,-}=P_{i,+}$ and $P_{n,i,-}=P_{n,i,+}$ when there is no element of $\mathfrak{V}$ between $W_i$ and $W_{i+1}$.

\begin{figure}[!htb]
    \centering
    \includegraphics[width=\linewidth]{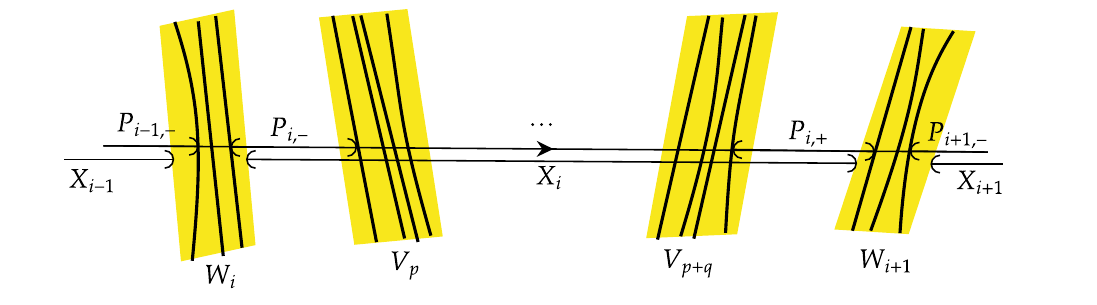}
    \caption{Definition of $P_{i,\pm}$. Between $W_i$ and $W_{i+1}$ all yellow strips belong to $\mathfrak{V}$.}
    \label{fig:W}
\end{figure}
With this notation, it remains to compare  
\[
\mathrm{T}^{-z\mathsf{c}_n}_{P_{n,0,+}^1}\left(\prod_{i=1} ^{f-1} \mathrm{T}^{z\mathsf{c}_n}_{P_{n,i,-}^0}\mathrm{T}^{-z\mathsf{c}_n}_{P_{n,i,+}^1}  \right)\mathrm{T}^{z\mathsf{c}_n}_{P_{n,f,-}^0} 
\]
and
\[
\mathrm{T}^{-z\mathsf{c}}_{P_{0,+}^1}\left(\prod_{i=1} ^{f-1} \mathrm{T}^{z\mathsf{c}}_{P_{i,-}^0}\mathrm{T}^{-z\mathsf{c}}_{P_{i,+}^1}  \right)\mathrm{T}^{z\mathsf{c}}_{P_{f,-}^0}.
\]
Due to Set-up~\ref{set}(\ref{set:totalcocycle}), Set-up~\ref{set}(\ref{set:dPTbound}) and Set-up~\ref{set}(\ref{set:shearing}),
\[
\dG(\mathrm{T}_{P_{n,i,+}^1}^{-z\mathsf{c}_n},\mathrm{T}_{P_{i,+}^1}^{-z\mathsf{c}})<2A\cdot \frac{\delta}{f}
\]
for each $i=0,1,\cdots,f-1$. Likewise, we obtain
\[
\dG(\mathrm{T}_{P_{n,i,-}^0}^{-z\mathsf{c}_n},\mathrm{T}_{P_{i,-}^0}^{-z\mathsf{c}})<2A\cdot \frac{\delta}{f}
\]
for $i=1,2,\cdots,f$. Since there are $f$ such terms, the total error is smaller than $4AM \delta$.

Overall, whenever $n>N$, the distance between $\varphi^{z\mathsf{c}_n}_{\widetilde{\gamma}}(\rho)$ and $\varphi^{z\mathsf{c}}_{\widetilde{\gamma}}(\rho)$ is at most
\[
2\delta + 2AQM\delta^{\nu'} + AM\delta+2AQM \delta^{\nu'}+4AM\delta <\epsilon
\]
by Set-up~\ref{set}(\ref{set:delta}). This completes the proof. 
\end{proof}

Let $V$ be a finite dimensional vector space with an involution. We say that an element $\mathsf{c}\in \HD(S;V)$ is a $V$-weighted multicurve if the support of $\mathsf{c}$ consists of disjoint simple closed curves. 

\begin{corollary}\label{cor:holomorphic}
    Let $\mathsf{G}$ be a complex semisimple connected Lie group and let $\rho_0\in \Hom_\Theta(\pi_1(S), \mathsf{G})$ be a smooth point.  Choose a twisted $\mathfrak{h}_\Theta$-valued measured lamination $\mathsf{c}\in \HD(S;\mathfrak{h}_\Theta)$ that can be weakly approximated by a sequence of $\mathfrak{h}_\Theta$-weighted multicurves. Then for a small neighborhood $D\subset \mathbb{C}$ containing 0 and a small neighborhood $W$ of $\rho_0$, the map $H^\mathsf{c}:D\times W\to \Hom_\Theta(\pi_1(S),\mathsf{G})$ defined by $(z,\rho)\mapsto \Lambda^{z\mathsf{c}}(\rho)$ is holomorphic. 
\end{corollary}

\begin{proof}
    We know that $H^\mathsf{c}$ is holomorphic if $\supp(\mathsf{c})$ is a $\mathfrak{h}_\Theta$-weighted multicurve. In general,  we find a sequence of $\mathfrak{h}_\Theta$-weighted multicurves $\{\mathsf{c}_n\}$ that converges weakly to $\mathsf{c}$. By Theorem~\ref{thm:convergence}, the sequence of holomorphic maps $H^{\mathsf{c}_n}$ converges to $H^\mathsf{c}$ uniformly on compact sets. Therefore, $H^\mathsf{c}$ is also holomorphic.
    \end{proof}

Let $\mathsf{G}_{\mathrm{split}}$ be a split real form of a complex connected simple Lie group $\mathsf{G}$. The $\mathsf{G}$-Hitchin component is a connected component of the space of $\Theta$-Anosov representations with $\Theta=\Delta$. In this case we have $\mathfrak{a}= \mathfrak{a}_\Theta$. 

\begin{corollary}\label{cor:analytic}
    Let $\mathsf{G}_{\mathrm{split}}$ be a split real form of a complex simple Lie group $\mathsf{G}$. Fix a twisted $\mathfrak{a}$-valued measured lamination $\mathsf{c}\in \HD(S;\mathfrak{a})$ that can be weakly approximated by a sequence of $\mathfrak{a}$-weighted multicurves. Then the cataclysm deformation $\Lambda^\mathsf{c}:\Hit_{\mathsf{G}_{\mathrm{split}}}(S) \to \Hit_{\mathsf{G}_{\mathrm{split}}}(S)$ is real analytic. 
\end{corollary}
\begin{proof}
Since $\mathsf{G}_{\mathrm{split}}$ is an analytic subgroup of $\mathsf{G}$, $\Hit_{\mathsf{G}_{\mathrm{split}}}(S)\subset \Hom_\Theta(\pi_1(S),\mathsf{G})/\mathsf{G}$ is a real analytic submanifold.

We observe that $\Lambda^\mathsf{c}:\Hom(\pi_1(S),\mathsf{G})/\mathsf{G}\to \Hom(\pi_1(S),\mathsf{G})/\mathsf{G}$ is holomorphic on a neighborhood of a Hitchin representation if $\mathsf{c}$ is an $\mathfrak{a}$-weighted multicurve. For a general $\mathsf{c}$, we use Theorem~\ref{thm:convergence} to show that $\Lambda^\mathsf{c}$ is  holomorphic. The result follows since $\Hit_{\mathsf{G}_{\mathrm{split}}}(S)\subset \Hom_\Theta(\pi_1(S),\mathsf{G})/\mathsf{G}$ is preserved by $\Lambda^\mathsf{c}$.
\end{proof}

\section{Applications to Hitchin components}\label{sec:app}

In this section, we discuss several applications of cataclysm deformations on Hitchin components. 

Throughout this section $\mathsf{G}_\mathrm{split}$ will denote the split real form of a complex simple Lie group $\mathsf{G}$. Then, the Cartan subalgebra $\mathfrak{h}$ of $\mathsf{G}$ and the maximal abelian subalgebra $\mathfrak{a}$ of $\mathsf{G}_{\mathrm{split}}$ are related by $\mathfrak{h} = \mathfrak{a}\otimes\mathbb{C}=\mathfrak{a}\oplus \ui \mathfrak{a}$. 

Let $U$ be a domain in $\mathsf{G}$ where the complex Jordan projection $\lambda^\mathbb{C}$ is defined. (See Section~\ref{sec:lie}). Note that $U$ contains all purely loxodromic elements.

We will be only concerned with $\mathsf{G}_{\mathrm{split}}$-Hitchin representations and their small deformations inside $\Hom(\pi_1(S), \mathsf{G})$. In this setting, we may assume the following properties.  
\begin{itemize}
    \item Every representation is Borel Anosov; that is, $\Theta$-Anosov with $\Theta= \Delta$. In this case, we have that  $\mathfrak{h}_\Theta = \mathfrak{h}$, $\mathfrak{p}_\Theta = \mathfrak{p}$ and $\mathfrak{p}_\Theta^\mathrm{op} = \mathfrak{p}^\mathrm{op}$, where $\mathfrak{h}$ is the Cartan subalgebra of $\mathfrak{g}$ and  $\mathfrak{p}$ is the the minimal parabolic subalgebra of $\mathfrak{g}^\bR$. Moreover, $\mathsf{P}\cap \mathsf{P}^\mathrm{op}$ is connected and equal to the Cartan subgroup $\mathsf{H}$. 
    \item For any non-trivial $\gamma\in \pi_1(S)$ and any $\rho\in \Hom(\pi_1(S), \mathsf{G})$,  $\rho(\gamma)$ is conjugate into  $U\cap \mathsf{H}^+$. Note that for a Borel Anosov representation, we can only assume that $\rho(\gamma)$ is conjugate into $\mathsf{H}^+$. The extra condition that $\rho$ is close enough to a $\mathsf{G}_\mathrm{split}$-Hitchin representation allows us to have this further property. 
\end{itemize} 

\subsection{Cataclysms associated with measured laminations}

We are concerned with a particular type of cataclysms that are induced from measured geodesic laminations. Choose a non-trivial element $X\in \mathfrak{h}$ and a measured lamination $\mu\in \ML(S)$. Orient  each leaf of the orientation double cover $\mu^\mathrm{or}$ of the support of $\mu$ continuously and consider the map 
\begin{equation}\label{eq:muX}
\mu^{X}: \gamma \mapsto\begin{cases} X \mu(\gamma) & \text{ if }\gamma \text{ intersects } \mu^\mathrm{or} \text{ positively}\\
 -w_0X \mu(\gamma) & \text{ if }\gamma \text{ intersects } \mu^\mathrm{or} \text{ negatively}
\end{cases}
\end{equation}
that assigns an element of $\mathfrak{h}$ to each short oriented arc $\gamma$ transverse to $\mu^\mathrm{or}$. One can check that this gives rise to a twisted $\mathfrak{h}$-valued measured lamination. Note that $\mu^X$ depends on the choice of the orientation of $\mu^\mathrm{or}$.

In this section, we are concerned with cataclysm deformations of the form $\Lambda^{\mu^X}$ for $\mu\in \ML(S)$, and $X\in \mathfrak{h}$.

\begin{remark}
    If $\mu=(\gamma, s)\in \mathcal{S}\times \mathbb{R}_+\subset  \ML(S)$ is a weighted simple closed geodesic, then $\Lambda^{t\mu^X}$, $t\in \bR$, is the Goldman flow along $\gamma$ with respect to $X$. 
\end{remark}

Very often, we will assume that $X$ is invariant under the opposite involution for the following two reasons. If $X$ satisfies  $X=-w_0X$,  we may simply define $\mu^X(\gamma)= X\mu(\gamma)$ for oriented arcs $\gamma$. Then, since 
\[
\mu^X(\overline{\gamma})=X\mu(\overline{\gamma}) = X\mu (\gamma)=-w_0 X \mu(\gamma)=-w_0\mu^X(\gamma),
\]
$\mu^X$ becomes a twisted $\mathfrak{h}$-valued measured lamination. Note that, in this case, $\mu^X$ is defined without ambiguity. In general, however, $\mu^X$ depends on the choice of the orientation of $\mu^\mathrm{or}$.  Furthermore, the assumption  $X= -w_0X$ is crucial when we discuss the convergence of a sequence of twisted measured laminations associated with  $X$. 
\begin{lemma}\label{lem:convergetoo}
Let $\{\mu_n\}$ be a sequence of measured laminations converges to $\mu$ weakly and let $X$ be an element of $\mathfrak{h}$.  If $-w_0X = X$, then, $\{\mu_n ^X\}$ converges weakly to $\mu^X$.    
\end{lemma}
\begin{proof}
Since $X=-w_0X$, we know that $\mu_n^X(\gamma)=X\mu_n(\gamma)$ regardless of the orientation of $\gamma$. Hence, the lemma follows. 
\end{proof}

Without the condition $X= -w_0 X$, the weak convergence of $\{\mu_n\}$ does not guarantee the weak convergence of $\{\mu_n ^X\}$. 

\begin{example}\label{ex:counter} Consider a surface $S$ of genus 3, an oriented simple closed geodesic $\gamma'$ and  simple closed geodesics $\gamma$, $\eta$ with $i(\gamma, \eta) =2$ as in Figure~\ref{fig:cocycle}. Choose a non-trivial $X\in \mathfrak{h}$ with $X= w_0 X$. Let $\gamma_n$ be the image of $\gamma$ after applying the Dehn twist along $\eta$ $n$-times. Let $\mu_n=(\gamma_n, \frac{1}{2n})$ be the sequence in $ \mathcal{S}\times \bR^+$. It converges weakly to the weighted simple closed curve $\mu=(\eta, 1)\in \mathcal{S}\times \bR_+$. In particular, $\mu^X(\gamma') = \pm X$ (depending on the choice of the orientation of $\eta$). On the other hand, the condition $X=w_0X$ forces $\mu_n^X(\gamma')= 0$ for all $n>0$. Thus, $\{\mu^X_n\}$ does not converge weakly to $\mu^X$. 
\end{example}
\begin{figure}[!htb]
    \centering
    \includegraphics[width=\linewidth]{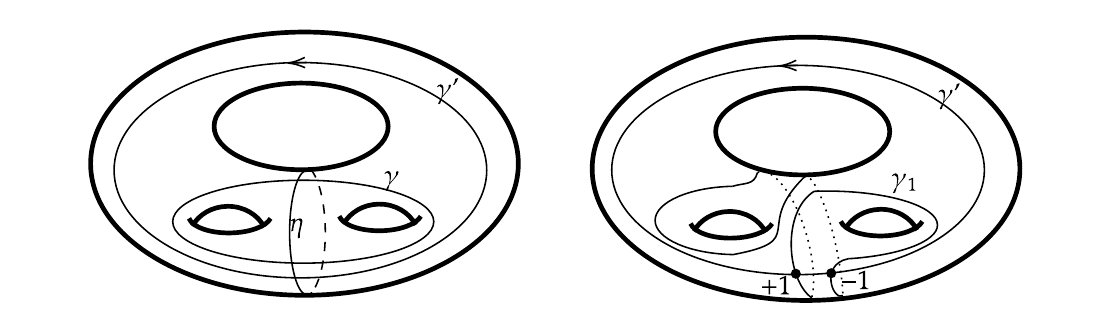}
    \caption{Left: The configuration of geodesics for Example~\ref{ex:counter}. Right: After the Dehn twist along $\eta$,  $\gamma_1$ and $\gamma'$ have two intersections with opposite sign. Thus, we obtain  $\mu_1^X(\gamma') = \frac{1}{2}(X-w_0X) = 0$.}
    \label{fig:cocycle}
\end{figure}

Let $\mathbf{a}$ and $\mathbf{b}$ be non-trivial elements in $\mathfrak{h}^*$.   Let $\gamma\in \pi_1(S)\setminus\{1\}$. Define the function $\ell^\mathbf{a}_\gamma:\Hit_{\mathsf{G}_\mathrm{split}}(S)\to \bR$ by
\[
\ell^\mathbf{a} _\gamma ([\rho]) = \mathbf{a} (\lambda(\rho(\gamma))).
\]
This quantity may be referred to as the $\mathbf{a}$-\emph{length function}. Since the Jordan projection $\lambda$ is analytic on the set of purely loxodromic elements, $\ell^{\mathbf{a}}_\gamma$ is an analytic function on $\Hit_{\mathsf{G}_\mathrm{split}}(S)$. 

Observe that
\[
\ell^\mathbf{a}_{\gamma^{-1}} (\rho) = \mathbf{a}(\lambda(\rho(\gamma^{-1}))) = \mathbf{a}(-w_0\lambda(\rho(\gamma)))=\ell^{-w_0\mathbf{a}}_{\gamma}(\rho).
\]
If we further assume that $\mathbf{a}$ is fixed by the opposite involution, $\mathbf{a}=-w_0 \mathbf{a}$, then $\ell^{\mathbf{a}}_\gamma$ does not depend on the orientation of $\gamma$, which is an expected property of the concept of ``length.'' For this reason, we will always assume that $\mathbf{a}$ is fixed by the opposite involution when we discuss the $\mathbf{a}$-length.

Choose a simple closed geodesic $\eta$, regarded as a weighted closed curve with weight 1. Let $\mathbf{b}^\vee\in \mathfrak{h}$ be the dual of $\mathbf{b}$, that is, $\mathbf{b}(Y) = \kf (\mathbf{b}^\vee,Y)$ for all $Y\in \mathfrak{h}$.  Then $\eta^{\mathbf{b}^\vee}$ is a twisted $\mathfrak{h}$-valued measured lamination so that $\Lambda^{t\eta^{\mathbf{b}^\vee}}$ is defined for small enough $t$. 

In the following subsection, we will compute the first variation of the $\mathbf{a}$-length function $\ell^\mathbf{a}_\gamma$ along the 1-parameter family of cataclysm deformations $\Lambda^{t\mu^{\mathbf{b}^\vee}}$, for $-\epsilon<t<\epsilon$. To avoid cumbersome notation, we write $\mu^\mathbf{b}$ instead of $\mu^{\mathbf{b}^\vee}$.

\subsection{Goldman product formula revisited} 
The main tool to compute the first variation of a length function along a cataclysm  is the Goldman product formula \cite{Goldman86}. In this subsection, we first briefly recall the result and interpret it in terms of geodesic currents. Then, we derive a general variation formula. 

As in the previous section, $\mathbf{a},\mathbf{b}\in \mathfrak{h}^*$ are functionals invariant under the opposite involution and $\gamma,\eta$ are elements of $\pi_1(S)\setminus\{1\}$. For technical reasons, we assume that $\gamma$ and $\eta$ are primitive and $\eta$ is simple. For a non-primitive element $\gamma^n$, we have $\ell^\mathbf{a}_{\gamma^n} = |n| \cdot \ell^\mathbf{a}_\gamma$. Using this relation, one can generalize most of the results in this section to non-primitive elements.  We may view $\gamma$ and $\eta$ as oriented closed geodesics in $S$. We also regard $\eta$ as a weighted simple closed curve with weight 1 so that the cataclysm $\Lambda^{t\eta^\mathbf{b}}$, for $-\epsilon<t<\epsilon$, is well defined. 

Recall that, in Section~\ref{sec:lie}, we defined the Goldman functions  $\widehat{\mathbf{a}\circ\lambda^{\mathbb{C}} }$ and $\widehat{\mathbf{b}\circ\lambda^{\mathbb{C}} }$ of $\mathbf{a}\circ\lambda^{\mathbb{C}}$ and $\mathbf{b}\circ\lambda^{\mathbb{C}}$, respectively. For each  $p=(\vec{\gamma}_{x_p}, \vec{\eta}_{x_p})\in \gamma\sharp \eta$, let $\gamma_p$ be the element in $\pi_1(S,x_p)$ that is represented by the parametrized closed geodesic emanating from $x_p$ and tangent to the vector $\vec{\gamma}_{x_p}\in \UT_{x_p} S$. Similarly, we define the element $\eta_p\in \pi_1(S,x_p)$. The sign of $p\in \gamma\sharp \eta$ is 1 or $-1$ depending on whether $p$ is a positive or negative intersection. The Goldman product formula in this context reads:
\begin{equation}\label{eq:goldman}
\frac{\der }{\der t} \ell^\mathbf{a}_\gamma (\Lambda^{t\eta^\mathbf{b}}\rho)=\sum_{p\in\gamma\sharp \eta} (\operatorname{sign} p)\cdot \kf(\widehat{\mathbf{a}\circ\lambda^{\mathbb{C}} }(\rho(\gamma_p)), \widehat{\mathbf{b}\circ\lambda^{\mathbb{C}} }(\rho(\eta_p))).
\end{equation}

The standing assumptions on $\rho$ in this section are much stricter than what the original theorem requires. Indeed, the Goldman product formula holds for general invariant functions and representations, provided that the representation is contained in the smooth locus of the character variety. 

Now we rewrite (\ref{eq:goldman}) in terms of geodesic currents.

Let $\mathfrak{T}\subset \mathcal{F}\times \mathcal{F}^\mathrm{op}$ be the space of transverse flags. We define the function $\kf^{\mathbf{a},\mathbf{b}}:\mathfrak{T}\times \mathfrak{T}\to \mathbb{C}$ by
\[
\kf^{\mathbf{a},\mathbf{b}}(\mathfrak{x},\mathfrak{y},\mathfrak{z},\mathfrak{w}) = \kf(\Ad_{\mathsf{g}(\mathfrak{x},\mathfrak{y})}\mathbf{a}^\vee,\Ad_{\mathsf{g}(\mathfrak{z},\mathfrak{w})}\mathbf{b}^\vee)
\]
where $\mathsf{g}(\mathfrak{x},\mathfrak{y})$ and $\mathsf{g}(\mathfrak{z},\mathfrak{w})$ are elements in $\mathsf{G}$ defined as in (\ref{eq:sligher}). As explained in Lemma~\ref{lem:welldef},  the function $\kf^{\mathbf{a},\mathbf{b}}$ is well-defined no matter how we choose $\mathsf{g}(\mathfrak{x},\mathfrak{y})$ and $\mathsf{g}(\mathfrak{z},\mathfrak{w})$.

Let $\rho$ be a  Borel Anosov representation with the limit map $\xi_\rho:\partial_\infty \pi_1(S)\to \mathcal{F}$. Let $\vec{\mathcal{G}}$ be the space of oriented geodesics in $\mathbb{H}^2$ and let $\vec{\mathcal{G}}^\times$ be the space of pairs of oriented geodesics $(g,h)$ that intersect transversely.  Now we define the function $\kf^{\mathbf{a},\mathbf{b}}_\rho:\vec{\mathcal{G}}^{2}\to \mathbb{C}$. For $g\in \vec{\mathcal{G}}$,  Denote by $g^+$ and $g^-$ the forward and backward endpoints of $g$, respectively. Then, 
\[
\kf^{\mathbf{a},\mathbf{b}}_\rho (g,h)= \mathrm{sign}(p) \cdot \kf^{\mathbf{a},\mathbf{b}}(\xi_\rho(g^+),\xi^\mathrm{op}_\rho (g^-), \xi_\rho (h^+), \xi^\mathrm{op}_\rho (h^-))
\]
for $(g,h)\in \vec{\mathcal{G}}^\times$ and $p\in \vec{g}\sharp \vec{h}$. We declare $\kf^{\mathbf{a},\mathbf{b}}_\rho (g,h)=0$ when  $g$ and $h$ do not transversely intersect. 
Notice that
\[
\kf^{\mathbf{a},\mathbf{b}}_\rho (g,h) =- \kf^{\mathbf{b},\mathbf{a}}_\rho (h,g)
\]
and
\[
\kf^{\mathbf{a},\mathbf{b}}_\rho (\overline{g},h)=-\kf^{w_0\mathbf{a},\mathbf{b}}_\rho (g,h),
\]
where $\overline{g}$ is the oriented geodesic with the opposite orientation. Since $\mathbf{a}$ and $\mathbf{b}$ are invariant under the opposite involution, we obtain 
\[
\kf^{\mathbf{a},\mathbf{b}}_\rho (\overline{g},h)=\kf^{\mathbf{a},\mathbf{b}}_\rho (g,\overline{h})=\kf^{\mathbf{a},\mathbf{b}}_\rho (g,h),
\]
showing that the function $\kf^{\mathbf{a},\mathbf{b}}_\rho$ can be viewed as the function on the space $\mathcal{G}^2$ of pairs of unoriented geodesics.

Let $\widetilde{\gamma}$ and $\widetilde{\eta}$ be the full lifts of $\gamma$ and $\eta$ to $\mathbb{H}^2$, respectively. When establishing the Goldman product formula, we defined, for each $p=(\vec{\gamma}_{x_p}, \vec{\eta}_{x_p})\in \gamma\sharp \eta$, the  elements $\gamma_p$ and $\eta_p$ in $\pi_1(S,x_p)$. We choose a lift $\widetilde{x_p}\in \mathbb{H}^2$ for each $x_p$. Once we have done this, we can canonically realize $\gamma_p$  as a deck transformation having a unique geodesic axis $\widetilde{\gamma}_p\subset \widetilde{\gamma}$ passing through $\widetilde{x_p}$. We define similarly the geodesic  $\widetilde{\eta}_p$  in $\mathbb{H}^2$ that is invariant under the $\eta_p$-action. They are oriented toward the attracting fixed points of $\gamma_p$ and $\eta_p$ respectively.  Let  $I(\gamma, \eta) := \{(\widetilde{\gamma}_p, \widetilde{\eta}_p)\mid p\in \gamma\sharp \eta\}$ regarded as a subset of $\mathcal{G}^\times(S)$. Then,
\[
I(\gamma, \eta)=\{(g,h)\in \mathcal{G}^\times(S)\mid g\subset \widetilde{\gamma}\text{, }\eta\subset \widetilde{\eta}\}.
\]

Recall that we have $\mathsf{g}(\xi_\rho(\widetilde{\gamma}_p ^+), \xi^\mathrm{op}_\rho( \widetilde{\gamma}_p^- ))\in \mathsf{G}$ mapping $(\mathfrak{p}, \mathfrak{p}^\mathrm{op})$ to $(\xi_\rho(\widetilde{\gamma}_p ^+), \xi^\mathrm{op}_\rho( \widetilde{\gamma}_p^-))$. This element satisfies
\[
\mathsf{g}(\xi_\rho(\widetilde{\gamma}_p ^+), \xi^\mathrm{op}_\rho( \widetilde{\gamma}_p^-))^{-1} \rho(\gamma_p) \mathsf{g}(\xi_\rho(\widetilde{\gamma}_p ^+), \xi^\mathrm{op}_\rho( \widetilde{\gamma}_p^-))\in U\cap \mathsf{H}^+.
\]
Likewise we have
\[
\mathsf{g}(\xi_\rho(\widetilde{\eta}_p ^+), \xi^\mathrm{op}_\rho( \widetilde{\eta}_p^-))^{-1}\rho(\eta_p)\mathsf{g}(\xi_\rho(\widetilde{\eta}_p ^+), \xi^\mathrm{op}_\rho( \widetilde{\eta}_p^-)) \in U\cap \mathsf{H}^+.
\]

By Lemma~\ref{lem:goldmanfunction}, we have
\[
\widehat{\mathbf{a}\circ \lambda^\mathbb{C}}(g) = \mathbf{a}^\vee,\quad \text{ and }\quad \widehat{\mathbf{a}\circ \lambda^\mathbb{C}}(g^{-1}) = w_0\mathbf{a}^\vee
\]
for $g\in U\cap \mathsf{H}^+$. It follows that for each $p\in \gamma\sharp \eta$, one can write
\[
(\sign p )\cdot \kf(\widehat{\mathbf{a}\circ\lambda^\mathbb{C}}(\rho(\gamma_p)), \widehat{\mathbf{b}\circ\lambda^\mathbb{C}}(\rho(\eta_p)))= \kf^{\mathbf{a}, \mathbf{b}}_\rho (\widetilde{\gamma}_p, \widetilde{\eta}_p).
\]
Therefore, 
\begin{align*}
\sum_{p\in \gamma\sharp \eta} (\sign p )\cdot \kf(\widehat{\mathbf{a}\circ\lambda^\mathbb{C}}(\rho(\gamma_p)), \widehat{\mathbf{b}\circ\lambda^\mathbb{C}}(\rho(\eta_p))) & = \sum_{p\in \gamma\sharp \eta}  \kf^{\mathbf{a}, \mathbf{b}}_\rho (\widetilde{\gamma}_p, \widetilde{\eta}_p)\\
&=\sum_{(g,h)\in I(\gamma, \eta)}  \kf^{\mathbf{a}, \mathbf{b}}_\rho (g,h)\\
&=\int_{\mathcal{G}^\times(S)} \kf^{\mathbf{a},\mathbf{b}}_\rho \;\der(\dirac_\gamma\times \dirac_\eta),
\end{align*}
where $\dirac_\gamma$ and $\dirac_\eta$ are the Dirac delta measures supported on the lift of the geodesics $\gamma$ and $\eta$.

\begin{lemma}\label{lem:bounded}
The function $\kf^{\mathbf{a},\mathbf{b}}_\rho:\mathcal{G}^2\to \mathbb{C}$ is continuous on $\mathcal{G}^\times$. Moreover, for any compact subset $K\subset \mathcal{G}$, the restricted function $\kf^{\mathbf{a},\mathbf{b}}_\rho:K\times K\to \mathbb{C}$ is bounded.  
\end{lemma}
\begin{proof}
By Lemma~\ref{lem:holomorphicconj}, $\mathsf{g}(\mathfrak{x},\mathfrak{y})$ can be chosen to be a holomorphic function on a neighborhood of $(\mathfrak{p}, \mathfrak{p}^\mathrm{op})$ in  $\mathfrak{T}$. Thus, $\kf^{\mathbf{a},\mathbf{b}}$ is continuous (in fact, holomorphic) on $\mathfrak{T}\times \mathfrak{T}$. Note that the map $\vec{\mathcal{G}}\to \mathcal{G}$, forgetting  the orientation, is a 2-fold covering. For each $(x,y)\in \mathcal{G}^\times$, we can find an open neighborhood $U$ of $(x,y)$ and a lift $\vec{U}$ of $U$ to $\vec{\mathcal{G}}^\times$ so that for every pair $(g,h)\in \vec{U}$, $g$ intersects $h$ positively. Then, 
\[
\kf^{\mathbf{a},\mathbf{b}}_\rho (g,h) = \kf^{\mathbf{a},\mathbf{b}}(\xi_\rho(g^+),\xi^\mathrm{op}_\rho (g^-), \xi_\rho (h^+), \xi^\mathrm{op}_\rho (h^-))
\]
on  $U$. Therefore, $\kf^{\mathbf{a},\mathbf{b}}_\rho$ is continuous on $U$. 

For the second claim,  lift $K$ to $\vec{\mathcal{G}}$ as a compact set $\vec{K}$. As we showed above, the function  $\kf^{\mathbf{a},\mathbf{b}}:\mathfrak{T}\times\mathfrak{T}\to \mathbb{C}$ is continuous. Since $\{(\xi_\rho(g^+), \xi_\rho^\mathrm{op}(g^-))\mid g\in \vec{K}\}$ is compact in $\mathfrak{T}$, $\sup_{(g,h)\in \vec{K}\times \vec{K}} \|\kf^{\mathbf{a},\mathbf{b}}_\rho(g,h)\|$ is finite. 
\end{proof}

\begin{example}\label{example} Let $\mathsf{G}_{\mathrm{split}}$ be the split real form of a complex simple adjoint Lie group $\mathsf{G}$. Fix a hyperbolic structure $X$ on $S$ and its Fuchsian holonomy representation  $\rho_X:\pi_1(S)\to \PSL_2(\bR)$. We define $\Theta_X:\mathcal{G}^\times\to \mathsf{G}_\mathrm{split}$ by
\[
\Theta_{X}(g,h):=\tau_\mathsf{G}\begin{pmatrix}
    \cos (\angle^{X}(g,h)/2) & \sin (\angle^{X}(g,h)/2) \\ -\sin (\angle^{X}(g,h)/2) & \cos(\angle^{X}(g,h)/2)
\end{pmatrix}
\]
where $\tau_\mathsf{G}:\PSL_2(\bR)\to \mathsf{G}_\mathrm{split}$ is the canonical representation.  Then, Goldman \cite{Goldman86} (also \cite{jung2025}) observed that one can write
\[
\kf^{\mathbf{a},\mathbf{b}}_{\rho_X}(g,h) = \kf(\mathbf{a}^\vee,\Ad_{\Theta_X(g,h)}\mathbf{b}^\vee).
\]
\end{example}

\begin{lemma}\label{lem:derivative}  Let $\mathsf{G}$ be a complex simple adjoint group. Let $\mathbf{a},\mathbf{b}\in \mathfrak{a}^*$, where both are fixed by the opposite involution. Let $\gamma\in \pi_1(S)\setminus\{1\}$ be a primitive element that can be represented by an oriented closed geodesic.  Let $\mu\in \ML(S)$. Define $\ell^\mathbf{a}_\gamma:(-\varepsilon,\varepsilon)\times \Hit_{\mathsf{G}_{\mathrm{split}}}(S)\to \bR$ by
    \[
    \ell^\mathbf{a}_\gamma(t,[\rho]):=\mathbf{a}(\lambda(\Lambda^{t\mu^\mathbf{b}}\rho(\gamma))).
    \]
    Then, $\ell^\mathbf{a}_\gamma$ is a real analytic function and 
    \[
    \left.\frac{\partial \ell^\mathbf{a}_\gamma}{\partial t}\right|_{t=0}=\int_{\mathcal{G}^\times(S)}\kf^{\mathbf{a},\mathbf{b}}_\rho\;\der(\dirac_\gamma\times \mu)
    \]
    is a real analytic function on $\Hit_{\mathsf{G}_\mathrm{split}}(S)$.
    \end{lemma}
\begin{proof}

Regard $\mathbf{a}$, and $\mathbf{b}$ as the complexified linear functionals belonging to $\mathfrak{h}^*$. Due to Lemma~\ref{lem:domain}, we can find a disk $D\subset\mathbb{C}$ centered at $0$ and a neighborhood $W$ of $\rho$ such that $\rho_z:=\Lambda^{z\mu^\mathbf{b}}(\rho)$ and $\rho_{n,z}:=\Lambda^{z\mu_n^\mathbf{b}}(\rho)$ are defined for all $z\in D$, all $\rho\in W$ and all large enough $n$. Define the function $f:D\times W\to \mathbb{C}$ by $f(z,\rho) := \mathbf{a}(\lambda^\mathbb{C}(\rho_z(\gamma)))$. We approximate $\mu$ by a sequence $\{\mu_n=(\eta_n, s_n)\}$ of weighted simple closed curves. Let $f_n(z,\rho) := \mathbf{a}(\lambda^\mathbb{C}(\rho_{n,z}(\gamma)))$. By Corollary~\ref{cor:holomorphic} and Lemma~\ref{lem:holomorphicconj}, both $f$ and $f_n$ are holomorphic. Since $\ell^\mathbf{a}_\gamma(t,\rho)$ is the restriction of $f$ to real points contained in $D\times W$, we conclude that $\ell^\mathbf{a}_\gamma$ is real analytic.

We now compute $\frac{\partial f_n}{\partial z}$ and $\frac{\partial f}{\partial z}$.  Write $z=x+ y\ui$. As we discussed above, we have 
\begin{align*}
    \frac{\partial f_n}{\partial x}&=\int_{\mathcal{G}^\times(S)} \kf^{\mathbf{a},\mathbf{b}}_{\rho_{n,z}}
    \;\der(\dirac_\gamma\times \mu_n),\\
    \frac{\partial f_n}{\partial y}&=\ui\int_{\mathcal{G}^\times(S)} \kf^{\mathbf{a},\mathbf{b}}_{\rho_{n,z}} \;\der(\dirac_\gamma\times \mu_n).
\end{align*}
Therefore,
\[
\frac{\partial f_n}{\partial z} =\frac{1}{2}\left(\frac{\partial f_n}{\partial x}-\ui  \frac{\partial f_n}{\partial y}\right)= \frac{\partial f_n}{\partial x}=\int_{\mathcal{G}^\times(S)} \kf^{\mathbf{a},\mathbf{b}}_{\rho_{n,z}} \;\der(\dirac_{\gamma}\times \mu_n). 
\]

By Theorem~\ref{thm:convergence},  $\{f_n\}$ converges to $f$ uniformly on compact subsets of $D\times W$. Since $f_n$ are holomorphic, $\{\frac{\partial f_n}{\partial z}\}$ converges to $\frac{\partial f}{\partial z}$ uniformly on compact subsets as well. It remains to show that for each $(z,\rho)\in D\times W$, we can find a subsequence $\{n_i\}$  such that
\[
\int_{\mathcal{G}^\times(S)} \kf^{\mathbf{a},\mathbf{b}}_{\rho_{n_i,z}} \;\der(\dirac_{\gamma}\times \mu_{n_i})\to\int_{\mathcal{G}^\times(S)} \kf^{\mathbf{a},\mathbf{b}}_{\rho_{z}} \;\der(\dirac_{\gamma}\times \mu)
\]
as $i\to\infty$.

We rewrite the integral in a slightly different form. Pick a point $p\in \gamma\setminus \bigcup\eta_n$ that is not a self-intersection point. Let $\widetilde{\gamma}$ be a complete geodesic lift of $\gamma$. For two consecutive lifted points $p_1$ and $p_2$ of $p$ lying on  $\widetilde{\gamma}$, let $\widetilde{\gamma}_0\subset \widetilde{\gamma}$ be an open geodesic arc joining $p_1$ and $p_2$. Define
\[
\mathcal{Y}_\gamma :=\{g\in \mathcal{G}\mid g\text{ intersects }\widetilde{\gamma}_0\text{ transversely} \}
\]
be a relatively compact open set in $\mathcal{G}$. Then,  we can write
\[
\int_{\mathcal{G}^\times(S)} \kf^{\mathbf{a},\mathbf{b}}_{\rho_{n,z}} \;\der(\dirac_\gamma\times \mu_n)=\int_{\{\widetilde{\gamma}\}\times\mathcal{Y}_\gamma} \kf^{\mathbf{a},\mathbf{b}}_{\rho_{n,z}} \; \der (\dirac_\gamma\times \mu_n)=\int_{\mathcal{Y}_\gamma} \kf^{\mathbf{a},\mathbf{b}}_{\rho_{n,z}}(\widetilde{\gamma}, g) \; \der \mu_n(g).
\]

Now, we appeal to Lemma~\ref{lem:continuity}. To this end, we first need to find a compact set $K\subset \mathcal{Y}_\gamma$ containing $\mathcal{Y}_\gamma \cap \supp(\mu)$ and $\mathcal{Y}_\gamma \cap \supp(\mu_n)$ for infinitely many $n$. Then  we claim that, for each $(z,\rho)$, the sequence of functions $\{g\mapsto\kf^{\mathbf{a},\mathbf{b}}_{\rho_{n,z}}(\widetilde{\gamma},g)\}$ defined on $\mathcal{Y}_\gamma$ has a subsequence that converges uniformly on $K$ to $g\mapsto \kf^{\mathbf{a},\mathbf{b}}_{\rho_z}(\widetilde{\gamma},g)$.

By Lemma~\ref{lem:limit}, passing to a subsequence if necessary, we may assume that the sequence $\{|\mu_n|\}$ of underlying maximal laminations converges to the maximal geodesic lamination $|\mu|$ containing $\supp(\mu)$ in the Hausdorff topology. Let $\mathcal{T}$ be a train track carrying $\mu$ and let $\widetilde{\mathcal{T}}$ be its lift. We may choose $\mathcal{T}$ small enough so that $\widetilde{\mathcal{T}}\cap \widetilde{\gamma}_0$ is contained in a proper compact geodesic subarc $ \widetilde{\gamma}_0'$ of $\widetilde{\gamma}_0$ and each connected component $\widetilde{\mathcal{T}}\cap \widetilde{\gamma}_0$ is a tie of $\widetilde{\mathcal{T}}$. Define 
\[
K:=\{g\in \mathcal{G}\mid g\text{ is carried by } \widetilde{\mathcal{T}}\text{ and  intersects } \widetilde{\gamma}_0'\text{ transversely}\}.
\]
Since $|\mu_n|$ converges to $|\mu|$ we can find an integer $N$ such that $\widetilde{\mu_n}\subset \widetilde{\mathcal{T}}$ for all $n>N$, where $\widetilde{\mu_n}$ is the lift of $|\mu_n|$. Notice that $K$ is compact and it contains all $\supp (\mu_n)\cap \mathcal{Y}_\gamma$, $n>N$, and $\supp(\mu)\cap \mathcal{Y}_\gamma$.

Now we show that the sequence of functions $\{g\mapsto\kf^{\mathbf{a},\mathbf{b}}_{\rho_{n,z}}(\widetilde{\gamma},g)\}$  converges uniformly on $K$ to $g\mapsto \kf^{\mathbf{a},\mathbf{b}}_{\rho_z}(\widetilde{\gamma},g)$.  

Since $\{\rho_{n,z}\}$ converges to $\rho_z$, \cite[Theorem~6.1]{BCLS2015} shows that
\begin{align*}
Q&:=\overline{\{(\xi_{\rho_{n,z}}(\widetilde{\gamma}^+),\xi^\mathrm{op}_{\rho_{n,z}}(\widetilde{\gamma}^-))\mid n\in \mathbb{N}\}}\\
R&:=\overline{\{(\xi_{\rho_{n,z}}(\eta^+),\xi^\mathrm{op}_{\rho_{n,z}}(\eta^-))\mid \eta\in K,\,n\in \mathbb{N}\}}
\end{align*}
are compact subsets in $\mathfrak{T}$. Here all geodesics are oriented so that $\widetilde{\gamma}$ intersects $\eta$ positively. Equip $\mathcal{F}_\Theta$ with a metric, say $\der_\mathcal{F}$, which induces a metric on $Q\times R$. Given $\epsilon>0$, there exists $\delta$ such that whenever $(\mathfrak{x},\mathfrak{y},\mathfrak{z},\mathfrak{w}), (\mathfrak{x}',\mathfrak{y}',\mathfrak{z}',\mathfrak{w}')$ in $Q\times R$ satisfy 
\[
\der((\mathfrak{x},\mathfrak{y},\mathfrak{z},\mathfrak{w}), (\mathfrak{x}',\mathfrak{y}',\mathfrak{z}',\mathfrak{w}')):=\der_\mathcal{F}(\mathfrak{x},\mathfrak{x}')+\der_\mathcal{F}(\mathfrak{y},\mathfrak{y}')+\der_\mathcal{F}(\mathfrak{z},\mathfrak{z}')+\der_\mathcal{F}(\mathfrak{w},\mathfrak{w}')<\delta,
\]
we have $|\kf^{\mathbf{a},\mathbf{b}}(\mathfrak{x},\mathfrak{y},\mathfrak{z},\mathfrak{w})-\kf^{\mathbf{a},\mathbf{b}}(\mathfrak{x}',\mathfrak{y}',\mathfrak{z}',\mathfrak{w}')|<\epsilon$. Again, by   \cite[Theorem~6.1(4)]{BCLS2015}, we can find $N$ such that 
\[
\sup_{x\in \partial_\infty\pi_1(S)} \der_{\mathcal{F}}(\xi_{\rho_{n,z}} (x),\xi_{\rho_z}(x))<\frac{\delta}{4},\qquad \sup_{x\in \partial_\infty\pi_1(S)} \der_{\mathcal{F}}(\xi^\mathrm{op}_{\rho_{n,z}} (x),\xi^\mathrm{op}_{\rho_z}(x))<\frac{\delta}{4}
\]
hold whenever $n>N$. It follows that, for any $n>N$ and any $\eta\in K$,
\[
|\kf^{\mathbf{a},\mathbf{b}}_{\rho_{n,z}}(\widetilde{\gamma}, \eta)-\kf^{\mathbf{a},\mathbf{b}}_{\rho_{z}}(\widetilde{\gamma}, \eta)|<\epsilon
\]
holds. This establishes the claim.
\end{proof}

\subsection{The Poisson Bracket} As an application, we first obtain a generalized version of the Goldman product formula \cite[Theorem 3.5]{Goldman86} for Hamiltonian vector fields associated to measured laminations. 

Let $\mathbf{a}\in \mathfrak{a}^*$ be a non-trivial element invariant under the opposite involution,  and let $\mu$ be a measured lamination. We consider the vector field associated to the cataclysm deformation $\Lambda^{t\mu^\mathbf{a}}$:
\[
\mathcal{H}^\mathbf{a}_\mu :=\frac{\der }{\der t} \Lambda^{t\mu^\mathbf{a}}.
\]

If $\mu$ is a weighted simple closed curve, it is known that $\mathcal{H}^{\mathbf{a}}_\mu$ is a Hamiltonian vector field with respect to the Atiyah-Bott-Goldman symplectic form $\omega_{\mathrm{ABG}}$ \cite[Theorems 4.5, 4.7]{Goldman86}. In other words, 
\[
\der \ell^{\mathbf{a}}_\mu  = \omega_{\mathrm{ABG}}(\mathcal{H}^\mathbf{a}_\mu, -).
\]

For a general measured lamination $\mu$, choose a sequence $\{\mu_i\}$ of weighted simple closed curves that converges weakly to  $\mu$. By Theorem~\ref{thm:convergence}, the sequence of flows $\{ \Lambda^{t\mu_i^{\mathbf{a}}}\}$ and the sequence of associated vector fields $\{\mathcal{H}^{\mathbf{a}}_{\mu_i}\}$ converge uniformly to $\Lambda^{t\mu^\mathbf{a}}$ and to $\mathcal{H}^{\mathbf{a}}_{\mu}$ on compact sets. Since $\Hit_{\mathsf{G}_\mathrm{split}}(S)$ is contractible, we know that the limiting 1-form $\omega_{\mathrm{ABG}}(\mathcal{H}^{\mathbf{a}}_{\mu}, -)$ is exact. It follows that  there is a smooth function, denoted by $\ell^\mathbf{a}_\mu$, satisfying 
\begin{equation}\label{eq:hamiltonianpotential}
\der\ell^\mathbf{a}_\mu = \omega_\mathrm{ABG}(\mathcal{H}^\mathbf{a}_\mu,-).
\end{equation}

We may call $\ell^{\mathbf{a}}_\mu$, defined up to an additive constant, a length function of $\mu$. At this stage, we are merely asserting that there exists a function $\ell^\mathbf{a}_\mu$ satisfying (\ref{eq:hamiltonianpotential}). We are neither giving an explicit definition of $\ell^\mathbf{a}_\mu$, nor claiming that $\lim_{i\to \infty}\ell^\mathbf{a}_{\mu_i}$ exists. What we can say is that there is a sequence $c_i$ of real numbers such that $\{\ell^\mathbf{a}_{\mu_i}+c_i\}$ converges uniformly on compact sets to $\ell^\mathbf{a}_\mu$. Nonetheless, for any function that satisfies (\ref{eq:hamiltonianpotential}), we are able to compute the Poisson bracket as follows
\[
\{\ell^\mathbf{a}_\mu, \ell^\mathbf{b}_{\tau}\} = \lim_{i\to \infty}\{\ell^\mathbf{a}_{\mu_i}, \ell^\mathbf{b}_{\tau}\}=\lim_{i\to \infty}\int_{\mathcal{G}^\times(S)}\kf^{\mathbf{a},\mathbf{b}}_\rho \;\der(\mu_i \times \tau).
\]
By evaluating this limit, we obtain the following result. 

\begin{theorem} Let $S$ be a closed oriented surface of genus at least two. Let $\mathsf{G}_\mathrm{split}$ be the split real form of a complex simple adjoint group $\mathsf{G}$. Let $\mathbf{a},\mathbf{b}$ be non-trivial elements of $\mathfrak{a}^*$ that are fixed by the opposite involution. Then, for two measured laminations $\mu$ and $\tau$ on $S$, and for $\rho\in \Hit_{\mathsf{G}_\mathrm{split}}(S)$, we have  
\[
\{\ell^\mathbf{a}_\mu, \ell^\mathbf{b}_{\tau}\}(\rho)=\omega_\mathrm{ABG}(\mathcal{H}^\mathbf{a}_\mu , \mathcal{H}^\mathbf{b}_\tau)|_\rho = \int_{\mathcal{G}^\times (S)}\kf^{\mathbf{a},\mathbf{b}}_\rho \;\der(\mu \times \tau).
\]    
\end{theorem}

\begin{proof}
Let $D$ be the interior of a fundamental domain for the surface $S$ such that the boundary of the closure of $D$ does not contain intersection points $x\cap y$ for $x\in \supp(\mu)$, $y\in \supp(\tau)$. Define a measurable function $\breve{\kf}^{\mathbf{a},\mathbf{b}}_\rho$ on $\mathcal{G}\times\mathcal{G}$ by
\[
\breve{\kf}^{\mathbf{a},\mathbf{b}}_\rho (x,y) = \begin{cases}
    \kf^{\mathbf{a},\mathbf{b}}_\rho (x,y) & \text{if }(x,y)\in \mathcal{G}^\times\text{ and } x\cap y \text{ belongs to }D\\
    0 & \text{ otherwise }
\end{cases}.
\]
Let
\[
\mathcal{D} = \{g\in \mathcal{G}\mid g\cap D\ne\emptyset\}.
\]
Then, $\mathcal{D}$ is a relatively compact open subset of $\mathcal{G}$. Observe that
\[
\int_{\mathcal{G}^\times(S)} \kf^{\mathbf{a},\mathbf{b}}_\rho\;\der(\mu\times \tau) =\int_{\mathcal{D}\times \mathcal{D}} \breve{\kf}^{\mathbf{a},\mathbf{b}}_\rho \;\der(\mu \times \tau)=\int_{\mathcal{D}}\left(\int_{\mathcal{D}}\breve{\kf}^{\mathbf{a},\mathbf{b}}_\rho(x,y) \;\der\mu(x) \right)\der \tau(y).
\]
Similarly,
\[
\int_{\mathcal{G}^\times(S)}\kf^{\mathbf{a},\mathbf{b}}_\rho \;\der\mu_n \;\der\tau= \int_\mathcal{D}\left( \int_\mathcal{D} \breve{\kf}^{\mathbf{a},\mathbf{b}}_\rho (x,y) \;\der\mu_n(x) \right)\der\tau(y).
\]

For each $y\in \mathcal{D}$, let $y_0=y\cap D$ and let
\[
\mathcal{X}_y=\{g\in \mathcal{D}\mid g\text{ intersects }y_0\text{ transversely}\}.
\]
As we did in the proof of Lemma~\ref{lem:derivative}, choose a compact set $K\subset \mathcal{X}_y$ and an integer $N$ such that 
\[
\left\{\left.x\in \left(\bigcup_{n>N}\supp(\mu_n)\right)\cup \supp(\mu)\,\right|\,x\text{ intersects } y_0 \text{ transversely}\right\} \subset K.
\]

Note that  $\kf^{\mathbf{a},\mathbf{b}}_\rho (x,y)=\breve{\kf}^{\mathbf{a},\mathbf{b}}_\rho(x,y)$ for  $x\in \mathcal{X}_y$. By Lemma~\ref{lem:bounded}, the function $x\mapsto \breve{\kf}^{\mathbf{a},\mathbf{b}}_\rho(x,y)$ is continuous on $\mathcal{X}_y$. We can also find a compactly supported continuous function $s:\mathcal{G}\to \mathbb{R}$ such that $s=1$ on $K$ and $s=0$ on $\mathcal{G}\setminus \mathcal{X}_y$.  Since $\mu_n\to \mu$ weakly, we know that 
\[
\int_{\mathcal{D}}\breve{\kf}^{\mathbf{a},\mathbf{b}}_\rho(x,y)\;\der\mu_n(x)=\int_{\mathcal{G}}s(x)\cdot \breve{\kf}^{\mathbf{a},\mathbf{b}}_\rho(x,y)\;\der\mu_n(x)
\]
converges to 
\[
 \int_{\mathcal{G}}s(x)\cdot\breve{\kf}^{\mathbf{a},\mathbf{b}}_\rho(x,y)\;\der \mu(x)=\int_{\mathcal{D}}\breve{\kf}^{\mathbf{a},\mathbf{b}}_\rho(x,y)\;\der \mu(x)
\]
as $n\to \infty$. Moreover, by Lemma~\ref{lem:bounded},
\[
\left|\int_{\mathcal{D}} \breve{\kf}^{\mathbf{a},\mathbf{b}}_\rho(x,y)\;\der\mu_n (x)\right|<\mu_n(\mathcal{D}) \cdot \sup_{(g,h)\in \overline{\mathcal{D}}\times \overline{\mathcal{D}}} \kf^{\mathbf{a},\mathbf{b}}_\rho(g,h)
\]
is uniformly bounded. Therefore, by using the dominated convergence theorem, we obtain
\begin{align*}
\lim_{n\to\infty}\int_{\mathcal{G}^\times(S)}\kf^{\mathbf{a},\mathbf{b}}_\rho \;\der(\mu_n \times \tau)&=\lim_{n\to\infty}\int_\mathcal{D}\left( \int_\mathcal{D} \breve{\kf}^{\mathbf{a},\mathbf{b}}_\rho (x,y) \;\der\mu_n(x) \right)\der\tau(y)\\
&=\int_\mathcal{D}\left( \int_\mathcal{D} \breve{\kf}^{\mathbf{a},\mathbf{b}}_\rho (x,y) \;\der\mu(x) \right)\der\tau(y)\\
&=\int_{\mathcal{G}^\times(S)}\kf^{\mathbf{a},\mathbf{b}}_\rho \;\der(\mu \times\tau). \qedhere
\end{align*}
\end{proof}

\subsection{Strongly dense representations} 
This subsection is devoted to proving Theorem~\ref{thm:main}.  We will  assume that $w_0=-1$ for the last of this section. Equivalently, the opposite involution is the identity. 

\begin{lemma}\label{lem:goodhyperbolicstr}
    Let $\mu\in\ML(S)$ be a measured lamination whose support fills $S$ and let $\mathbf{a},\mathbf{b}\in \mathfrak{a}^*$ be such that $\kf(\mathbf{a}^\vee,\mathbf{b}^\vee)\ne0$. Then there exists a hyperbolic structure $X$ with Fuchsian holonomy $\rho:\pi_1(S)\to \PSL_2(\bR)$ such that
    \[
    \int_{\mathcal{G}^\times(S)}\kf^{\mathbf{a},\mathbf{b}}_{\rho}\;\der(\dirac_\gamma\times\mu)\ne 0
    \]
    for all  closed geodesics $\gamma$.
\end{lemma}
\begin{proof}By Lemma~\ref{lem:derivative},
\[
\Phi_\gamma:  X\mapsto  \int_{\mathcal{G}^\times(S)}\kf^{\mathbf{a},\mathbf{b}}_{\rho_X}\;\der(\dirac_\gamma\times \mu)
\]
is a real analytic function on the Teichm\"uller space. Thus, it suffices to show the existence of a hyperbolic structure $X$  such that $\Phi_\gamma(X)\ne 0 $ for each closed parametrized geodesic $\gamma$.

    To this end, recall Example~\ref{example}:
    \[
    \kf^{\mathbf{a},\mathbf{b}}_{\rho_X}(g,h)= \kf(\mathbf{a}^\vee,\Ad_{\Theta_X(g,h)}\mathbf{b}^\vee).
\]
    We start with the hyperbolic structure $X_0$ and apply sufficiently long earthquake  $E_{t\mu}:S\to S$ along $\mu$. Let $X_n = E_{n\mu}(X_0)$.  Then, as observed in \cite{jung2025}, given $\epsilon>0$ there exists $N$ such that whenever $n>N$,  we have $\angle^{X_n}(\gamma, g)<\epsilon$ for all leaves $g$ of $\supp(\mu)$ intersecting $\gamma$. Therefore, $\Theta_{X_n}(\gamma,g)\to \mathrm{Id}$ as $n\to \infty$. It follows that $\kf(\mathbf{a}^\vee,\Ad_{\Theta_{X_n}(\gamma,g)}\mathbf{b}^\vee)\ne 0$ for all large enough $n$. 
\end{proof}

We are ready to show Theorem~\ref{thm:main}. For reader's convenience, we restate the result here. 

\begin{theorem} Let $\mathsf{G}_{\mathrm{split}}$ be a split real form of a complex simple Lie group of adjoint type $\mathsf{G}$. Assume that the longest element of the Weyl group of $\mathsf{G}$ is $-1$.  Fix a symmetric generating set for $\pi_1(S)$ and let $|\cdot|$ be the associated word metric. There exist a sequence of $\mathsf{G}_{\mathrm{split}}$-Hitchin representations $\{\rho_n\}$ and constants $A,B>0$ that satisfy the following properties: 
 \begin{enumerate}
      \item each $\rho_n$ is not strongly dense,
     \item $\langle\rho_n(x),\rho_n(y)\rangle$ is Zariski dense for all non-commuting  $x,y\in \pi_1(S)$ with $|x|,|y|<An-B$,
     \item $\{\rho_n\}$ converges to a strongly dense $\mathsf{G}_{\mathrm{split}}$-Hitchin representation.
 \end{enumerate}
\end{theorem}

\begin{proof} We mainly follow the argument of \cite{jung2025}.
According to \cite{sambarino}, there exist finitely many codimension-1 subspaces $\mathfrak{h}_1,\cdots,\mathfrak{h}_n$ in $\mathfrak{a}$ such that if $\rho|\langle x,y\rangle$ is not Zariski dense then, for some $i\in \{1,2,\cdots, n\}$, we have $\lambda(\rho(z))\in \mathfrak{h}_i$ for all $z\in \langle x,y\rangle$. That is, a representation $\rho$ is strongly dense if it satisfies $\ell^{\mathbf{a}_i}_\gamma(\rho)\ne 0$ for all $\gamma\in \pi_1(S)\setminus\{1\}$ and all $i=1,2,\cdots, n$.

Choose $\mathbf{b}\in\mathfrak{a}^*$ such that $\mathbf{b}(\mathfrak{h}_i)\ne0$ for all $i=1,2,\cdots,n$.  If we realize $\mathfrak{h}_i$ as the kernel of $\mathbf{a}_i\in \mathfrak{a}^*$, the condition $\mathbf{b}(\mathfrak{h}_i)\ne0$ means that $\kf(\mathbf{a}_i^\vee,\mathbf{b}^\vee)\ne0$. Since we are assuming $w_0=-1$,  $\mathbf{a}_i$ and $\mathbf{b}$ automatically satisfy the conditions $\mathbf{a}_i^\vee=-w_0\mathbf{a}_i^\vee$ and $\mathbf{b} ^\vee = -w_0 \mathbf{b}^\vee$.

Fix a pseudo-Anosov diffeomorphism $\phi:S\to S$ with stable invariant measured lamination $\mu^{\mathrm{st}}$. Since $\supp(\mu^{\mathrm{st}})$ is filling, Lemma~\ref{lem:goodhyperbolicstr} implies that there exists a hyperbolic structure $X_0$ with the Fuchsian holonomy $\rho_0$ such that 
\[
\int_{\mathcal{G}^\times(S)} \kf^{\mathbf{a}_i,\mathbf{b}}_{\rho_0} \;\der(\dirac_\gamma\times\mu^{\mathrm{st}})\ne 0
\]
for all closed geodesics $\gamma$ and all $i$.

Let $\rho_t=\Lambda^{t (\mu^{\mathrm{st}})^\mathbf{b}}(\rho_0)$. By Lemma~\ref{lem:derivative}, 
\[
\frac{\der }{\der t}\ell^{\mathbf{a}_i}_\gamma(\rho_t)= \int_{\mathcal{G}^\times(S)} \kf^{\mathbf{a}_i,\mathbf{b}}_{\rho_0}\; \der(\dirac_\gamma\times \mu^{\mathrm{st}})\ne 0
\]
for all $i=1,2,\cdots,n$ and all closed geodesics $\gamma$. In particular,  $t\mapsto \ell^{\mathbf{a}_i}_\gamma(\rho_t)$, $i=1,2,\cdots,n$, are non-constant and real-analytic. It follows that their zero sets 
\[
\mathcal{Z}(i,\gamma) := \{t\in \bR\mid \ell^{\mathbf{a}_i}_\gamma(\rho_t)=0\}
\]
are discrete for all $i$ and all $\gamma\in\pi_1(S)\setminus\{1\}$. This also implies that 
\[
\bigcup_{\substack{\gamma\in\pi_1(S)\setminus\{1\} \\ i=1,2,\cdots,n}}\mathcal{Z}(i,\gamma)
\]
is nowhere dense. Therefore, one can find a real number $\epsilon>0$ such that $\ell^{\mathbf{a}_i}_\gamma(\rho_\epsilon) \ne 0$ for all closed geodesics $\gamma$ and all $i$. Then,  $\rho_\epsilon$ is strongly dense.

Let $c$ be a separating simple closed geodesic and let $S_1$ and $S_2$ be the completions of connected components of $S\setminus c$. We regard $c$ as a weighted simple closed curve $(c,1)\in \mathcal{S}\times \mathbb{R}_+$ belonging to $\ML(S)$. Since the projective class of $\mu^{\mathrm{st}}$ is the unique attracting fixed point of the $\phi$-action on $\mathbb{P}\ML(S)$, one can find a sequence $\{k_n\}$ of real numbers such that $\{k_n \phi ^nc\}$ converges weakly to $\epsilon\mu^{\mathrm{st}}$ in $\ML(S)$.  Lemma~\ref{lem:convergetoo} and Theorem~\ref{thm:convergence} show that the sequence $\eta_n:=\Lambda^{(k_n\phi^nc)^\mathbf{b}}(\rho_0)$ converges to the strongly dense representation $\rho_\epsilon=\Lambda^{(\epsilon\mu^{\mathrm{st}})^\mathbf{b}}(\rho_0)$. It is clear that each $\eta_n$ is not strongly dense since $\eta_n$ restricted to the subgroup $\phi^n_*\pi_1(S_1)$ is Fuchsian. Therefore, the sequence $\{\eta_n\}$ satisfies conditions (i) and (iii) of the theorem.

The above argument shows that, for each $n>0$, if $\langle x,y\rangle$ contains a non-trivial element $z$ with $i(z,\phi^n c) \ne 0$, then $\langle \eta_n(x),\eta_n(y)\rangle$ is Zariski dense. Let
\[
K_n=\{z\in \pi_1(S)\setminus\{1\}\mid i(z, \phi^n c)= 0\}.
\]
We claim that, for a fixed word metric $|\cdot|$ on $\pi_1(S)$, the quantity 
\[
\inf_{z\in K_n} |z|
\]
grows linearly in $n$. 

To this end, equip $S$ with a hyperbolic structure $X$ and let $\ell_X$ be the length function with respect to this hyperbolic structure. Isotope $\phi$ so that $\phi^n(c)$ is a closed geodesic with respect to the hyperbolic structure $X$ on $S$. Observe that the condition $z\in K_n$ implies that $z\subset \phi^n(S_1)$ or $z\subset \phi^n(S_2)$. Without loss of generality, we may assume that $z\subset\phi^n(S_1)$.  Then, we have
\[
|z|\ge \frac{1}{C_1} \ell_X(z)-C_2
\]
for constants $C_1>1$, and $C_2>0$ depending only on $X$. Let $z'\subset \phi^n(S_1)$ be a shortest closed geodesic contained in $\phi^n(S_1)$. Note that $z'$ is necessarily simple and $z'\in K_n$. We also know that
\[
\ell_X(z)\ge \ell_X(z')\ge C_3\cdot i(z', c)
\]
for some constant $C_3>0$ depending only on $X$. It follows that
\[
|z|\ge\frac{C_3}{C_1} i(z', c) -C_2.
\] 
By \cite[Lemma 2.1]{minski}, we have
\[
|z| \ge \frac{C_3}{2C_1}\der_{\mathscr{C}}(z',c)-\frac{C_3}{2C_1}-C_2
\] 
where $\der_{\mathscr{C}}$ is the distance in the curve complex of $S$. Since $\phi$ is pseudo-Anosov, \cite[Proposition 4.6]{minski} shows that 
\[
\der_\mathscr{C} (c, \phi^n (c))\ge C_4\cdot |n|
\]
for some constant $C_4>0$. Moreover, by definition of $K_n$, we know that $\der_{\mathscr{C}}(z',\phi^nc) \le 1$.  Hence,
\[
\der_{\mathscr{C}}(z',c) \ge     \der_{\mathscr{C}}(\phi^nc,c)- \der_{\mathscr{C}}(z', \phi^nc)\ge C_4\cdot n-1.
\]
Consequently, for any $z\in K_n$, we have
\[
|z|\ge \frac{C_3C_4}{2C_1}n-\frac{C_3}{C_1}-C_2.
\]
Since $C_1$, $C_2$, $C_3$, and $C_4$ do not depend on $z$, we conclude that the sequence of representations $\{\eta_n\}$ satisfies (ii). 
\end{proof}

We believe that the set of strongly dense representations has empty boundary. To achieve this, it suffice to prove that for any $\mathsf{G}_\mathrm{split}$-Hitchin representation $\rho$, any filling measured lamination $\mu$ and any closed geodesic $\gamma$, the real analytic function $t\mapsto \ell^{\mathbf{a}}_\gamma(\Lambda^{t\mu^\mathbf{b}}\rho)$ is non-constant if $\kf(\mathbf{a}^\vee,\mathbf{b}^\vee)\ne 0$. We leave this as a closing question. 
\begin{question}
    Let $\mathbf{a},\mathbf{b}\in \mathfrak{a}^*$ be linear functionals that are invariant under the opposite involution such that $\kf(\mathbf{a}^\vee,\mathbf{b}^\vee) \ne 0$. Then, for any filling measured lamination $\mu$ and any closed geodesic $\gamma$, is the function $t\mapsto \ell^{\mathbf{a}}_\gamma (\Lambda^{t\mu^\mathbf{b}}\rho)$ non-constant? 
\end{question}

    \bibliographystyle{alpha}
    \bibliography{ref.bib}

\begin{thebibliography}{GGKW17}

\bibitem[BCLS15]{BCLS2015}
Martin Bridgeman, Richard Canary, Fran\c{c}ois Labourie, and Andres Sambarino.
\newblock The pressure metric for {A}nosov representations.
\newblock {\em Geom. Funct. Anal.}, 25(4):1089--1179, 2015.

\bibitem[BD17]{BD}
Francis Bonahon and Guillaume Dreyer.
\newblock Hitchin characters and geodesic laminations.
\newblock {\em Acta Math.}, 218(2):201--295, 2017.

\bibitem[BGGT12]{tao}
Emmanuel Breuillard, Ben Green, Robert Guralnick, and Terence Tao.
\newblock Strongly dense free subgroups of semisimple algebraic groups.
\newblock {\em Israel J. Math.}, 192(1):347--379, 2012.

\bibitem[BGL24]{breuillard}
Emmanuel Breuillard, Robert Guralnick, and Michael Larsen.
\newblock Strongly dense free subgroups of semisimple algebraic groups {II}.
\newblock {\em J. Algebra}, 656:143--169, 2024.

\bibitem[BL25]{bridgeman2025}
Martin Bridgeman and François Labourie.
\newblock Ghost polygons, {P}oisson bracket and convexity.
\newblock {\em Preprint, arXiv:2307.04380}, 2025.

\bibitem[Bon88]{bonahoncurrent}
Francis Bonahon.
\newblock The geometry of {T}eichm\"uller space via geodesic currents.
\newblock {\em Invent. Math.}, 92(1):139--162, 1988.

\bibitem[Bon96]{bonahon96}
Francis Bonahon.
\newblock Shearing hyperbolic surfaces, bending pleated surfaces and {T}hurston's symplectic form.
\newblock {\em Ann. Fac. Sci. Toulouse Math. (6)}, 5(2):233--297, 1996.

\bibitem[BPS19]{BPS}
Jairo Bochi, Rafael Potrie, and Andr\'es Sambarino.
\newblock Anosov representations and dominated splittings.
\newblock {\em J. Eur. Math. Soc. (JEMS)}, 21(11):3343--3414, 2019.

\bibitem[EM06]{EM06}
D.~B.~A. Epstein and A.~Marden.
\newblock Convex hulls in hyperbolic space, a theorem of {S}ullivan, and measured pleated surfaces [mr0903852].
\newblock In {\em Fundamentals of hyperbolic geometry: selected expositions}, volume 328 of {\em London Math. Soc. Lecture Note Ser.}, pages 117--266. Cambridge Univ. Press, Cambridge, 2006.

\bibitem[GGKW17]{gueritaud}
Fran\c{c}ois Gu\'{e}ritaud, Olivier Guichard, Fanny Kassel, and Anna Wienhard.
\newblock Anosov representations and proper actions.
\newblock {\em Geom. Topol.}, 21(1):485--584, 2017.

\bibitem[GLW26]{guichard_wienhard_labourie}
Olivier Guichard, Fran\c{c}ois Labourie, and Anna Wienhard.
\newblock Positivity and representations of surface groups.
\newblock {\em Forum Math. Pi}, 14:Paper No. e6, 2026.

\bibitem[Gol86]{Goldman86}
William~M. Goldman.
\newblock Invariant functions on {L}ie groups and {H}amiltonian flows of surface group representations.
\newblock {\em Invent. Math.}, 85(2):263--302, 1986.

\bibitem[GW12]{GW12}
Olivier Guichard and Anna Wienhard.
\newblock Anosov representations: domains of discontinuity and applications.
\newblock {\em Invent. Math.}, 190(2):357--438, 2012.

\bibitem[Hit92]{hitchin92}
N.~J. Hitchin.
\newblock Lie groups and {T}eichm\"{u}ller space.
\newblock {\em Topology}, 31(3):449--473, 1992.

\bibitem[Jun25]{jung2025}
Hongtaek Jung.
\newblock Generic properties of hitchin representations.
\newblock {\em Preprint, arXiv:2407.08487}, 2025.

\bibitem[Ker83]{Kerckhoff}
Steven~P. Kerckhoff.
\newblock The {N}ielsen realization problem.
\newblock {\em Ann. of Math. (2)}, 117(2):235--265, 1983.

\bibitem[Ker85]{kerckhoff_analytic}
Steven~P. Kerckhoff.
\newblock Earthquakes are analytic.
\newblock {\em Comment. Math. Helv.}, 60(1):17--30, 1985.

\bibitem[KLP17]{KLP3}
Michael Kapovich, Bernhard Leeb, and Joan Porti.
\newblock Anosov subgroups: dynamical and geometric characterizations.
\newblock {\em Eur. J. Math.}, 3(4):808--898, 2017.

\bibitem[Kou98]{kourouniotis}
Christos Kourouniotis.
\newblock On the continuity of bending.
\newblock In {\em The {E}pstein birthday schrift}, volume~1 of {\em Geom. Topol. Monogr.}, pages 317--334. Geom. Topol. Publ., Coventry, 1998.

\bibitem[Lab06]{L06}
Fran\c{c}ois Labourie.
\newblock Anosov flows, surface groups and curves in projective space.
\newblock {\em Invent. Math.}, 165:51--114, 2006.

\bibitem[Lee26]{lee}
Ricky Lee.
\newblock Strongly dense representations of hyperbolic 3-manifold groups.
\newblock {\em Proc. Amer. Math. Soc.}, 154(3):1311--1323, 2026.

\bibitem[LRW22]{long}
D.~D. Long, A.~W. Reid, and M.~Wolff.
\newblock Most {H}itchin representations are strongly dense, 2022.
\newblock arXiv 2202.09306, To appear in Michigan Math. J.

\bibitem[McM98]{mcmullen}
Curtis~T. McMullen.
\newblock Complex earthquakes and {T}eichm\"uller theory.
\newblock {\em J. Amer. Math. Soc.}, 11(2):283--320, 1998.

\bibitem[MM99]{minski}
Howard~A. Masur and Yair~N. Minsky.
\newblock Geometry of the complex of curves. {I}. {H}yperbolicity.
\newblock {\em Invent. Math.}, 138(1):103--149, 1999.

\bibitem[Ota01]{otal}
Jean-Pierre Otal.
\newblock {\em The hyperbolization theorem for fibered 3-manifolds}, volume~7 of {\em SMF/AMS Texts and Monographs}.
\newblock American Mathematical Society, Providence, RI; Soci\'et\'e{} Math\'ematique de France, Paris, 2001.
\newblock Translated from the 1996 French original by Leslie D. Kay.

\bibitem[Pfe22]{pfeil}
Mareike Pfeil.
\newblock Cataclysms for {A}nosov representations.
\newblock {\em Geom. Dedicata}, 216(6):Paper No. 61, 31, 2022.

\bibitem[PH92]{penner}
R.~C. Penner and J.~L. Harer.
\newblock {\em Combinatorics of train tracks}, volume 125 of {\em Annals of Mathematics Studies}.
\newblock Princeton University Press, Princeton, NJ, 1992.

\bibitem[Sam24]{sambarino}
Andr\'es Sambarino.
\newblock Infinitesimal {Z}ariski closures of positive representations.
\newblock {\em J. Differential Geom.}, 128(2):861--901, 2024.

\bibitem[Wol83]{wolpert}
Scott Wolpert.
\newblock On the symplectic geometry of deformations of a hyperbolic surface.
\newblock {\em Ann. of Math. (2)}, 117(2):207--234, 1983.

\end{thebibliography}

\end{document}